\documentclass[a4paper,11pt]{amsart}

\usepackage[T1]{fontenc}
\usepackage{mathrsfs, dsfont}
\usepackage{tikz}
\usetikzlibrary{calc}
\usetikzlibrary{intersections}
\usepackage{comment}
\usepackage{psfrag}

\definecolor{colore1}{RGB}{205,255,0}
\definecolor{colore2}{RGB}{255,255,0}
\definecolor{colore3}{RGB}{255,191,0}
\definecolor{coloreLinea1}{RGB}{0,0,255}
\definecolor{coloreLinea2}{RGB}{255,0,0}
\definecolor{coloreLinea3}{RGB}{148,0,105}

\tikzstyle{numeriLinea}=[black, pos=0.8, fill=white]

\usepackage{tikz}
\usepackage{comment}

\definecolor{colore1}{RGB}{205,255,0}
\definecolor{colore2}{RGB}{255,255,0}
\definecolor{colore3}{RGB}{255,191,0}
\definecolor{coloreLinea1}{RGB}{0,0,255}
\definecolor{coloreLinea2}{RGB}{255,0,0}
\definecolor{coloreLinea3}{RGB}{148,0,105}

\usepackage{amssymb,latexsym, amsthm, pifont}
\usepackage{hyperref}
\usepackage{cancel}
\usepackage[top=4cm,bottom=2.5cm,left=2.3cm,right=2.3cm,%
	bindingoffset=0cm]{geometry}
\usepackage{psfrag}
\usepackage{verbatim}

\usepackage{graphicx}
\usepackage{color}
\usepackage{nomencl} 

\theoremstyle{plain}
\newtheorem{lemma}{Lemma}[section]
\newtheorem{proposition}[lemma]{Proposition}
\newtheorem{theorem}[lemma]{Theorem}

\newtheorem*{claim}{Claim}

\theoremstyle{definition}

\theoremstyle{remark}
\newtheorem{remark}[lemma]{Remark}

\numberwithin{equation}{section}

\makenomenclature
\RequirePackage{ifthen}
\renewcommand{\nomgroup}[1]{%
 \ifthenelse{\equal{#1}{A}}{\item[\bfseries{General mathematical symbols}]}{%
 \ifthenelse{\equal{#1}{B}}{\item[]\item[\bfseries{Common Symbols for Systems of Conservation Laws}]}{{%
 \ifthenelse{\equal{#1}{C}}{\item[]\item[\bfseries{Symbols introduced in the present paper}]}{}}}}}
 \newcommand{\nomunit}[1]{\renewcommand{\nomentryend}{\hspace*{\fill}#1}}

\newcommand{\R}{\mathbb{R}}
\newcommand{\N}{\mathbb{N}}

\newcommand{\V}{\mathcal V}

\newcommand{\mcS}{\mathcal{S}}

\newcommand{\loc}{\text{\rm loc}}

\newcommand{\Ll}{\mathcal L}

\newcommand{\norm}[1]{\left\|#1\right\|}

\DeclareMathOperator{\TV}{\mathrm{TotVar}}

\DeclareMathOperator{\pt}{\partial_{\mathit{t}}}
\DeclareMathOperator{\px}{\partial_{\mathit{x}}}

\newcounter{stepnb}
\newcounter{substepnb}
\newcommand{\firststep}{\setcounter{stepnb}{0}}

\newcommand{\step}[1]{{{\sc \addtocounter{stepnb}{1}\noindent $\circleddash$ Step \arabic{stepnb}:} #1.}} 

\title[A counter-example to the regularity of systems]{Schaeffer's regularity theorem for scalar conservation laws does not extend to systems}
\author[L. Caravenna]{Laura Caravenna}
\address{L.C. Dipartimento di Matematica,
Universit\`a degli Studi di Padova,
Via Trieste 63, 35121 Padova, Italy}
\email{caravenna@math.unipd.it}

\author[L.~V.~Spinolo]{Laura V.~Spinolo}
\address{L.V.S. IMATI-CNR, via Ferrata 1, I-27100 Pavia, Italy.}
\email{spinolo@imati.cnr.it}

\begin{document}

\maketitle

{
\rightskip .85 cm
\leftskip .85 cm
\parindent 0 pt
\begin{footnotesize}

{\sc Abstract.}
Schaeffer's regularity theorem for scalar conservation laws can be loosely speaking formulated as follows. Assume that the flux is uniformly convex, then for a generic smooth initial datum the admissible solution is smooth outside a locally finite number of curves in
the $(t,x)$ plane. Here the term ``generic'' is to be interpreted in a suitable sense, related to the Baire Category Theorem. Whereas other regularity results valid for scalar conservation laws with convex fluxes have been extended to systems of conservation laws with genuinely nonlinear characteristic fields, in this work we exhibit an explicit counterexample which rules out the possibility of extending Schaeffer's Theorem. 
The analysis relies on careful interaction estimates and uses fine properties of the wave front-tracking approximation. 

\medskip\noindent
{\sc Keywords:} conservation laws, shock formation, regularity, Schaeffer Theorem. 

\medskip\noindent
{\sc MSC (2010):} 35L65 

\end{footnotesize}

}

\tableofcontents

\section{Introduction}
We are concerned with mild regularity properties for systems of conservation laws in one space dimension, namely equations in the form
\begin{equation}
\label{e:cl}
 \partial_t U + \partial_x \big[ G(U) \big] = 0.
\end{equation} 
In the previous expression, the unknown $U$ attains values in $\R^N$ and depends on ${(t, x) \in [0, + \infty[ \times \R}$. The flux function $G: \R^N \to \R^N$
 is of class $C^2$. If $N=1$, we call~\eqref{e:cl} \emph{scalar conservation law}. In 1973, Schaeffer~\cite{Schaeffer} established a regularity result (see Theorem~\ref{T:Schaeffer} below) that applies to scalar conservation laws. This paper aims at showing that this result does not extend to the case of systems, i.e.~to the case when $N>1$.

When $N>1$, system~\eqref{e:cl} is called \emph{strictly hyperbolic} if the Jacobian matrix $DG(U)$ admits $N$ real and distinct eigenvalues 
\begin{equation}
\label{e:sh}
 \lambda_1(U) < \dots < \lambda_N(U). 
\end{equation}
We term $\vec r_1 (U), \dots, \vec r_N (U)$ the corresponding right eigenvectors of $DG(U)$ and we say that the $i$-th characteristic field is \emph{genuinely nonlinear} 
if 
\begin{equation}
\label{e:gnlgenerale}
 \nabla \lambda_i (U) \cdot \vec r_i (U) \ge c >0, \quad \text{for every $U \in \R^N$}
\end{equation}
and for some suitable constant $c>0$. 
In the previous expression, $\cdot$ denotes the standard scalar product in $\R^N$. If the left hand side of~\eqref{e:gnlgenerale} is identically zero, then the $i$-th characteristic field is termed \emph{linearly degenerate}. 

In the present paper we deal with the Cauchy problem posed by coupling~\eqref{e:cl} with the initial datum 
 \begin{equation}
\label{e:cau}
 U(0, \cdot) = U_0
\end{equation}
and we refer to the books by Dafermos~\cite{Dafermos} and Serre~\cite{Serre} for a comprehensive introduction to systems of conservation laws. In particular, it is well-known that, even if $U_0$ is smooth and~\eqref{e:cl} is a scalar conservation law, the classical solution of~\eqref{e:cl},~\eqref{e:cau} breaks down in finite time owing to the formation of discontinuities. The Cauchy problem~\eqref{e:cl},~\eqref{e:cau} can be interpreted in the sense of distributions, but in general distributional solutions fail to be unique. In the attempt at restoring uniqueness, various \emph{admissibility conditions} have been introduced: we refer again to~\cite{Dafermos,Serre} for an overview.

In the following we briefly go over some well-posedness and regularity results for systems of conservation laws. We firstly focus on the scalar case $N=1$. The celebrated work by Kru{\v{z}}kov~\cite{Kru} establishes global existence and uniqueness results in the class of so-called \emph{entropy admissible} solutions of the Cauchy problem~\eqref{e:cl},~\eqref{e:cau} under the assumption that $U_0 \in L^\infty$. Regularity properties of entropy admissible solutions have been investigated in several papers: here we only mention some of the main contributions and we refer to~\cite{Dafermos, Serre} for a more complete discussion. First, a famous result by Ole{\u\i}nik~\cite{Oleinik} establishes the following regularization effect: when the flux $G\in C^{2}$ is uniformly convex,  for every $t>0$ the solution $U(t, \cdot)$ has bounded total variation, namely $U(t, \cdot) \in BV (\R)$, even if $U_0$ is only in $L^\infty$. More recently, Ambrosio and De Lellis~\cite{AmbrosioDeLellis} improved
Oleinik's result showing that, except at most countably many times, the
solution $U(t, \cdot)$ is actually a \emph{special function of bounded
variation}, namely $U(t, \cdot) \in SBV (\R)$; we refer to~\cite[\S~4]{AFP}
for the definition of $SBV (\R)$. This is a regularizing effect of the
nonlinearity. A result due to Schaeffer~\cite{Schaeffer}, moreover, states
that for a \emph{generic} smooth initial datum the admissible solution of the
Cauchy problem is even better than this: it develops at most a  locally finite
number of discontinuity curves, see Theorem~\ref{T:Schaeffer} below. In the following statement, we denote by $\mcS(\R)$ the Schwartz space of rapidly decreasing functions, endowed with the standard topology (see~\cite[p.133]{ReedSimon} for the precise definition). 
\begin{theorem}[Schaeffer]
\label{T:Schaeffer}
Assume that $N=1$ and that the flux $G$ is smooth and uniformly convex, namely $G'' (U)\ge c >0$ for some constant $c>0$ and for every $U \in \R$.

Then there is a set $\mathfrak F \subseteq \mcS(\R)$ that enjoys the following properties:
\begin{enumerate}
\item $\mathfrak F$ is of the first category in $\mcS(\R)$, namely
\begin{equation}
\label{e:firstc}
 \mathfrak F = \bigcup_{k=1}^\infty C_k, \quad \text{$C_k$ is closed and has empty interior, for every $k$}. 
\end{equation}
\item For every $U_0 \in \mcS(\R) \setminus \mathfrak F$, the entropy admissible solution of the Cauchy problem~\eqref{e:cl},~\eqref{e:cau}
enjoys the following regularity. For every open bounded set $\Omega \subseteq [0, + \infty[ \times \R $ there is a finite number of Lipschitz continuous curves $\Gamma_1, \dots, \Gamma_m \subseteq \R^2$ such that 
$$
 U \in C^{\infty} \left( \Omega \setminus \cup_{i=1}^m \Gamma_i \right) 
$$
\end{enumerate}
\end{theorem}
The curves $\Gamma_1, \dots, \Gamma_m$ are usually termed \emph{shocks}. We briefly comment the above result. First, the assumption that $G$ is uniformly convex can be relaxed, see for instance  Dafermos~\cite{Daf}. See also~\cite{TWZ} for recent related results. Second, a characterization of the set $\mathfrak F$ can be found in a paper by Tadmor and Tassa~\cite{TadmorTassa}. Third, the result is sharp in the sense that one cannot hope that the regularity holds for \emph{every} smooth initial datum. More precisely, even in the case $G(U) = U^2 /2$ several authors constructed initial data in $\mcS (\R)$ that develop infinitely many shocks on compact sets; see for instance the counter-example exhibited by Schaeffer himself~\cite[\S~5]{Schaeffer}. Among recent works, we mention  the construction by Adimurthi, Ghoshal and Veerappa Gowda~\cite{Adimurthi}.

The present paper aims at discussing whether or not Schaeffer's Theorem~\ref{T:Schaeffer} extends to systems, i.e.~to the case when~$N>1$. 
Investigating whether or not the number of shocks is (generically) finite is motivated not only by intrinsic interest, but also by applications. In particular, knowing that the limit solution admits at most finitely many shocks simplifies the study of several approximation schemes. As an example, we recall that the proof of the convergence of the vanishing viscosity approximation in the case when the limit solution has finitely many, non interacting shocks was provided by Goodman and Xin~\cite{GoodmanXin} and it is considerably simpler than the proof in the general case, which is due to Bianchini and Bressan~\cite{BiaBre}. 

We now recall some well-posedness and regularity results for systems of conservation laws. The pioneering work by Glimm~\cite{Glimm} established existence of a global in time, distributional solutions of the Cauchy problem~\eqref{e:cl},~\eqref{e:cau} under the assumptions that the system is strictly hyperbolic, that each characteristic field is either genuinely nonlinear or linearly degenerate and that the total variation of the initial datum $U_0$ is sufficiently small. Uniqueness results were obtained in a series of papers by Bressan and several collaborators: we refer to the book~\cite{Bre} for an overview. In the following, we call the solution constructed by Glimm \emph{admissible solution} of the Cauchy problem~\eqref{e:cl},~\eqref{e:cau}. Note that this solution can be also recovered as the limit of a wave front-tracking approximation~\cite{Bre} and of a second order approximation~\cite{BiaBre}. 

Several regularity results that apply to scalar conservation laws with convex fluxes have been extended to systems of conservation laws where every vector field is genuinely nonlinear (i.e.~condition~\eqref{e:gnlgenerale} holds for every $i=1, \dots, N$). See, for instance, the works by Glimm and Lax~\cite{GlimmLax}, Liu~\cite{Liu_decay} and Bressan and Colombo~\cite{BressanColombo} for possible extensions of the decay estimate by Ole{\u\i}nik~\cite{Oleinik}. Moreover, the $SBV$ regularity result by Ambrosio and De Lellis~\cite{AmbrosioDeLellis} has been extended to the case of systems, see Dafermos~\cite{Dafermos_sbv} for self-similar solutions, Ancona and Nguyen~\cite{AnconaNguyen} for Temple systems and Bianchini and Caravenna~\cite{BiaCar} for general systems where every characteristic field is genuinely nonlinear. 

The main result of the present paper states that, contrary to 
the results by Ole{\u\i}nik~\cite{Oleinik} and Ambrosio and De Lellis~\cite{AmbrosioDeLellis}, 
 Schaeffer's Theorem~\ref{T:Schaeffer} does not extend to systems. 
\begin{theorem}
\label{T:main} There are a flux function $G: \R^3 \to \R^3$, a compact set $K \subseteq [0, + \infty[ \times \R$ and a set $ \mathfrak B \subseteq \mathcal S (\R)$ that enjoy the following properties:
\begin{enumerate}
\item \label{item:1main}system~\eqref{e:cl} is strictly hyperbolic and every characteristic field is genuinely nonlinear, namely~\eqref{e:sh} holds and moreover property~\eqref{e:gnlgenerale} is satisfied for every $i=1, 2, 3$. 
\item \label{item:2main} The set $ \mathfrak B$ is non empty and open in $\mcS(\R)$. 
\item \label{item:3main} For every $U_0 \in \mathfrak B$ the admissible solution of the Cauchy problem~\eqref{e:cl},~\eqref{e:cau} has infinitely many shocks  in the compact set $K$
\end{enumerate}
\end{theorem}
Some remarks are in order:
\begin{itemize}
\item in the statement of the above theorem by \emph{shock} we mean a  Lipschitz continuous curve $x = \Gamma (t)$ at which $U$ is discontinuous. 
\item The Baire Theorem implies that any set of the first category~\eqref{e:firstc} has empty interior. Since the set of ``bad data'' $\mathfrak B$ is open and non empty, it cannot be of the first category and hence Theorem~\ref{T:main} provides a counter-example to the possibility of extending Schaeffer's Theorem to the case of systems. 
\item By looking at the explicit construction one can infer that $\mathfrak B$ satisfies the following further requirement. For every $U_0 \in \mathfrak B$, the total variation of $U_0$ is sufficiently small to apply the existence and uniqueness results in~\cite{Bre,Glimm}. This means that the counter-example provided by Theorem~\ref{T:main} belongs to the same class where we have well-posedness. \item Our construction is explicit, in the sense that we provide an explicit formula for the function $G$, the compact $K$ and the set $\mathfrak B$, see~\eqref{e:pertSyst}, \eqref{E:secondapalla} and the construction in~\S~\ref{S:step1}. 
\item Our construction shows, as a byproduct, that a wave-pattern containing infinitely many shocks can be robust with respect to suitable 
perturbations of the initial data. 
\item Our counter-example requires 3 dimensions, namely $N=3$. It is known that $2 \times 2$ systems are usually much better behaved than higher dimension systems, see for instance the discussion in~\cite[\S~XI\!I]{Dafermos}. An interesting question that is still to be addressed is whether or not Scheffer's Theorem extends to (suitable classes of) $2 \times 2$ systems~\footnote{We thank Alberto Bressan for this remark.}. 
\end{itemize}
To conclude, we briefly outline the proof of Theorem~\ref{T:main}.
The set $\mathfrak B$ will be basically obtained by considering small $W^{1, \infty}$ perturbations of a certain function $\widetilde U$. The main point in the proof is then constructing $G$ and $\widetilde U$ in such a way that
\begin{enumerate}
\item when $U_0 =\widetilde U$  the admissible solution of the Cauchy problem~\eqref{e:cl}-\eqref{e:cau} develops infinitely many shocks, and
\item the same happens when $U_0$ is a small perturbation of $\widetilde U$. 
\end{enumerate}
We choose as flux function $G$ a particular representative of a family of fluxes introduced by Baiti and Jenssen~\cite{BaJ}. Note that in~\cite{BaJ} the authors exhibit a wave-pattern containing infinitely many shocks. Actually, the original wave-pattern in~\cite{BaJ} contains large amplitude waves, but the construction can be adapted to obtain a wave-pattern with small total variation. Although we use several results established in~\cite{BaJ}, our analysis is quite different from the one in~\cite{BaJ}. 
More precisely, there are three main challenges in adapting the construction in~\cite{BaJ} for our goals :
\begin{itemize}
\item we need to show that the wave-pattern in~\cite{BaJ} can be exhibited by a solution with smooth initial datum: this issue is tackled by relying on the notion of \emph{compression wave}, see \S~\ref{sss:cw}. 
\item A much more severe obstruction is the fact that the wave-pattern in~\cite{BaJ} is a priori \emph{not robust} with respect to perturbations. We refer to the discussion at the beginning of \S~\ref{ss:tildeu} for a more detailed explanation. Here we just point out that, owing to this lack of robustness, we have to introduce a more complicated construction than the original one in~\cite{BaJ}. Even in the case when the initial datum is exactly $\widetilde U$, the structure of the admissible solution is much more complex than the one considered in~\cite{BaJ}. 
\item The analysis in~\cite{BaJ} relies on the construction of \emph{explicit} solutions. In our case, computing explicit solutions is prohibitive and hence we argue by introducing a wave front-tracking approximation. We perform careful interaction estimates to gain precise information on the structure of the approximate solution and we eventually pass to the limit by using 
fine properties of the wave front-tracking approximation established by Bressan and LeFloch~\cite{BressanLeFloch}. 
\end{itemize}
The paper is organized as follows. 
In \S~\ref{s:overview} for the reader's convenience we go over some previous results. More precisely, in~\S~\ref{ss:wft} we recall some of the main properties of the wave front-tracking approximation, while in~\S~\ref{SS:bjcon} we introduce the Baiti-Jenssen system and recall some of the main properties. In \S~\ref{SS:bjin} we establish preliminary estimates on admissible solutions of the Baiti-Jenssen system. In~\S~\ref{S:step1} we construct the function $\widetilde U$. In~\S~\ref{s:pprop} we establish the proof of Theorem~\ref{T:main}. In particular, we  show that the solution of the Cauchy problem with initial datum $\widetilde U$ develops infinitely many shocks and that this behavior is robust with respect to perturbations of $\widetilde U$.  

For the reader's convenience, we collect the notation of this paper at Page~\pageref{notations}.
 
\nomenclature[A]{$W^{1,\infty}$:}{the space of Lipschitz continuous functions\nomunit{}}
\nomenclature[A]{$\mcS(\R)$:}{the Schwartz space of rapidly decreasing functions, endowed with the standard topology (see for instance~\cite[p.133]{ReedSimon} for the precise definition)\nomunit{}}
\nomenclature[A]{$\vec z_1 \cdot \vec z_2:$}{the Euclidian scalar product between the vectors $\vec z_1, \; \vec z_2 \in \R^N$\nomunit{}}
\nomenclature[A]{$\norm{\cdot}_{W^{1 \infty}}$:}{the standard norm in the Sobolev space $W^{1 \infty}$\nomunit{}}
\nomenclature[A]{$\TV U$:}{the total variation of the function $U: \R \to \R^N$, see~\cite[\S~3.2]{AFP} for the precise definition\nomunit{}}
\nomenclature[A]{$\Ll^N$:}{the Lebesgue measure on $\R^N$\nomunit{}}
\nomenclature[A]{a.e. $x$:}{for $\Ll^1$-almost every $x$\nomunit{}}
\nomenclature[A]{a.e. $(t, x)$:}{for $\Ll^2$-almost every $(t,x)$\nomunit{}}
\nomenclature[A]{$F'$:}{the first derivative of the differentiable  function $F: \R \to \R^N$\nomunit{}}
\nomenclature[A]{$F(x^\pm)$:}{the left and right limit of the function $F$ at $x$ (whenever they exist)\nomunit{}}
\nomenclature[C]{$u,w,v$:}{the first, second and third component of the vector-valued function $U$\nomunit{See \S~\ref{sss:baj1}}}
\nomenclature[C]{$\eta$:}{the perturbation parameter in the flux function $F_\eta$ \nomunit{See~\eqref{e:pertSyst},~\eqref{e:c:parametri2}}}
\nomenclature[C]{$r$:}{a strictly positive parameter \nomunit{See~\eqref{e:c:parametri3bis},~\eqref{e:parameters},~\eqref{e:palla}}}
\nomenclature[C]{$\lambda_i(U)$:}{the $i$-th eigenvalue of the Jacobian matrix $JF_{\eta}$ \nomunit{See~\eqref{e:eigenvaluese}}}
\nomenclature[C]{$\widetilde U$, $\widetilde U_{\varsigma}$:}{\color{purple}the function $\widetilde U: \R \to \R^3$ and its mollification \nomunit{See~\S~\ref{sss:tildeu},~\eqref{e:mollificazione},~\eqref{E:secondapalla}}}
\nomenclature[C]{$\vec r_i(U)$:}{the $i$-th right eigenvector of the Jacobian matrix $JF_{\eta}$\nomunit{See~\S~\ref{sss:baj2}}}
\nomenclature[C]{$\nu$, $h_{\nu}$:}{parameter and mesh size for the wave front-tracking approximation \nomunit{See~\S~\ref{sss:mesh}}}
\nomenclature[C]{$x^\nu_{i}$:}{mesh points for the wave front-tracking approximation \nomunit{See~\eqref{e:meshsize}}}
\nomenclature[C]{$U^\nu$:}{wave front-tracking approximation of the admissible solution $U$ \nomunit{See~\S~\ref{ss:wft}}}
\nomenclature[C]{$U^\nu_0$:}{wave front-tracking approximation of the initial datum $U_{0}$ \nomunit{See~\S~\ref{ss:wft}}}
\nomenclature[C]{$\mu_\nu$:}{the threshold for using the accurate Riemann solver \nomunit{See~\S~\ref{ss:wft}}}
\nomenclature[C]{$V$:}{the Lipschitz continuous function $V: \R \to \R^3$ \nomunit{See~\eqref{e:VI}}}
\nomenclature[C]{$W$:}{the piecewise constant function $W: \R \to \R^3$ \nomunit{See~\eqref{e:W}}}
\nomenclature[C]{$\Psi$:}{the function $\Psi: \R \to \R^3$ \nomunit{See~\S\S~\ref{sss:psi},~\ref{sss:tildeu}}}
\nomenclature[C]{$\underline U'$, $\underline U''$, $\underline U^\ast$ and $\underline U^{\ast \ast}$:}{fixed states in $\R^3$ \nomunit{See~\eqref{e:statedef}}}
\nomenclature[C]{$\widetilde T$:}{the strictly positive interaction time \nomunit{See~\eqref{e:T},~\eqref{e:c:parametriT}}}
\nomenclature[C]{$q$:}{a strictly positive parameter that we fix equal to $20$ \nomunit{See Lemma~\ref{l:infty},~Remark~\ref{rem:qrho}}}
\nomenclature[C]{$\mathfrak q$, $\mathfrak p$:}{strictly positive parameters}  \nomunit{See~\eqref{e:qs}}
\nomenclature[C]{$\omega$:}{a strictly positive parameter \nomunit{See~\eqref{e:uduetre},~\eqref{e:c:parametri2}}}
\nomenclature[C]{$\zeta_w$:}{a strictly positive parameter \nomunit{See~\eqref{e:psi},~\eqref{e:c:parametribis},~\eqref{e:c:parametri5}}}
\nomenclature[C]{$\zeta_c$:}{a strictly positive parameter \nomunit{See~\eqref{e:psi},~\eqref{e:c:parametri3bis},~\eqref{e:c:parametri5}}}
\nomenclature[C]{$\varepsilon$:}{a strictly positive, sufficiently small parameter\nomunit{See~Proposition~\eqref{p:perturbation}}
\nomenclature[C]{$\delta$:}{a strictly positive parameter}  \nomunit{See~\eqref{e:cosaedelta},~\eqref{e:c:parametribis}}
\nomenclature[C]{$\rho$:}{the strictly positive parameter in~\eqref{e:c:parametrirho} \nomunit{See~\S~\ref{sss:tildeu},~Remark~\ref{rem:qrho}}}
\nomenclature[C]{$U_I$, $U_{II}$ and $U_{I\!I\!I}$:}{fixed states in $\R^3$ \nomunit{See~\eqref{e:uduetre}}}
\nomenclature[C]{$\mathfrak R_\ell, \dots, \mathfrak R_r$:}{open subsets of $\R$ \nomunit{See~\eqref{e:regions}}}
\nomenclature[A]{$\mathcal O(1)$:}{any function satisfying $ 0 < c \leq \mathcal O(1) \leq C$ for suitable constants $c,C>0$. The precise value of $C$ and $c$ can vary from line to line\nomunit{}}
\nomenclature[A]{$D_{i}[\sigma,\bar U]$:}{the $i$-wave fan curve through $\bar U$ \nomunit{See~\eqref{e:wavefan}}}
\nomenclature[A]{$R_i [s, \bar U]$:}{the $i$-rarefaction curve through $\bar U$ \nomunit{See~\eqref{e:intcur}}}
\nomenclature[A]{$S_i[s, \bar U]$:}{the $i$-shock curve through $\bar U$ \nomunit{See~\S~\ref{ss:wft}}}
\nomenclature[C]{$\V_i$:}{the strength of a shock $i$ \nomunit{See Page~\pageref{label:strength}}}

\section{Overview of previous results}
\label{s:overview} For the reader's convenience, in this section we go over some previous results that we will need in the following. More precisely, we proceed as follows:
\begin{itemize}
\item[\S~\ref{ss:wft}:] we quickly summarize the wave front-tracking algorithm~\cite{Bre} and we fix some notation. 
\item[\S~\ref{SS:bjcon}:] we introduce the Baiti-Jenssen system and 
we discuss some of its properties. 
\end{itemize}
\subsection{The wave front-tracking approximation algorithm}
\label{ss:wft}
In this paragraph we briefly go over the version of the wave 
front-tracking algorithm discussed in~\cite{Bre} (see in particular Chapter 7 in there). We refer to~\cite{Bre} and to the books by Dafermos~\cite[\S~14.13]{Dafermos} and by Holden and Risebro~\cite{HoldenRisebro} for a more extended discussion and for a comprehensive list of references. Also, in the following discussion we assume that each characteristic field is genuinely nonlinear (i.e., that~\eqref{e:gnlgenerale} holds true) because this hypothesis is satisfied by our system. 

We first introduce some notation. We recall that the \emph{$i$-wave fan curve through $\bar U$} is
\begin{equation}
\label{e:wavefan}
 D_i [s, \bar U] : =
 \left\{
 \begin{array}{ll}
 R_i [s, \bar U] & s \ge 0\\
 S_i [s, \bar U] & s <0. \\
 \end{array}
 \right.
\end{equation}
In the previous expression, $R_i$ is the integral curve of $\vec r_i$ passing through $\bar U$, namely the solution of the Cauchy problem 
\begin{equation}
\label{e:intcur}
\left\{
\begin{array}{lll}
 \displaystyle{\frac{d R_i}{ds} = \vec r_i (U)}, \\
 \phantom{ciao} \\
 R_i [0, \bar U] = \bar U. \\
\end{array}
\right.
\end{equation} 
Also, we denote by $S_i$ the $i$-Hugoniot locus, i.e.~the set of states that can be joined to $\bar U$ by a shock of the $i$-family, namely by a $i$-shock. The \emph{speed of the shock} can be computed by using the Rankine-Hugoniot conditions. 
We call the absolute value $|s|$ \emph{strength} of the shock between $\bar U$ and $S_i [s, \bar U]$. \label{label:strength}

We are now ready to outline the the construction of the wave front-tracking approximation. We fix a small parameter $\nu >0$ and we denote by $U^\nu$ the wave front-tracking approximation. 
The final goal is to show that when $\nu \to 0^+$ the family $U^\nu$ converges to the admissible solution of the Cauchy problem~\eqref{e:cl}-\eqref{e:cau}. The main steps to construct $U^\nu$ are the following (we refer to~\cite[\S~7]{Bre} for a detailed discussion):
\begin{enumerate}
\item we construct $U^{\nu}_0$, a piecewise constant approximation of the initial datum $U_0$. 
\item At each discontinuity of $U^{\nu}_0 $ we solve the Riemann problem between the left and the right state by relying on the Lax Theorem~\cite{Lax}. We want to define $U^\nu$ in such a way that $U^\nu (t, \cdot)$ is piecewise constant for almost every $t>0$. Hence, we replace the rarefaction waves in the Lax solution of the Riemann problem with a suitably defined piecewise constant approximation. The resulting approximate solution is called \emph{accurate Riemann solver}. 
\item We repeat the above procedure at each discontinuity point of $U^{\nu}_0$ and we define $U^\nu$ by juxtaposing the approximate solution of each Riemann problem. In this way, $U^\nu$ is piecewise constant and has a finite number of discontinuity lines. By a slight 
abuse of notation, we call \emph{rarefaction waves} the discontinuity lines corresponding to rarefactions. We can also introduce a notion of \emph{strength} for the rarefaction wave (see~\cite[Chapter 7]{Bre} for the technical details). 
\item Let us consider the point at which two \emph{waves} (i.e., discontinuity lines) \emph{interact} (i.e.~cross each other). The interaction determines a new Riemann problem, which is solved by using the same procedure as in step 2. above. In this way we can extend the wave front-tracking approximation $U^\nu$ after the first interaction occurs. 
\item In principle, it may happen that the number of discontinuity lines of $U^\nu$ blows up in finite time: this would prevent us from defining $U^\nu$ globally in time. The number of discontinuities can blow up if for instance $U^\nu$ contains a wave pattern like the one illustrated in Figure~\ref{F:infty}. 
\item To prevent the number of discontinuities from blowing up, we introduce the so-called \emph{non physical waves}. The exact definition is quite technical and it is given in~\cite[\S7.2]{Bre}, but the basic idea is the following. We introduce a threshold $\mu_\nu$ and we consider an interaction point. If the product between the strengths of the incoming waves is bigger than $\mu_\nu$, then we use the \emph{accurate Riemann solver} defined at step 2. If it is smaller, we use a so-called \emph{simplified Riemann solver}. The \emph{simplified Riemann solver} involves a minimum number of outgoing waves. Basically, all the new waves are packed together in a single \emph{non physical wave}, which travels at a faster speed than any other wave. 
\item The analysis in~\cite[\S7]{Bre} shows that, by relying on a suitable choice of the approximate and of the simplified Riemann solver, of the approximate initial datum $U^{\nu}_0$ and of the threshold $\mu_\nu$, one can prove that the approximate wave front-tracking solutions $U^\nu$ converge as $\nu \to 0^+$ to the unique admissible solution of the Cauchy problem~\eqref{e:cl},~\eqref{e:cau}. 
\end{enumerate}
\subsection{The Baiti-Jenssen system}
\label{SS:bjcon}
In this paragraph we recall some results from~\cite{BaJ}. More precisely, we proceed as follows. 
\begin{itemize}
\item[\S~\ref{sss:baj1}:] we introduce the explicit expression of the Baiti-Jenssen system and we comment on it. 
\item[\S~\ref{sss:baj2}:] we recall the explicit expression of the eigenvalues and we go over the structure of the wave fan curves. 
\end{itemize}
\subsubsection{The system}
\label{sss:baj1}
We introduce the Baiti-Jenssen system. We fix $\eta \in \, ]0, 1[$ and we define the function $F_\eta: \R^3 \to \R^3$ by setting 
\begin{equation}
\label{e:pertSyst}
F_{\eta}(U) : = 
\left(
\begin{matrix}
	\displaystyle{ 4 \big[ (v-1) u - w \big] + \eta p_1(U) \phantom{\int}} \\
	v^2 \\
	\displaystyle{4 \Big\{ {v (v-2)u} - (v-1) w \Big\}
	 + \eta p_3(U) } \\
	\end{matrix}\right) 
\end{equation}	
In the above expression, $u$, $v$ and $w$ denote the components of $U$, namely $U=(u, v, w)$. 
The functions $p_{1}$ and $p_{3}$ are given by 
\begin{align}
\label{e:pq}
 p_1(U) &= \frac{1}{2} \Big\{ [w -(v-2) u ]^2 -[ w-vu]^2\big] \Big\}=2 u w- 2 u^{2} (v-1 ), \\
 p_3(U) &= \frac{1}{2} 
 \Big\{v [w -(v-2) u ]^2- (v-2) [ w-vu]^2 \Big\}= w^2-u^2 (v-2) v .
\end{align}
In the following we are concerned with the system of conservation laws 
\begin{equation}
\label{e:cl2}
 \pt U + \px \big[ F_\eta(U) \big] = 0,
\end{equation} 
which we term \emph{Baiti-Jenssen system}. Two remarks are here in order. First,~\eqref{e:pertSyst} is exactly system (3.11) in~\cite{BaJ} provided that we choose 
$\varepsilon = \eta$, $g(v) = v^2$, $a(v) = v$, $b(v) = v-2$, $c=4$. The reason why 
we only consider a particular representative of the class of systems considered in~\cite{BaJ} is because we want to simplify the analysis and the exposition. Indeed, some parts of the proof of Theorem~\ref{T:main} are already fairly technical and we have decided to keep the rest as simple as possible. 
However, we are confident that our argument can be extended to much more general classes of systems. 

Second, the celebrated existence and uniqueness results~\cite{Glimm,Bre} mentioned in the introduction imply that there are constants $C>0$ and $\delta>0$ such that, if $U_0$ is a compactly supported function satisfying
$$
\TV\, U_0 \leq \delta, 
$$ 
then the Cauchy problem obtained by coupling~\eqref{e:cl2} with the condition $U(0, \cdot)=U_0$ has a unique, global in time admissible solution which satisfies
$$
\TV\, U(t, \cdot) \leq C \, \TV\, U_0, \quad \text{for every $t>0$}. 
$$ 
In principle, both $\delta$ and $C$ depend on $\eta$. However, by looking at the proof of the convergence of the wave front-tracking approximation one realizes that 
$C$ and $\delta$ only depend on bounds on $F_\eta$ and its derivatives of various orders. Since all these functions are uniformly bounded in $\eta$, then we can choose $C$ and $\delta$ in such a way that they \emph{do not} depend on $\eta$. In the following, we will let $\eta$ vary but we will always assume that the function $U$ attains values in the unit ball, namely $|U| < 1$. This will be \emph{a posteriori} justified because we will choose a compactly supported initial datum with sufficiently small total variation. 
\subsubsection{Eigenvalues and wave fan curves}
\label{sss:baj2}
 We now recall some features of system~\eqref{e:pertSyst} and we refer to~\cite[pp. 841-843]{BaJ} for the proof. 
First, the eigenvalues of the Jacobian matrix $DF_{\eta} (U)$ are 
\begin{align}
\label{e:eigenvaluese}
&\lambda_{1}(U)=2 \eta \big[w-(v-2)u \big]-4 &&<&&\lambda_{2}(U)=2v &&<&& \lambda_{3}(U)=2 \eta \big[ w-v u \big]+4.
\end{align}
Note that 
\begin{equation}
\label{e:boundautovalori}
 -6 < \lambda_1 (U)< - \frac{5}{2} < -2 
 < \lambda_2 (U) < 2 < 3 < \lambda_3 (U) 
 < 5 
 \quad \text{if $|U| < 1$ and $0 \leq \eta < \frac{1}{4}$} 
\end{equation}
and hence in particular 
\begin{equation}
\label{e:boundautovalori2}
 | \lambda_1 (U)|, \; | \lambda_2 (U)|, \; 
 | \lambda_3 (U) | 
 < 6 
 \quad \text{if $|U| < 1$ and $0 < \eta < \frac{1}{4}$}. 
\end{equation}
Note that~\eqref{e:boundautovalori} implies that the system is strictly hyperbolic if $|U| < 1$ and $0 \leq \eta <1/4$. Note furthermore that $2$ is a Lipschitz constant of each eigenvalue if $|U| < 1$ and $0 \leq \eta <1/4$. The first and the third right eigenvectors are 
\begin{equation}
\label{e:eigenvectors12}
\vec r_{1}(U)=\left(\begin{matrix}1 \\ 0 \\ v\end{matrix}\right)
\quad \text{and} \quad
\vec r_{3}(U)=\left(\begin{matrix}1 \\ 0 \\ v-2 \end{matrix}\right),
\end{equation}
respectively. The explicit expression of the second eigenvector is not relevant here. 
Note however that the assumption of genuine nonlinearity is satisfied since 
\begin{subequations}
\label{e:gnl}
\begin{equation}
\label{e:gnl1}
\nabla\lambda_{1}(U)\cdot \vec r_{1}(U)=4 \eta >0, 
\qquad 
\nabla\lambda_{2}(U)\cdot \vec r_{2}(U) = 2 >0
\end{equation}
and
\begin{equation}
\label{e:gnl2}
\nabla\lambda_{3}(U)\cdot \vec r_{3}(U)= -4 \eta <0.
\end{equation}
\end{subequations}
Note that~\eqref{e:gnl2} implies~\eqref{e:gnlgenerale} provided that we change the orientation of $\vec r_3$. 
Owing to~\eqref{e:eigenvectors12}, the 1- and the 3-wave fan curve through $\bar U= (\bar u, \bar v, \bar w)$
are straight lines in the planes $v = \bar v$. More precisely, 
\begin{subequations}
\label{e:wavecurves13}
\begin{align}
D_{1}[\sigma; \bar U]=
\left(\begin{matrix}\sigma+\bar u \\ \bar v \\ \bar v\sigma+\bar w\end{matrix}\right)=\bar U + \sigma\vec r_{1}(\bar U)=\bar U + \sigma\vec r_{1}(\bar v),
\\
D_{3}[\tau; \bar U]=\left(\begin{matrix}\tau+\bar u \\ \bar v \\ (\bar v-2)\tau+\bar w\end{matrix}\right)=\bar U+\tau\vec r_{3}(\bar U)=\bar U+\tau\vec r_{3}(\bar v).
\end{align}
\end{subequations}
Owing to~\eqref{e:gnl}, we have that 
\begin{itemize}
\label{ite:analysis13waves}
\item[$\bullet$] if $\sigma<0$, then $\bar U$ and $D_{1}[\sigma;(\bar u,\bar v,\bar w)]$ are connected by a $1$-\emph{shock}. If $\sigma>0$, then $\bar U$ and $D_{1}[\sigma;(\bar u,\bar v,\bar w)]$ are connected by a $1$-\emph{rarefaction wave}. 
\item[$\bullet$] if $\tau<0$, then $\bar U$ and $D_{3}[\sigma;(\bar u,\bar v,\bar w)]$ are connected by a $3$-\emph{rarefaction wave}. If $\tau>0$, then $\bar U$ and $D_{3}[\sigma;(\bar u,\bar v,\bar w)]$ are connected by a $3$-\emph{shock}. 
\end{itemize}
To understand the structure of the second wave fan curve through $\bar U$ we use the following simple observation, which for future reference we state as a lemma. 
\begin{lemma}
\label{L:v}
Assume that $U=(u, v, w)$ is an admissible solution of the system of conservation laws~\eqref{e:cl2}. Then the second component $v$ is an entropy admissible solution 
of the scalar conservation law 
\begin{equation}
\label{e:v}
 \pt v + \px \big[v^2 \big] =0.
\end{equation}
\end{lemma}
\begin{proof}
Lemma~\ref{L:v} was used in~\cite{BaJ}, but we provide the proof for the sake of completeness. Owing to the analysis in~\cite{BiaBre} (see in particular Theorem 1 and \S~15 in there), the admissible solution $U$ can be recovered as the unique limit $\varepsilon \to 0^+$ of the second order approximation
$$
 \pt U_\varepsilon + \px \big[ F_\eta(U_\varepsilon ) \big] = \varepsilon \, \partial_{xx}^2 U_\varepsilon .
$$ 
We then conclude by considering the second component and recalling that the entropy admissible solution of a scalar conservation law is the unique limit of the vanishing viscosity approximation (see~\cite[\S~6.3]{Dafermos}). 
\end{proof}
By combining~\eqref{e:gnl} with Lemma~\ref{L:v} and by recalling that 
the flux in~\eqref{e:v} is convex we conclude tha we can choose the parametrization of $D_2$ in such a way that
\begin{itemize}
\item[$\bullet$] if $s<0$, then $\bar U$ and $D_2 [s, \bar U]$ are connected by a $2$-shock and the second component of $D_2 [s, \bar U]$ is $\bar v + s< \bar v$
\item[$\bullet$] if $s>0$, then $\bar U$ and $D_2 [s, \bar U]$ are connected by a $2$-rarefaction wave and the second component of $D_2 [s, \bar U]$ is $\bar v + s> \bar v$.
\end{itemize}
\section{Preliminary results concerning the Baiti-Jenssen system}
\label{SS:bjin}
This section concerns the Baiti-Jenssen system~\eqref{e:cl2}. It is divided into two parts:
\begin{itemize}
\item In \S~\ref{sss:13}, \S~\ref{sss:1133}, \S~\ref{sss:12} and \S~\ref{ss:22} we discuss interaction estimates for the Baiti-Jenssen system. More precisely, in \S~\ref{sss:13}, \S~\ref{sss:1133} we recall some analysis from~\cite{BaJ}. In \S~\ref{sss:12} we 
state a new version of a result established in~\cite{BaJ}. The proof is provided in the companion paper~\cite{CaravennaSpinolo}. In \S~\ref{ss:22} we go over a new interaction estimate established in~\cite{CaravennaSpinolo}. 
\item In \S~\ref{ss:wp} we discuss new results concerning the solution of the Riemann problem in the case when the left and the right states satisfy suitable 
structural assumptions. 
\end{itemize}
Both parts will be used in \S~\ref{s:pprop} in the analysis of the wave-front tracking approximation of a general class of Cauchy problems.
\subsection{Analysis of 1-3 interactions}
\label{sss:13}
In this paragraph we consider the interaction between a shock of the first family, i.e.~a 1-shock, and a 3-shock. 
More precisely, we term $U_\ell$, $U_m$ and $U_r$ the left, middle and right state before the interaction, respectively (see Figure~\ref{F:2311}, left part). In other words, 
\begin{equation}
\label{e:13prima}
 U_m = D_3[\tau, U_\ell], \qquad U_r = D_1 [\sigma, U_m]
\end{equation}
for some $\tau>0$, $\sigma <0$, where $D_{1}[\cdot]$ and $D_{3}[\cdot]$ are given in~\eqref{e:wavecurves13}. 

We now want to solve the Riemann problem between $U_\ell$ (on the left) and $U_r$ (on the right). 
We recall that the 1- and the 3-wave fan curves are just straight lines in planes where the $v$ component is constant, see~\eqref{e:wavecurves13}. The slope of the lines only depends on $v$. This implies that the 1- and the 3-wave fan curves commute
and the solution of the Riemann problem between $U_\ell$ (on the left) and $U_r$ (on the right) contains no 2-wave. In other words, the following holds. 
We denote by $U'_m$ the middle state after the interaction
(see again Figure~\ref{F:2311}, left part). From~\eqref{e:13prima} we get 
\begin{equation}
\label{e:13dopo}
 U'_m = D_1[\sigma, U_\ell], \qquad U_r = D_3[\tau, U'_m]. 
\end{equation}
\subsection{Analysis of 1-1 and 3-3 interactions} 
\label{sss:1133}
Owing to the particular structure~\eqref{e:wavecurves13} of the 1- and 3-wave fan curves, the incoming shocks in 1-1 interactions and 3-3 interactions simply merge. In particular, no new wave is produced. More precisely, we have the following: we focus on 3-3 interactions and we refer to Figure~\ref{F:2311} for a representation. We term $U_\ell$, $U_m$ and $U_r$ the left, middle and right state before the interaction, respectively. In other words, 
$$
 U_m = D_3[\tau_\ell, U_\ell], \qquad U_r = D_3 [\tau_r, U_m]
$$ 
for some $\tau_\ell$, $\tau_r >0$. Owing to~\eqref{e:wavecurves13}, we have $U_r = D_1[\tau_\ell+ \tau_r, U_{\ell}]
$ and hence the only outgoing wave is a 3-wave. The analysis of 1-1 interactions is completely analogous. 
\begin{figure}
\begin{center}
\begin{picture}(0,0)%
\includegraphics{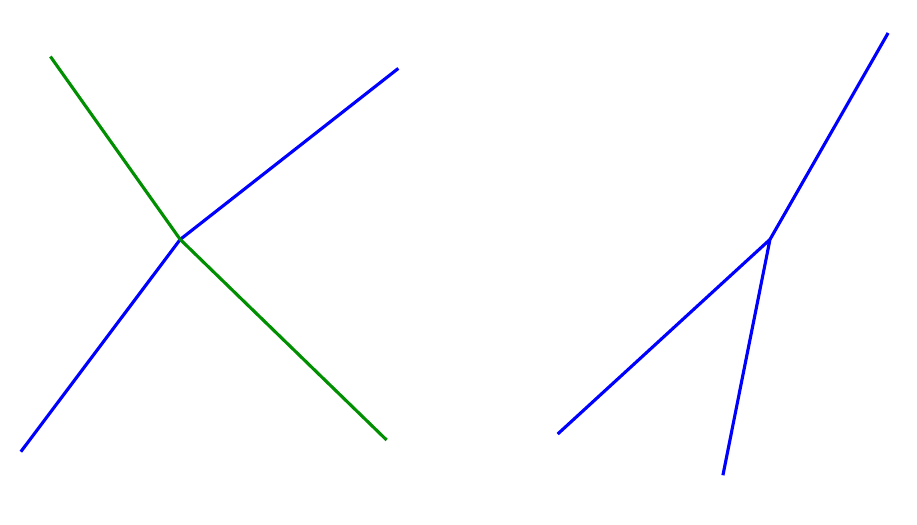}%
\end{picture}%
\setlength{\unitlength}{3947sp}%
\begingroup\makeatletter\ifx\SetFigFont\undefined%
\gdef\SetFigFont#1#2#3#4#5{%
  \reset@font\fontsize{#1}{#2pt}%
  \fontfamily{#3}\fontseries{#4}\fontshape{#5}%
  \selectfont}%
\fi\endgroup%
\begin{picture}(4391,2500)(436,-7565)
\put(2405,-5324){\makebox(0,0)[lb]{\smash{{{\color[rgb]{.259,0,1}$\tau$}%
}}}}
\put(1895,-6145){\makebox(0,0)[lb]{\smash{{{}{\color[rgb]{0,0,0}$U_r$}%
}}}}
\put(1357,-6740){\makebox(0,0)[lb]{\smash{{{}{\color[rgb]{0,0,0}$U_m$}%
}}}}
\put(2348,-7221){\makebox(0,0)[lb]{\smash{{{}{\color[rgb]{0,.56,0}$\sigma$}%
}}}}
\put(451,-7391){\makebox(0,0)[lb]{\smash{{{}{\color[rgb]{.259,0,1}$\tau$}%
}}}}
\put(649,-6230){\makebox(0,0)[lb]{\smash{{{}{\color[rgb]{0,0,0}$U_\ell$}%
}}}}
\put(508,-5494){\makebox(0,0)[lb]{\smash{{{}{\color[rgb]{0,.56,0}$\sigma$}%
}}}}
\put(1301,-5409){\makebox(0,0)[lb]{\smash{{{}{\color[rgb]{0,0,0}$U'_m$}%
}}}}
\put(3226,-5692){\makebox(0,0)[lb]{\smash{{{}{\color[rgb]{0,0,0}$U_\ell$}%
}}}}
\put(4529,-5126){\makebox(0,0)[lb]{\smash{{{}{\color[rgb]{.259,0,1}$\tau_\ell+\tau_r$}%
}}}}
\put(3594,-7051){\makebox(0,0)[lb]{\smash{{{}{\color[rgb]{0,0,0}$U_m$}%
}}}}
\put(4784,-6202){\makebox(0,0)[lb]{\smash{{{\color[rgb]{0,0,0}$U_r$}%
}}}}
\put(3878,-7533){\makebox(0,0)[lb]{\smash{{{}{\color[rgb]{.259,0,1}$\tau_r$}%
}}}}
\put(2971,-7334){\makebox(0,0)[lb]{\smash{{{}{\color[rgb]{.259,0,1}$\tau_\ell$}%
}}}}
\end{picture}%
\caption{A 1-3 interaction (left) and a 1-1 interaction (right). The value of $v$ is constant across each interaction}
\label{F:2311}
\end{center}
\end{figure}
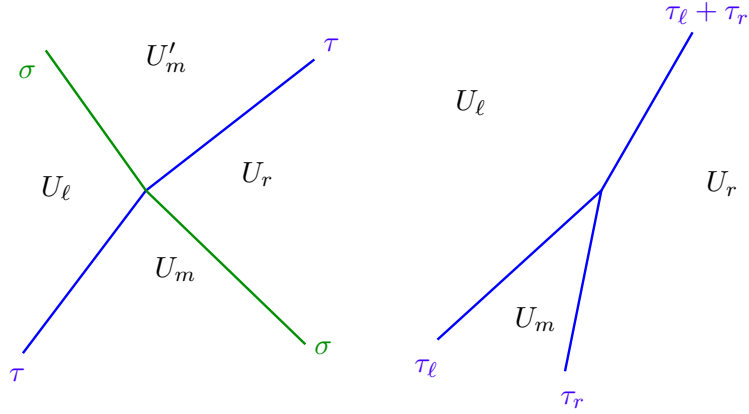 
\subsection{Analysis of 1-2 and 2-3 interactions}
\label{sss:12}
In this paragraph we explicitly discuss the interaction of a 1-shock with a 2-shock. The analysis of the interaction of a 2-shock with a 3-shock is completely analogous. Lemma~\ref{L:ie} below can be loosely speaking formulated as follows: if $\eta$ and the strength of the incoming shocks are sufficiently small, then the outgoing waves are three shocks (and hence, in particular, no outgoing wave is a rarefaction). Also, we have a bound from below and from above on the strength of the outgoing shocks. Note that a result similar to Lemma~\ref{L:ie} is established in~\cite{BaJ}: the novelty of Lemma~\ref{L:ie} is that we have a more precise estimate on the strength of the outgoing 3-shock, compare the left part of~\eqref{e:noraref} with~\cite[eq. (5.9)]{BaJ}. Also, in the case of Lemma~\ref{L:ie} we restrict to data with small total variation. The proof of Lemma~\ref{L:ie} is provided in~\cite{CaravennaSpinolo} and is based on perturbation argument: one firstly establishes Lemma~\ref{L:ie} in the case when $\eta =0$ and then considers the case $\eta >0$. 

To give the formal statement of Lemma~\ref{L:ie} we introduce some notation. We term $U_\ell$, $U_m$ and $U_r$ the left, middle and right state before the interaction, respectively. See Figure~\ref{F:2interazione}, left part, for a representation. In other words, 
\begin{equation}
\label{e:iprima}
 U_m = D_2[ s, U_\ell], \qquad U_r = D_1 [\sigma, U_m] \quad 
 \text{for some $s<0$, $\sigma <0$}.
\end{equation} 
 Also, we denote by $U'_m$ and $U''_m$ the new intermediate states after the interaction, namely 
\begin{equation}
\label{e:idopo}
 U'_{m} = D_1[ \sigma', U_\ell], \qquad U''_m = D_2 [s', U'_m], \qquad U_r = D_{3} [\tau, U''_m] 
\end{equation}
for some $\sigma'$, $s'$ and $\tau \in \R$. 
Here is the formal statement of our result. 
\begin{lemma}
\label{L:ie}
Assume that~\eqref{e:iprima} and~\eqref{e:idopo} hold. Then 
\begin{equation}
\label{e:esse}
 s'= s. 
\end{equation} 
Also, there is $\varepsilon >0$ such that the following holds. 
If $|U_\ell|, \, |s|, \, |\sigma| \leq 1/4$ and $0 \leq \eta < \varepsilon$, then 
\begin{equation}
\label{e:noraref}
-2|\sigma|<\sigma'<-\frac{|\sigma|}{2} \quad \text{and} \quad \frac{1}{100} s\sigma <\tau < 10s \sigma .
\end{equation} 
\end{lemma} 
\subsection{Analysis of 2-2 interactions}
\label{ss:22}
In this paragraph we recall a result from~\cite{CaravennaSpinolo} concerning the interaction between two 2-shocks. 
As usual, we term $U_\ell$, $U_m$ and $U_r$ the left, middle and right state before the interaction. We refer to Figure~\ref{F:2interazione}, right part, for a representation.
Lemma~\ref{C:22shocks} can be loosely speaking formulated as follows. Fix a constant $a>0$ and assume that $U_\ell$, $U_m$ and $U_r$ are all sufficiently close to some state $(a, 0, -a)$. Then the outgoing waves are three shocks. 

The proof of Lemma~\ref{C:22shocks} is given in~\cite{CaravennaSpinolo} and it is divided into two parts: we firstly establish the result in the case $\eta =0$ by relying on the explicit expression of the 2-wave fan curve. We then extend it to the case $\eta >0$ by using a perturbation argument. 

Here is the formal statement. 
\begin{lemma}
\label{C:22shocks} There is a sufficiently small constant $\varepsilon >0$ such that the following holds. Fix a constant $a$ such that $0<a<1/2$ and set $U^{\sharp}: =(a, 0, - a)$. Assume that 
$$
 |U_\ell - U^{\sharp} | \leq \varepsilon a, \quad 
 s_1, s_2 <0, \quad |s_1|, \; |s_2| <\varepsilon a , \quad 0 \leq \eta \leq \varepsilon a \,. 
$$
 Assume furthermore that 
 $$
 U_r = D_2 \Big[s_2, D_2 [s_1, U_\ell] \Big]. 
 $$
 Then there are $\sigma <0$ and $\tau >0$ such that
 \begin{equation}
 \label{e:dini}
 U_r = D_3 \Big[ \tau, D_2 \big[ s_1 + s_2, D_1 [\sigma, U_\ell] \big]\Big]\,.
 \end{equation}
\end{lemma}
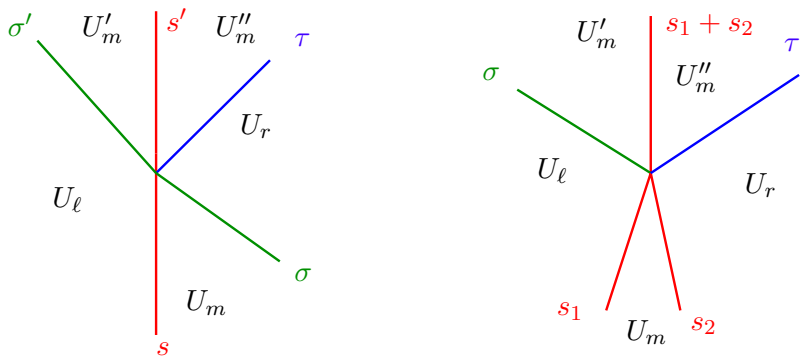
\begin{figure}
\begin{center}
\begin{picture}(0,0)%
\includegraphics{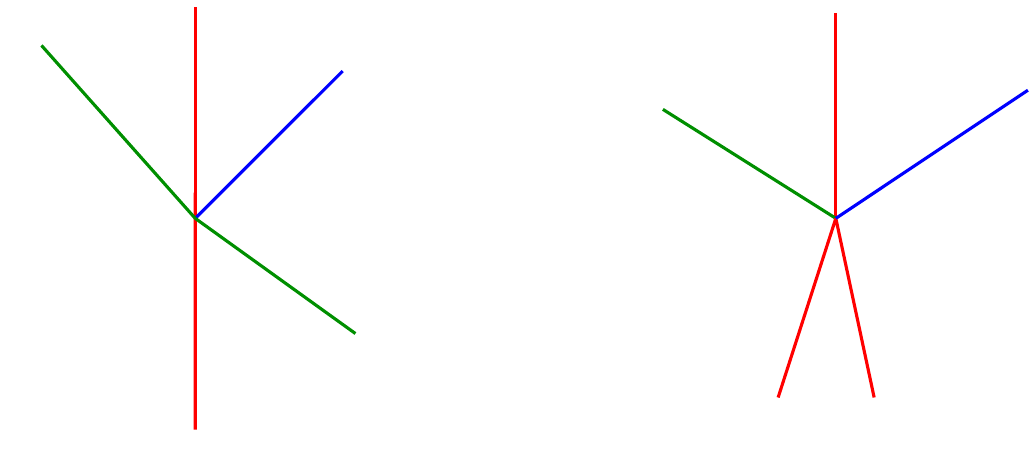}%
\end{picture}%
\setlength{\unitlength}{3947sp}%
\begingroup\makeatletter\ifx\SetFigFont\undefined%
\gdef\SetFigFont#1#2#3#4#5{%
  \reset@font\fontsize{#1}{#2pt}%
  \fontfamily{#3}\fontseries{#4}\fontshape{#5}%
  \selectfont}%
\fi\endgroup%
\begin{picture}(4956,2188)(436,-7951)
\put(3741,-6830){\makebox(0,0)[lb]{\smash{{{}{\color[rgb]{0,0,0}$U_\ell$}%
}}}}
\put(1742,-5938){\makebox(0,0)[lb]{\smash{{{}{\color[rgb]{0,0,0}$U''_m$}%
}}}}
\put(2234,-6000){\makebox(0,0)[lb]{\smash{{{}{\color[rgb]{.259,0,1}$\tau$}%
}}}}
\put(1896,-6553){\makebox(0,0)[lb]{\smash{{{}{\color[rgb]{0,0,0}$U_r$}%
}}}}
\put(2234,-7475){\makebox(0,0)[lb]{\smash{{{}{\color[rgb]{0,.56,0}$\sigma$}%
}}}}
\put(1558,-7659){\makebox(0,0)[lb]{\smash{{{}{\color[rgb]{0,0,0}$U_m$}%
}}}}
\put(1373,-7936){\makebox(0,0)[lb]{\smash{{{}{\color[rgb]{1,0,0}$s$}%
}}}}
\put(728,-7014){\makebox(0,0)[lb]{\smash{{{}{\color[rgb]{0,0,0}$U_\ell$}%
}}}}
\put(912,-5938){\makebox(0,0)[lb]{\smash{{{}{\color[rgb]{0,0,0}$U'_m$}%
}}}}
\put(451,-5938){\makebox(0,0)[lb]{\smash{{{}{\color[rgb]{0,.56,0}$\sigma'$}%
}}}}
\put(1435,-5908){\makebox(0,0)[lb]{\smash{{{}{\color[rgb]{1,0,0}$s'$}%
}}}}
\put(3403,-6215){\makebox(0,0)[lb]{\smash{{{}{\color[rgb]{0,.56,0}$\sigma$}%
}}}}
\put(3864,-7690){\makebox(0,0)[lb]{\smash{{{}{\color[rgb]{1,0,0}$s_1$}%
}}}}
\put(4294,-7844){\makebox(0,0)[lb]{\smash{{{}{\color[rgb]{0,0,0}$U_m$}%
}}}}
\put(4694,-7782){\makebox(0,0)[lb]{\smash{{{}{\color[rgb]{1,0,0}$s_2$}%
}}}}
\put(5032,-6922){\makebox(0,0)[lb]{\smash{{{}{\color[rgb]{0,0,0}$U_r$}%
}}}}
\put(5278,-6031){\makebox(0,0)[lb]{\smash{{{}{\color[rgb]{.259,0,1}$\tau$}%
}}}}
\put(4602,-6246){\makebox(0,0)[lb]{\smash{{{}{\color[rgb]{0,0,0}$U''_m$}%
}}}}
\put(4540,-5908){\makebox(0,0)[lb]{\smash{{{}{\color[rgb]{1,0,0}$s_1+s_2$}%
}}}}
\put(3987,-5969){\makebox(0,0)[lb]{\smash{{{}{\color[rgb]{0,0,0}$U'_m$}%
}}}}
\end{picture}%
\caption{A 1-2 interaction (left) and a 2-2 interaction (right).}
\label{F:2interazione}
\end{center}
\end{figure} 
\subsection{The Riemann problem with well-prepared data}
\label{ss:wp}
In this paragraph we discuss the structure of the solution of Riemann problems with ``well-prepared'' data. More precisely, Lemmas~\ref{l:wp1} and~\ref{l:wp1bis} below state that, under suitable structural assumptions on the constant states $U^{-}$ and $U^{+}$, the solution of the Riemann problem is obtained by juxtaposing three shocks (and hence, in particular, it contains no rarefaction wave). In~\S~\ref{ss:wft:id} we will use these results to discuss the wave-front tracking approximation of the initial datum for a general class of Cauchy problems. 
\begin{lemma}
\label{l:wp1}
There is $0< \varepsilon <1$ such that the following holds.
Fix $U_I \in \R^3$ such that $|U_I| \leq 1/2$. Let $\vec r_{1I}$, $\vec r_{2I}$ and $\vec r_{3I}$ be the vectors 
\begin{equation}
\label{e:ivettori}
 \vec r_{1I} : = \vec r_{1} (U_I), \qquad 
 \vec r_{2I} : = \vec r_{2} (U_{I}), \qquad
 \vec r_{2I} : = \vec r_{3} (U_I). 
\end{equation} 
If $U^-, U^+ \in \R^3$ satisfy 
\begin{equation}
\label{e:hyp1}
 |U^- - U_I | < \varepsilon
\end{equation}
and 
\begin{align}
\label{e:hyp2a}
 | U^+ - U^- + b \vec r_{1I} +b \vec r_{2I} - b \vec r_{2I}| < \varepsilon b 
\end{align}
for some $0<b<\varepsilon$, then the following holds. 
There are $\tau, \sigma$ and $s$ such that 
\begin{equation}
\label{e:soloshock}
 0< \tau < 2b, \quad -2b<\sigma<0, \quad -2b<s<0
\end{equation}
and
\begin{equation}
\label{e:ts}
 U^+ = D_3 \Big[ \tau, D_2 \big[ s, D_1 [\sigma, U^- ]\big]\Big].
\end{equation}
\end{lemma}
\begin{proof}
First, we point out that, if $\varepsilon$ is sufficiently small, then~\eqref{e:hyp2a} implies that 
\begin{equation}
\label{e:bunduetre}
 U^+ - U^- = - b_1 \vec r_{1I} - b_2 \vec r_{2I} + b_3 \vec r_{3I}
\end{equation}
for some $b_1, b_2, b_3$ satisfying
\begin{equation}
\label{e:bunduetre1}
 \frac{1}{2} b < b_1, \, b_2, \, b_3 < \frac{3}{2} b. 
\end{equation}
Next, we use the Local Invertibility Theorem and we determine $\tau, \, s$ and $\sigma$ satisfying~\eqref{e:ts}. Owing to the regularity of the inverse map, we can infer from~\eqref{e:bunduetre} and~\eqref{e:bunduetre1} that 
\begin{equation}
\label{e:pic1}
|\sigma| + |s| + |\tau| < C b.
\end{equation} 
Here and in the rest of the proof, $C$ denotes a universal constant. The precise value of $C$ can vary from line to line.
Next, we recall that the wave fan curve $D_1$ satisfies~\eqref{e:wavecurves13} and we
introduce the notation
\begin{equation}
\label{e:svi1}
 U'_m = D_1 [\sigma, U^-] = U^- + \sigma \vec r_1 (U^-) = U^- +
 \sigma \vec r_{1I} + \sigma \Big[ \vec r_1 (U^-) - \vec r_{1I} \Big]. 
\end{equation}
Also, we term 
\begin{equation}
\label{e:svi2}
 U''_m : = D_2 [s, U'_m] =
 U'_m + s \vec r_{2I}+
 s \Big[ \vec r_2 (U'_m) - \vec r_{2I} \Big]+ 
 \Big[ D_2 [s, U'_m] -U'_m - s \vec r_{2} (U'_m) \Big] 
\end{equation}
By using~\eqref{e:ts} and the explicit expression of the wave fan curve $D_3$ (see~\eqref{e:wavecurves13}) we arrive at 
\begin{equation}
\label{e:sviluppo}
\begin{split}
 U^+ & = U^- + \sigma \vec r_{1I} +
 s \vec r_{2I} + \tau \vec r_{3I} \\
 & \quad +\underbrace{ \sigma \Big[ \vec r_1 (U^-) - \vec r_{1I} \Big]+
 s \Big[ \vec r_2 (U'_m) - \vec r_{2I} \Big]+ 
 \Big[ D_2 [s, U'_m] - U'_m -s \vec r_{2} (U'_m) \Big]+
 \tau \Big[ \vec r_3 (U''_m) - 
 \vec r_{3I} \Big] }_{
 \displaystyle{\mathcal R(\sigma, s, \tau, U^-)}}
 \end{split}
\end{equation}
We recall that $ \vec r_{2} (U'_m)$ is the derivative $d D_2 [s, U'_m]/ ds$ computed at $s=0$. By using~\eqref{e:hyp1},~\eqref{e:pic1}, we obtain that the rest term $\mathcal R$ can be controlled as follows: 
 \begin{equation}
 \label{e:erre1}
 \begin{split}
 |\mathcal R(\sigma, s, \tau, U^-)| & \leq C b (\varepsilon + \varepsilon) +C b^2 + C b \varepsilon \\
 & \leq Cb \varepsilon. 
 \end{split}
 \end{equation} 
 To establish the last inequality, we have use the assumption that $b < \varepsilon$. 
 Next, we compare~\eqref{e:sviluppo} with~\eqref{e:bunduetre} and by using~\eqref{e:erre1} we deduce that 
$$
 |b_1 + \sigma | + |b_2 + s| + |b_3 - \tau | < C \varepsilon b. 
$$ 
Owing to~\eqref{e:bunduetre1}, this implies~\eqref{e:soloshock} provided that $\varepsilon$ is sufficiently small. The proof of the lemma is complete. 
\end{proof}
We only sketch the proof of the following lemma because it is similar to Lemma~\ref{l:wp1}.
Note, furthermore, that Lemma~\ref{l:wp1} can be recovered from Lemma~\ref{l:wp1bis} by taking the limit $\xi \to 0^+$. However, we decided the give the complete statement and proof of Lemma~\ref{l:wp1} to highlight the basic ideas underpinning Lemmas~\ref{l:wp1bis} and~\ref{l:wp1tris}.
\begin{lemma}
\label{l:wp1bis}
There is $0< \varepsilon <1$ such that the following holds.
Fix $U_I \in \R^3$ such that $|U_I| < 1/2$. Let $\vec r_{1I}$, $\vec r_{2I}$ and $\vec r_{3I}$ be the same vectors as in~\eqref{e:ivettori}. 
 Assume that $U_I, U^-, V^-, U^+ \in \R^3$, and $b, \xi \in \R$ satisfy the following conditions: formula~\eqref{e:hyp1} holds and moreover 
 \begin{equation}
\label{e:constraintetaxi}
|V^{-}-U^{-}|< \sqrt{\varepsilon} \frac{ b}{\xi}, \quad 0< b< \varepsilon, \quad 0< \xi< \sqrt{ \varepsilon b}. 
\end{equation} 
Assume furthermore that 
\begin{subequations}
\begin{align}
\label{e:hyp2aCompression1} 
&\text{either} && |U^+ -U^{-}-D_1[- \xi, V^-]+V^{-} + b \vec r_{1I} + b \vec r_{2I} - b \vec r_{3I} |<  b /4
\\
\label{e:hyp2aCompression2} 
&\text{or} && |U^+ -U^{-}- D_2 [-\xi, V^-] +V^{-}+ b \vec r_{1I} + b \vec r_{2I} - b \vec r_{3I}| < b/4
\\
\label{e:hyp2aCompression3}
 &\text{or} && |U^+-U^{-} - D_3 [\xi, V^-] +V^{-}+ b \vec r_{1I} + b \vec r_{2I} - b \vec r_{3I}| < b/4. 
\end{align} 
\end{subequations}
Then~\eqref{e:ts} holds for some $\tau, \, \sigma, \, s$ such that
\begin{subequations}
\label{e:tscompression}
\begin{align}
\label{e:tscompression1} 
 0< \tau < 2b, \quad -2b<s<0, \quad - 2 b - \xi < \sigma < - \xi && 
 \text{if~\eqref{e:hyp2aCompression1} holds} \\
\label{e:tscompression2} 
 0< \tau < 2b, \quad - 2 b - \xi < s < - \xi , \quad - 2b < \sigma < 0 && 
 \text{if~\eqref{e:hyp2aCompression2} holds} \\
\label{e:tscompression3} 
 \xi < \tau < \xi+2 b, \quad -2b<s<0, \quad - 2b < \sigma < 0 && 
 \text{if~\eqref{e:hyp2aCompression3} holds}. 
\end{align}
\end{subequations}
\end{lemma}
\begin{proof}
We only consider the case when~\eqref{e:tscompression2} holds since the analysis of the other cases is analogous, but simpler. We first rewrite~\eqref{e:hyp2aCompression2} as
\begin{equation}
\label{e:pocopoco}
 \Big| U^+ - U^{-} 
 + b \vec r_{1}(U^{-}) + (b+\xi) \vec r_{2}(U^{-}) - b \vec r_{3}(U^{-})
 + \mathcal R_1 (\xi, U_{I}, U^-, V^{-}) \Big| \leq  b/4,
\end{equation}
where the term $ \mathcal R_1$ is defined by setting 
\begin{equation}
\label{e:restoerreuno}
\begin{split}
\mathcal R_1 (b, \xi, U_{I}, U^-, V^{-}) : = & 
 b \Big[ \vec r_{2I}-\vec r_{2}(U^{-}) \Big] +b \Big[ \vec r_{1I}-\vec r_{1}(U^{-}) \Big] 
 +b \Big[ \vec r_{3I}-\vec r_{3}(U^{-}) \Big] \\
 & - \Big[ D_2 [-\xi, V^-] -V^{-}+\xi \vec r_{2}(V^{-}) \Big]
 +\xi \Big[ \vec r_{2}(V^{-})-\vec r_{2}(U^{-}) \Big]. \\
 \end{split}
\end{equation}
Owing to~\eqref{e:hyp1} and~\eqref{e:constraintetaxi}, it satisfies 
\begin{equation}
\label{e:erreunopiccolo}
 | \mathcal R_1 (b, \xi, U_{I}, U^-, V^{-}) | \leq C \varepsilon b + C \xi^2 + 
 C \xi |V^{-}-U^{-}|
 \leq C \sqrt{ \varepsilon} b
\end{equation}
Here and in the rest of the proof, $C$ denotes a universal constant. Its precise value can vary from line to line. 
Next, we use the Local Invertibility Theorem to determine $\tau, \, s$ and $\sigma$ satisfying~\eqref{e:ts}. Owing to the regularity of the inverse map, we have 
\begin{equation}
\label{e:bicsi}
|\sigma| + |s| + |\tau| < C (b + \xi). 
\end{equation} 
We define $U'_m$ and $U''_m$ as~\eqref{e:svi1} and~\eqref{e:svi2} and by arguing as in the proof of Lemma~\ref{l:wp1} we conclude that~\eqref{e:ts} implies
\begin{equation}
\label{e:esattoesatto}
\begin{split}
 U^+ & = U^- + \sigma \vec r_{1}(U^-) +
 s \vec r_{2} (U^-) + \tau \vec r_{3} (U^-) \\
 & \quad +\underbrace{ 
 s \Big[ \vec r_2 (U'_m) - \vec r_{2}(U^-) \Big]+ 
 \Big[ D_2 [s, U'_m] - U'_m -s \vec r_{2} (U'_m) \Big]+
 \tau \Big[ \vec r_3 (U''_m) - 
 \vec r_{3}(U^-) \Big] }_{
 \displaystyle{\mathcal R_2(\sigma, s, \tau, U^-)}}
 \end{split}
\end{equation}
By using~\eqref{e:bicsi} we obtain
\begin{equation}
\label{e:erreduepiccolo}
 | \mathcal R_2(\sigma, s, \tau, U^-)| \leq C (b +\xi)^2
\end{equation}
Finally, we compare~\eqref{e:pocopoco} and~\eqref{e:esattoesatto} and we use~\eqref{e:erreunopiccolo} and~\eqref{e:erreduepiccolo} and we obtain
\begin{equation}
\label{e:cisiamo}
 | \sigma + b| + |s + b + \xi| + |\tau - b| \leq b/4+C \sqrt{\varepsilon} b + C (b + \xi)^2.
\end{equation}
By using the inequality $\xi^2 \leq \varepsilon b $, we eventually arrive at~\eqref{e:tscompression2}. 
\end{proof}
By arguing as in the proof of Lemma~\ref{l:wp1bis}, we establish the following result. Note that Lemmas~\ref{l:wp1} and~\ref{l:wp1bis} can be both recovered as particular cases of Lemma~\ref{l:wp1tris}. 
\begin{lemma}
\label{l:wp1tris}
There is $0< \varepsilon <1$ such that the following holds.
Let $\mathfrak m$ be a Borel probability measure on $\R$.
Fix $U_I \in \R^3$ such that $|U_I| < 1/2$. Let $\vec r_{1I}$, $\vec r_{2I}$ and $\vec r_{3I}$ be the same vectors as in~\eqref{e:ivettori}. 
Fix $U^-, U^+ \in \R^3$ and assume that  
\begin{equation}
\label{e:hyp1bis}
|U^- - U_I | < \varepsilon.
\end{equation}
Assume, furthermore, that the 
 the functions 
$$
    V^- : \R \to \R^3, \qquad 
    \widetilde b, \widetilde \xi, \widetilde \xi_2, \widetilde \xi_3:  \R \to [0, + \infty[.       
$$
satisfy the following conditions for $\mathfrak m$-a.e.~$z \in \R$:
 \begin{subequations}
\begin{gather}
\label{e:constraintetaxitris}
 0\leq \widetilde b(z)< \varepsilon, \qquad 0\leq \widetilde \xi_{i}(z)< \sqrt{ \varepsilon \widetilde b(z)}\ \text{ for }i=1,2,3, \\
\label{e:contraintxi}
 \left[\widetilde \xi_{1}(z)+\widetilde \xi_{2}(z)+\widetilde \xi_{3}(z)\right]|\widetilde V^{-}(z)-U^{-}|< \sqrt{\varepsilon} \,\widetilde b(z). 
\end{gather} 
Finally, set 
\begin{align}
\label{e:integrali}
&&b=\int_\R  \widetilde b(z)d \mathfrak m(z) ,
&& \xi_{1}=\int_\R \widetilde \xi_{1}(z)d\mathfrak m(z) ,
&& \xi_{2}=\int_\R  \widetilde \xi_{2}(z)d\mathfrak m(z) ,
&& \xi_{3}=\int_\R  \widetilde \xi_{3}(z)d\mathfrak m(z) .
\end{align}
and 
assume that 
\begin{gather}
\label{e:hyp2aCompression1bis} 
\begin{split}
 \left|U^+ -U^{-}- \int_\R \left\{D_3 \Big[ \widetilde \xi_{3}(z), D_2 \big[ - 
 \widetilde \xi_{2}(z), D_1 [- \widetilde \xi_{1}(z), \widetilde V^- (z)]
 \big]\Big] - \widetilde V^- (z)\right\}d\mathfrak m(z)
\right.\quad&\\\left.\phantom{\Big[}
+ b \vec r_{1I} + b \vec r_{2I} -  b \vec r_{3I}\right|& < b / 4 . \end{split}
\end{gather} 
\end{subequations}
Then~\eqref{e:ts} holds for some $\tau, \, \sigma, \, s$ such that
\begin{align}
\label{e:tscompression1bis} 
& - 2 b -\xi_{1} \leq \sigma \leq - \xi_{1} , && - 2 b - \xi_{2} \leq s \leq - \xi_{2}, && \xi_{3} \leq \tau \leq  \xi_{3}+ 2b.
\end{align}
\end{lemma}

\section{Construction of the counter-example}
\label{S:step1}
In this section we start the construction of the set of ``bad data'' $\mathfrak B$ as in the statement of Theorem~\ref{T:main}. In other words, we want to construct 
$\mathfrak B$ in such a way that i) $\mathfrak B$ is open in the $\mathcal S (\R)$ topology and ii) for every initial datum in $\mathfrak B$ the solution of the Cauchy problem develops infinitely many shocks in a compact set. 
Loosely speaking, we will construct $\mathfrak B$ as a ball 
(in a functional space) centered at a particular function. 
What we actually do in this section is hence to construct an initial datum $\widetilde U$ satisfying the following requirements: first, the solution of the Cauchy problem with initial datum $\widetilde U$ develops infinitely many shocks. Second, this behavior is robust with respect to sufficiently small perturbations in the Sobolev space $W^{1 \infty}(\R)$. As we will see in~\S~\ref{ss:proofteo}, this is the key step to establish Theorem~\ref{T:main}. 
To construct $\widetilde U$ we proceed according to the following steps.
\begin{itemize}
\item[\S~\ref{ss:bwp}:] we go over the construction of a wave pattern with infinitely many shocks. This construction is basically the same as in~\cite{BaJ}.
\item[\S~\ref{ss:sc}:] we show that this wave pattern can be obtained from a Lipschitz continuous initial datum. However, this does not conclude the construction of $\widetilde U$. Indeed, at the beginning of~\S~\ref{ss:tildeu} we explain that in principle it it may happen that, if we take a very small perturbation of the initial datum, the solution of the Cauchy problem does no more develop infinitely many shocks. In other words, the wave pattern constructed in~\S~\ref{ss:bwp} and~\S~\ref{ss:sc} is not \emph{robust} with respect to perturbations. 
\item[\S~\ref{ss:tildeu}:] we modify the construction given in \S~\ref{ss:bwp} and in \S~\ref{ss:sc} in order to make it robust with respect to perturbations. We eventually obtain an initial datum $\widetilde U$ and Proposition~\ref{p:perturbation} states that 
the solution of the Cauchy problem with initial datum $\widetilde U$ develops infinitely many shocks and that this behavior is robust with respect to sufficiently small $W^{1 \infty}$ perturbations. The proof of Proposition~\ref{p:perturbation} is provided in~\S~\ref{s:pprop}. 
\end{itemize}
In the rest of the present section we always assume that the parameter $\eta$ in~\eqref{e:pertSyst} is sufficiently small to have that Lemma~\ref{L:ie} applies. 
 \subsection{A wave pattern with infinitely many shocks}
\label{ss:bwp} 
In this paragraph we exhibit a wave pattern containing infinitely many shocks. The construction is basically the same as in~\cite{BaJ}, however we recall it 
for the reader's convenience . 
\begin{lemma} 
\label{l:infty}
Fix $q>0$ and assume that $U_I$, $U_{I\!I}$, $U_{I\!I\!I} \in \R^3$ satisfy the following properties:
\begin{enumerate}
\item \label{ite:1inBaJindata} the solution of the Riemann problem between $U_I$ (on the left) and $U_{I\!I}$ (on the right)
contains 3 shocks and the strength of each shock is smaller than $1/4$. 
\item \label{ite:2inBaJindata} The solution of the Riemann problem between $U_{I\!I}$ (on the left) and $U_{I\!I\!I}$ (on the right)
contains 3 shocks and {the strength of each shock is smaller than $1/4$.}
\item \label{ite:3inBaJindata} The following chain of inequalities holds true: $v_I > v_{I\!I} > v_{I\!I\!I}$. 
\end{enumerate}
Then the admissible solution of the Cauchy problem obtained by coupling~\eqref{e:cl2} with the initial datum \begin{equation}
\label{e:W}
 W(x): =
\left\{
 \begin{array}{ll}
 U_I & x < - q \\
 U_{I\!I} & - q< x < q \\
 U_{I\!I\!I} & x> q. \\
 \end{array}
 \right.
\end{equation}
contains infinitely many shocks. 
\end{lemma} 
We refer to Figure~\ref{F:infty} for a representation of the wave pattern contained in the solution of the Cauchy problem obtained by coupling~\eqref{e:cl2} with the initial datum $W$. 
\begin{proof} 
First, we observe that, owing to property~\eqref{ite:3inBaJindata} in the statement of Lemma~\ref{l:infty}, 
\begin{equation}
\label{e:2speed}
 \mathrm{speed}_2[U_{I},U_{I\!I}]= v_I + v_{I\!I} > v_{I\!I} + v_{I\!I\!I} = \mathrm{speed}_2[U_{I\!I},U_{I\!I\!I}]
\end{equation}
In the previous expression, we denote by $\mathrm{speed}_2[U_{I},U_{I\!I}]$ the speed of the 2-shock in the solution of the Riemann problem between $U_I$ (on the left) and $U_{I\!I}$ (on the right). 
In other words, the $2$-shock that is generated at the point $(t, x)=(0, -q)$ is faster than the $2$-shock that is created at the point $(t, x)=(0, q)$ (see 
Figure~\ref{F:infty}). 

Next, we observe that the first interaction that occurs is 
the interaction between the 3-shock generated at $x=-q$ and the 1-shock generated at $x=q$, see again Figure~\ref{F:infty}. Owing to the analysis in \S~\ref{sss:13}, those two shocks essentially cross each other and, most importantly, no 2-wave is generated. After this interaction, the 1-shock generated at $x=q$ interacts with the 2-shock generated at $x=-q$. Owing to Lemma~\ref{L:ie}, this interaction produces three outgoing shocks and 
 the speed of the outgoing 2-shock is the same as the speed of the incoming 2-shock, which is the left hand side of~\eqref{e:2speed}. Also, the new 1-shock generated at this interaction will hit at some later time the left 2-shock: owing to Lemma~\ref{L:ie}, this interaction produces three outgoing shocks. The new 3-shock will then interact with the right 2-shock, producing three outgoing shocks. 
 This mechanism is repeated infinitely many times between $t=0$ and the time $t = \widetilde T$ at which the 2-shocks generated at $x=-q$ and $x=q$ interact, namely 
\begin{equation}
\label{e:T}
 \widetilde T = \frac{2q}{v_{I} - v_{I\!I\!I}}. 
\end{equation}
Note that, in general, owing to the nonlinearity, it may also happen 
that for instance two 3-shocks interact at some point on the right of the right 2-shock. However, owing to~\S~\ref{sss:1133}, these two shocks simply merge and no 2- or 3-waves are generated. 
\end{proof}
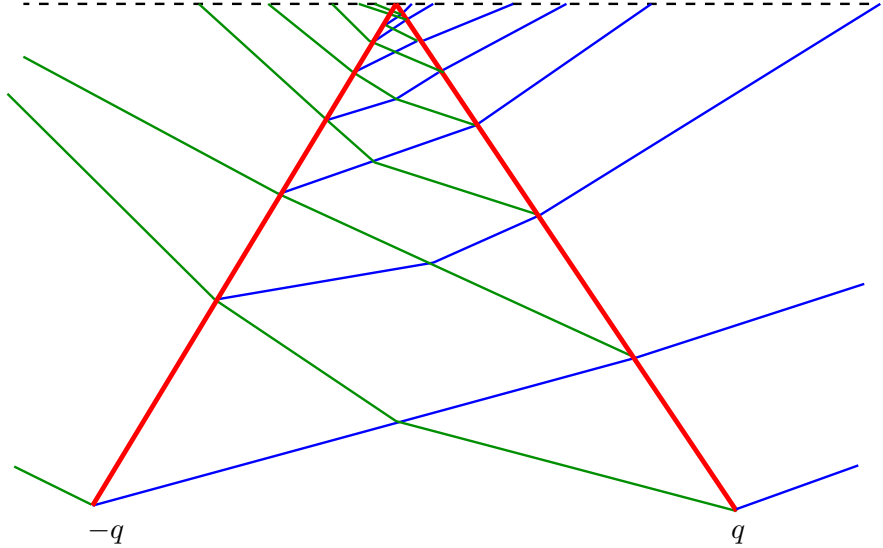
\begin{figure}
\begin{center}
\begin{picture}(0,0)%
\includegraphics{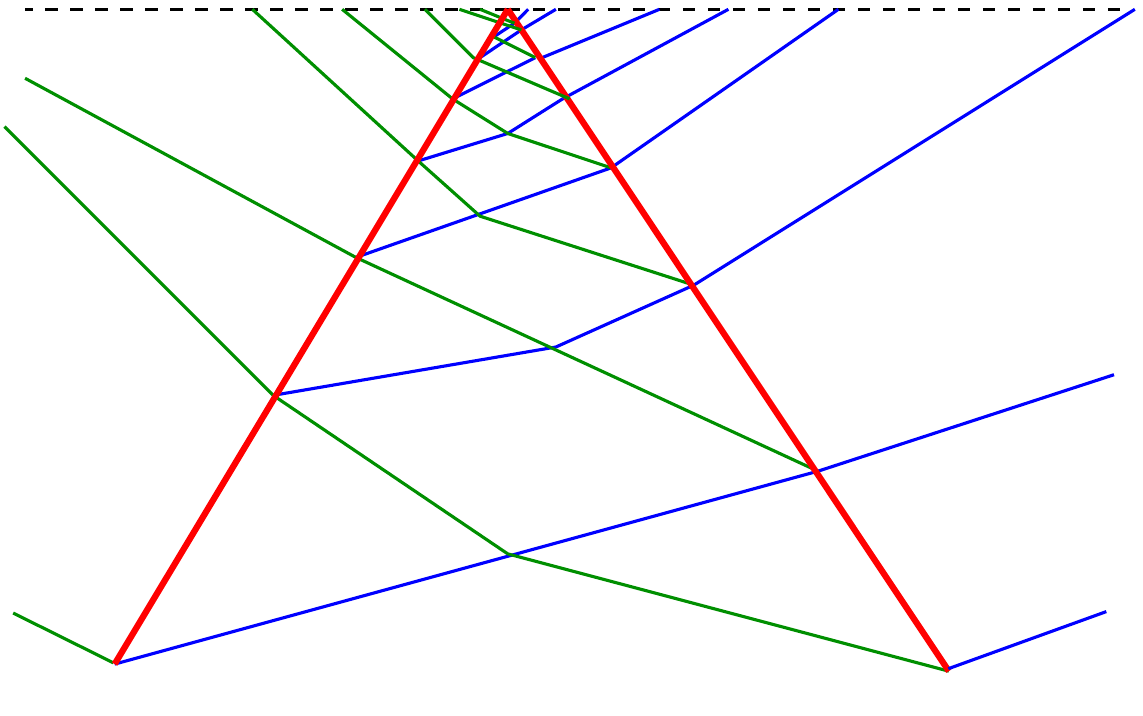}%
\end{picture}%
\setlength{\unitlength}{3947sp}%
\begingroup\makeatletter\ifx\SetFigFont\undefined%
\gdef\SetFigFont#1#2#3#4#5{%
  \reset@font\fontsize{#1}{#2pt}%
  \fontfamily{#3}\fontseries{#4}\fontshape{#5}%
  \selectfont}%
\fi\endgroup%
\begin{picture}(5470,3421)(462,-8582)
\put(981,-8536){\makebox(0,0)[lb]{\smash{{{}{$-q$}%
}}}}
\put(4983,-8536){\makebox(0,0)[lb]{\smash{{{}{$q$}%
}}}}
\end{picture}%
\caption{The solution of the Cauchy problem obtained by coupling system~\eqref{e:cl2} with the initial datum~\eqref{e:W}}\label{F:infty}
\end{center}
\end{figure} 
\subsection{Shock creation analysis}
\label{ss:sc}
This paragraph aims at showing that the wave pattern in Figure~\ref{F:infty} can be exhibited by a solution starting from a Lipschitz continuous initial datum. More precisely, we establish the following result. 
\begin{lemma}
\label{l:startlip} There is a sufficiently small constant $\varepsilon > 0$ such that the following holds. Fix $q =20$, and $U_I \in \R^3$ such that $|U_I| < 1/2$. Let $\omega \in \R$ satisfy $0 < \omega < \varepsilon$ and let $U_{I\!I}$ and $U_{I\!I\!I}$ be the states defined as follows:
\begin{subequations}
\label{e:uduetre}
\begin{equation}
 U_{I\!I} : = D_3 \Big[ \omega, D_2 \big[ - \omega, D_1 [- \omega, U_I] \big] \Big]
\end{equation} 
and 
\begin{equation}
 U_{I\!I\!I} : = D_3 \Big[ \omega, D_2 \big[ - \omega, D_1 [- \omega, U_{I\!I} ] \big] \Big]
\end{equation} 
\end{subequations}
Then the states $U_I$, $U_{I\!I}$ and $U_{I\!I\!I}$ satisfy the hypotheses of Lemma~\ref{l:infty}. Also, there is a Lipschitz continuous initial datum such that the solution $U$ of the Cauchy problem obtained by coupling~\eqref{e:cl2} with this initial datum satisfies $U(1, x) = W(x)$, where $W$ is the same as in~\eqref{e:W}. 
\end{lemma}
The fact that the states $U_I$, $U_{I\!I}$ and $U_{I\!I\!I}$ satisfy the hypotheses of Lemma~\ref{l:infty} follows from the remarks after formula~\eqref{e:wavecurves13}, so we are left to prove the second part of the lemma. The proof is organized as follows. Since we will use the notion of \emph{compression waves} in \S~\ref{sss:cw} we briefly go over this notion for the reader's convenience. In~\S~\ref{sss:technicalemma} we give a technical lemma. In~\S~\ref{sss:pfstartlip} we eventually complete 
the proof of Lemma~\ref{l:startlip}.

\subsubsection{Compression waves}
\label{sss:cw}
 Consider a general, strictly hyperbolic system of conservation laws~\eqref{e:cl}. We term $R_i[s, \bar U]$ the integral curve of $\vec r_i$ passing through $\bar U$, i.e.~the solution of the Cauchy problem~\eqref{e:intcur}. 
Assume that the $i$-th characteristic field is genuinely nonlinear, say $\nabla \lambda_i(U) \cdot \vec r_i (U) >0$ for every $U$. Let $\underline U: = R_i[\underline s, \bar U]$ for some \emph{negative} $\underline s<0$ and observe that the function
\begin{equation}
\label{e:compressionwave}
 U_{cw}(t, x) = 
 \left\{
 \begin{array}{ll}
 \bar U & x< \lambda_i (\bar U) t \\
 R_i[s, \bar U] & x = \lambda_i( R_i[s, \bar U] ) t, \; \underline s <s < 0 \\
 \underline U & x > \lambda_i (\underline U)t \\
 \end{array}
 \right.
\end{equation}
is a smooth solution of the conservation law on $]-\infty , 0[ \times \R$ and at $t=0$ it attains the values 
$$
 U(0, x) = 
 \left\{
 \begin{array}{ll}
 \bar U & x< 0 \\
 \underline U & x > 0 \\
 \end{array}
 \right.
$$
We term the function $U_{cw}$ defined as in~\eqref{e:compressionwave} a \emph{compression wave}. Loosely speaking, compression waves can be regarded as the backward in time analogous of rarefaction waves. 

\subsubsection{A technical lemma}
\label{sss:technicalemma}
First, we make a remark concerning the structure of the integral curves $R_1, \, R_2$ and $R_3$ of system~\eqref{e:pertSyst}. Owing to~\eqref{e:eigenvectors12}, we have the equalities 
\begin{equation}
\label{e:rare13}
 R_1 [\sigma, \bar U] = D_1 [\sigma, \bar U], \quad 
 R_3 [\tau, \bar U] = D_3 [\tau, \bar U]. 
\end{equation}
The proof of Lemma~\ref{l:startlip} is based on the following result.
\begin{lemma}
\label{l:solvecw}
Assume that the hypotheses of Lemma~\ref{l:startlip} are satisfied and that $U_{I\!I}$ and $U_{I\!I\!I}$ are defined by~\eqref{e:uduetre}. If the constant $\varepsilon$ in the statement of Lemma~\ref{l:startlip} is sufficiently small, then
\begin{subequations}
\begin{equation}
\label{e:uduesolve}
 U_{I\!I} : = D_1 \Big[ \sigma, R_2 \big[ s, D_3 [\tau, U_I]
 \big] \Big]
\end{equation}
for some $\tau>\displaystyle{\frac{1}{2} }\omega$, $s<-\displaystyle{\frac{1}{2} } \omega$ and $\sigma <-\displaystyle{\frac{1}{2} } \omega$. Also, 
\begin{equation}
\label{e:utresolve}
 U_{I\!I\!I} : = D_1 \Big[ \sigma^\ast, R_2 \big[ s^\ast, D_3 [
 \tau^\ast, U_{I\!I}] \big] \Big]
\end{equation}
\end{subequations}
for some $ \tau^\ast>\displaystyle{\frac{1}{2} } \omega$, $ s^\ast<-\displaystyle{\frac{1}{2} }\omega$ and $ \sigma^\ast <-\displaystyle{\frac{1}{2} }\omega$.
\end{lemma}
\begin{proof}
We only give the proof of~\eqref{e:uduesolve}, since the proof of~\eqref{e:utresolve} is entirely analogous. 

We basically proceed as in the proof of Lemma~\ref{l:wp1}. First, we point out that~\eqref{e:uduetre} implies that $|U_I - U_{I\!I}| \leq C \omega
\leq C \varepsilon$. Here and in the rest of the proof $C$ denotes some universal constant. Its precise value can vary from line to line. 

By using the Local Invertibility Theorem, we infer that the values of $\tau$, $s$ and $\sigma$ are uniquely determined by imposing~\eqref{e:uduesolve}. Also, we have 
\begin{equation}
\label{e:estimateomega}
 |\tau| + |\sigma | + |s| \leq C \omega. 
\end{equation}
We are left to prove that $\tau>\displaystyle{\frac{1}{2} }\omega$, $s<-\displaystyle{\frac{1}{2} } \omega$ and $\sigma <-\displaystyle{\frac{1}{2} } \omega$. We introduce some notation: we define the states $U'$ and $\underline U'$
by setting 
$$
 U': = D_1 [\omega, U_I], \quad \underline U': = D_3 [\tau, U_I]. 
$$ 
By using~\eqref{e:uduetre} we infer 
\begin{equation}
\begin{split}
\label{e:relazione1}
U_{I\!I} & = D_2 [- \omega, U'] + \omega \vec r_3 (v_{I\!I}) 
\\ &
= U' - \omega \vec r_2 (U') + \Big\{ D_2 [- \omega, U'] -
U' + \omega \vec r_2 (U') \Big\} + \omega \vec r_3 (v_{I\!I}) \\ &
= U_I - \omega \vec r_1 (v_I) - \omega \vec r_2 (U') + \Big\{ D_2 [- \omega, U'] -
U' + \omega \vec r_2 (U') \Big\} + \omega \vec r_3 (v_{I\!I}) \\& 
= U_I - \omega \vec r_1 (v_I) - \omega \vec r_2 (U_I) + \omega \vec r_3 (v_{I}) \phantom{\Big[} \\
& \quad +
\omega \Big\{ \vec r_2 (U_I) - \vec r_2 (U') \Big\} + 
\Big\{ D_2 [- \omega, U'] -
U' + \omega \vec r_2 (U') \Big\} + \omega \Big\{ \vec r_3 (v_{I\!I}) 
- \vec r_3 (v_{I}) \Big\}.\\
\end{split}
\end{equation}
Note that 
\begin{equation}
\label{e:ordine2} 
\omega \Big| \vec r_2 (U_I) - \vec r_2 (U') \Big|+ 
\Big| D_2 [- \omega, U'] -
U' + \omega \vec r_2 (U') \Big| + \omega \Big| \vec r_3 (v_{I\!I}) 
- \vec r_3 (v_{I}) \Big| \leq C \omega^2. 
\end{equation}
By using~\eqref{e:uduesolve} and by arguing as before we obtain 
\begin{equation}
\begin{split}
\label{e:relazione2}
U_{I\!I} & = U_I + \tau \vec r_3 (v_I) + s \vec r_2 (U_I) + 
\sigma \vec r_1 (v_I) \\
& \quad + 
s \Big\{ \vec r_2 (\underline U') -\vec r_2 (U_I) \Big\} + 
\Big\{ R_2 [s, \underline U'] - \underline U'
- s \vec r_2 (\underline U') \Big\} +
\sigma \Big\{ \vec r_1 (v_{I\!I}) - \vec r_1 (v_I) \Big\}, 
\end{split}
\end{equation}
where, owing to~\eqref{e:estimateomega}, 
\begin{equation}
\label{e:ordine22}
\Big| s \Big\{ \vec r_2 (\underline U') -\vec r_2 (U_I) \Big\} \Big|+ 
\Big| R_2 [s, \underline U'] - \underline U'
- s \vec r_2 (\underline U') \Big|+
\Big| \sigma \Big\{ \vec r_1 (v_{I\!I}) - \vec r_1 (v_I) \Big\} \Big|
\leq C \omega^2. 
\end{equation}
By comparing~\eqref{e:relazione1} and~\eqref{e:relazione2} and recalling~\eqref{e:ordine2} and~\eqref{e:ordine22} we obtain that
$$
 |\tau - \omega| + |s + \omega| + |\sigma + \omega| \leq C \omega^2.
$$
Since $\omega>0$, this implies that $\tau>\displaystyle{\frac{1}{2} }\omega$, $s<-\displaystyle{\frac{1}{2} } \omega$ and $\sigma <-\displaystyle{\frac{1}{2} } \omega$ provided that $\varepsilon$ (and hence $\omega$) is sufficiently small. 
\end{proof}
\subsubsection{Proof of Lemma~\ref{l:startlip}}
\label{sss:pfstartlip}
We are now ready to complete the proof of Lemma~\ref{l:startlip}. 

We fix $\omega>0$ and $U_I \in \R^3$, $|U_I| \leq 1/2$. We term $U_{I\!I}$ and $U_{I\!I\!I}$ the states satisfying~\eqref{e:uduetre}. We determine the values $\sigma,\; s, \; \tau, \; \sigma^\ast, \; s^\ast, \; \tau^\ast$ by using~\eqref{e:uduesolve} and~\eqref{e:utresolve}, respectively. Owing to Lemma~\ref{l:solvecw}, we have that $\sigma<0$, $s <0$ and $\tau >0$ and hence 
we can define the function 
$U(t, x)$ by ``juxtaposing'' six compression waves like~\eqref{e:compressionwave}. More precisely, we introduce the following notation:
\begin{equation}
\label{e:statedef}
 \underline U' : = D_3 [\tau, U_I] , 
 \quad \underline U'' : = R_2 [s, \underline U'], 
 \quad \underline U^\ast : = D_3 [\tau^\ast, U_{I\!I}], 
 \quad \underline U^{\ast \ast} : = R_2 [s^\ast, \underline U^\ast] 
\end{equation}
For $t \in [0, 1)$ we define the function $U(t, x)$ by setting
\begin{equation}
\label{e:zeta}
 U(t, x) :=
 \left\{
 \begin{array}{lrclr}
 U_I & &x& <-q + \lambda_3 (U_I) \cdot (t-1) \qquad
 \\
 D_3 [\varsigma, U_I] &
 \text{if there is $0< \varsigma < \tau$:} \hfill
 &x & 
 =- q+ \lambda_3 (D_3 [\varsigma, U_I])\cdot (t-1) \\
 \underline U' & 
 -q+\lambda_3 (\underline U') \cdot (t-1) <
 & x& 
 < -q + \lambda_2 
 (\underline U')\cdot (t-1) 
 \\
 R_2 [\varsigma, \underline U'] &
 \text{if there is $s< \varsigma < 0$:} \hfill
 &x& 
 = -q +\lambda_2 (R_2 [\varsigma, \underline U']) \cdot (t-1) 
 \\
 \underline U'' & 
 -q+\lambda_2 (\underline U'') \cdot (t-1) < 
 &x& < 
 -q+\lambda_1 (\underline U'')\cdot (t-1) 
 \\
 D_1 [\varsigma, \underline U''] &
 \text{if there is $ \sigma< \varsigma < 0$:} \hfill
 &x &
 =-q+\lambda_1 (D_1 [\varsigma, \underline U'']) \cdot (t-1)
 \\
 U_{I\!I} & 
 -q+\lambda_1 (U_{I\!I})\cdot (t-1) < &x& < q
 +\lambda_3 (U_{I\!I}) (t-1) 
 \\
 D_3 [\varsigma, U_{I\!I}] 
 &\text{if there is $0< \varsigma < \tau^\ast $:} \hfill
 & x& 
 =q+\lambda_3 (D_3 [\varsigma, U_{I\!I}]) \cdot (t-1)
 \\
 \underline U^\ast & 
 q + \lambda_3 (\underline U^\ast) \cdot (t-1) < 
 &x& 
 < q+ \lambda_2 (\underline U\cdot )\cdot (t-1) 
 \\
 R_2 [\varsigma, \underline U^\ast] &
 \text{if there is $s^\ast< \varsigma < 0$:} \hfill
 &x& 
 = q+\lambda_2 (R_2 [\varsigma, \underline U^\ast]) \cdot (t-1) 
 \\
 \underline U^{\ast \ast} & 
 q+ \lambda_2 (\underline U^{\ast \ast}) \cdot (t-1) < 
 &x& 
 < q+\lambda_1 (\underline U^{\ast \ast})\cdot (t-1) 
 \\
 D_1 [\varsigma, \underline U^{\ast \ast}] &
 \text{if there is $\sigma^\ast< \varsigma < 0$:} \hfill
 & x& 
 =q+\lambda_1 (D_1 [\varsigma, \underline U^{\ast \ast}]) \cdot (t-1) 
 \\
 U_{I\!I\!I} && x & > q+\lambda_{1} (U_{I\!I\!I})) \cdot (t-1) 
 \\
 \end{array}
 \right.
\end{equation}
Note that the above function is well defined because 
$$
 \lambda_1 (U_{I\!I})\cdot 
 (t-1) - q < \lambda_3 (U_{I\!I}) (t-1) +q.
$$
Indeed, $q=20 >12$ by assumption and $|\lambda_1 (U_{I\!I})|, \,|\lambda_3(U_{I\!I}) | < 6$ owing to~\eqref{e:boundautovalori2}.

Note furthermore that $U(t, x)$ is a locally Lipschitz continuous function on $[0, 1 [ \times \R$ and that $U(1, x) = W(x)$, where $W$ is the same functions as in~\eqref{e:W}. This concludes the proof of the lemma. 
\subsection{A more robust initial datum}
\label{ss:tildeu}
We firstly introduce our analysis with some heuristics.
The analysis in the previous paragraph shows that if the initial datum is given by the same smooth function $U(0, \cdot)$ as in~\eqref{e:zeta}, then the solution of the Cauchy problem exhibits a wave pattern like the one in Figure~\ref{F:infty} and hence, in particular, develops infinitely many shocks. However, the above behavior is not robust with respect to perturbations of $U(0, \cdot)$. The main obstruction that might prevent the formation of infinitely many shocks is the following. We recall that the strength of the shocks generated at time $t=1$ at the points $x=q$ and $x=-q$ is \emph{small}, more precisely it is of the order $\omega <1$. 
By applying the second interaction estimate in~\eqref{e:noraref}, we conclude that the strength of the 1- and 3-shocks bounced back and forth between the two 2-shocks is weaker and weaker as one approaches the intersection point between the two 2-shocks, i.e.~the tip of the triangle in Figure~\ref{F:infty}. This means that, no matter how small a perturbation wave is, if it hits the triangle at a point sufficiently close to the tip it might happen that the perturbation is bigger than the shocks it meets. This might prevent the formation of infinitely many shocks because it might happen that the perturbation annihilates the shock it meets.

In order to make the initial datum more robust with respect to perturbations we add to $U(0, \cdot)$ the function $\Psi$ defined in~\S~\ref{sss:psi}, which is monotone in the direction of the eigenvectors. Very loosely speaking, the heuristic idea underpinning this construction is that in this way only shocks come into play, and no rarefactions. This is made rigorous in \S~\ref{s:pprop} by considering the wave-front tracking approximation of the solution: we prove that the presence of the function $\Psi$ implies that at $t=0$ the wave-front tracking approximation contains only shock waves. This will be the first step in the analysis that will allow us to conclude that the solution of the Cauchy problem develops infinitely many shocks, and that this behavior is robust with respect to perturbations. 

We are left to make one last remark: by looking at the explicit expression of $U$ we realize that there are three compression waves 
that interact at the point $(t, x)=(1, -q)$ and other three that interact at the point $(t, x)=(1, q)$. In \S~\ref{ss:sepcomp} we modify 
the datum $U(0, \cdot)$ by distancing the compression waves one from the other. Loosely speaking, this will 
imply that the corresponding shocks will form at time $t=1$ and \emph{then} they will interact at some later time. This will simplify the perturbation analysis because it will rule out the possibility that the compression waves interact with each other before the corresponding shocks have formed. 

This paragraph is organized as follows: 
\begin{itemize}
\item[\S~\ref{ss:sepcomp}:] we modify $U(0, \cdot)$ by distancing the compression waves one from the other.
\item[\S~\ref{sss:psi}:] we construct the function $\Psi$ ``monotone in the direction of the eigenvectors'' .
\item[\S~\ref{sss:tildeu}:] we eventually define the initial datum $\widetilde U$ in such a way that the solution of the Cauchy problem develops infinitely many shocks and that this behavior is robust with respect to perturbations. See Proposition~\ref{p:perturbation}.
\end{itemize}

\subsubsection{Compression waves separation: definition of $V$}
\label{ss:sepcomp}
We firstly introduce some notation. We fix a sufficiently large $\rho >0$ (its precise value will be discussed in the following, see~\eqref{e:c:parametrirho}), we recall that the parameter $q=20$ is the same as in the statement of Lemmas~\ref{l:infty}
and~\ref{l:startlip} and we set
\begin{subequations}
\label{e:intervalli}
\begin{equation}
\label{e:qs}
 \mathfrak q: = q +3, \qquad \mathfrak p: = q-3.
\end{equation}
We also introduce the following notation: 
\label{e:regions}
\begin{equation}
\begin{array}{lll}
\mathfrak R_\ell: = ]-\rho, -\mathfrak q - \lambda_3 (U_I)[, 
& \phantom{\displaystyle{\int}} &
\mathfrak R^3_\ell: = ]-\mathfrak q - \lambda_3 (U_I), -\mathfrak q - \lambda_3 (\underline U') [, \\ 
\mathfrak R'_\ell: = ] -\mathfrak q - \lambda_3 (\underline U'), 
 -q - \lambda_2 (\underline U') [ & 
 \phantom{\displaystyle{\int}} & 
 \mathfrak R^2_\ell: = ] -q - \lambda_2 (\underline U'),
 -q - \lambda_2 (\underline U'') [, \\ 
 \mathfrak R''_\ell: = ] -q - \lambda_2 (\underline U''), 
 -\mathfrak p - \lambda_1 (\underline U'') [, & \phantom{\displaystyle{\int}} &
 \mathfrak R^1_\ell: = ] 
 -\mathfrak p - \lambda_1 (\underline U''), -\mathfrak p - \lambda_1 (U_{I\!I}) [, \\
 \mathfrak R_m: = ] 
 -\mathfrak p - \lambda_1 (U_{I\!I}) , \mathfrak p - \lambda_3 (U_{I\!I})[ ,
 & \phantom{\displaystyle{\int}} &
 \mathfrak R^3_r: = ]\mathfrak p - \lambda_3 (U_{I\!I}), \mathfrak p - \lambda_3 (\underline U^\ast) [, \\
 \mathfrak R'_r: = ] \mathfrak p - \lambda_3 (\underline U^\ast), 
 q - \lambda_2 (\underline U^\ast) [ & 
 \phantom{\displaystyle{\int}} & 
 \mathfrak R^2_r: = ] q - \lambda_2 (\underline U^\ast),
 q - \lambda_2 (\underline U^{\ast \ast}) [, \\ 
 \mathfrak R''_r: = ] q- \lambda_2 (\underline U^{\ast \ast}) , 
 \mathfrak q - \lambda_1 (\underline U^{\ast \ast}) [, & \phantom{\displaystyle{\int}} &
 \mathfrak R^1_r: = ] 
 \mathfrak q - \lambda_1 (\underline U^{\ast \ast}) , 
 \mathfrak q - \lambda_1 (U_{I\!I\!I}) [, \\
 \mathfrak{R}_r: = ] \mathfrak q - \lambda_1 (U_{I\!I\!I}) , \rho [ & \phantom{\displaystyle{\int}} & \\
 \end{array}
\end{equation}
\setlength{\unitlength}{3947sp}%

We also define the open sets $\mathfrak R_c$ and $\mathfrak R_w$ by setting 
\begin{equation}
\label{e:errec}
\mathfrak R_{c}=\mathfrak R_\ell \cup\mathfrak R'_\ell\cup\mathfrak R''_\ell\cup\mathfrak R_m\cup\mathfrak R_r\cup\mathfrak R'_r\cup\mathfrak R''_r
\end{equation}
and 
\begin{equation}
\label{e:errrew}
\mathfrak R_{w}= \mathfrak R^3_\ell\cup \mathfrak R^2_\ell\cup \mathfrak R^1_\ell\cup \mathfrak R^3_r\cup \mathfrak R^2_r\cup \mathfrak R^1_r,
 \end{equation}
\end{subequations}
respectively. To give an heuristic interpretation of the above notation we point out that, if 
we had $\mathfrak q = \mathfrak p=q$, then the intervals in~\eqref{e:regions} would be the same as in the right hand side of~\eqref{e:zeta}. In particular, we would have that the function $U(0, \cdot)$
is constant on $\mathfrak R_c$ and has a nonzero derivative on $\mathfrak R_w$.

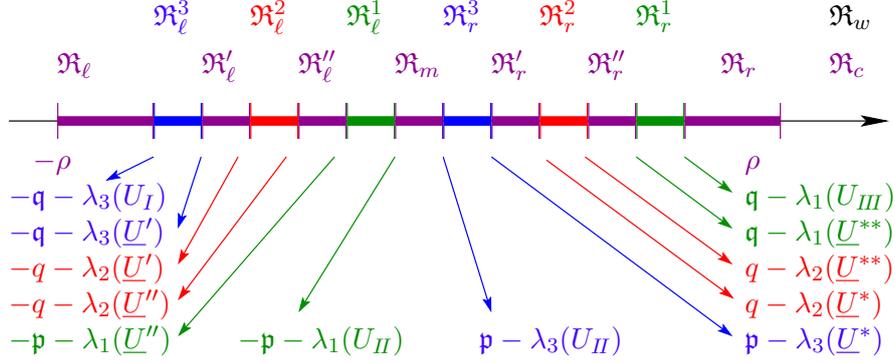
\begin{figure}[ht]
\centering
\begin{picture}(0,0)%
\includegraphics{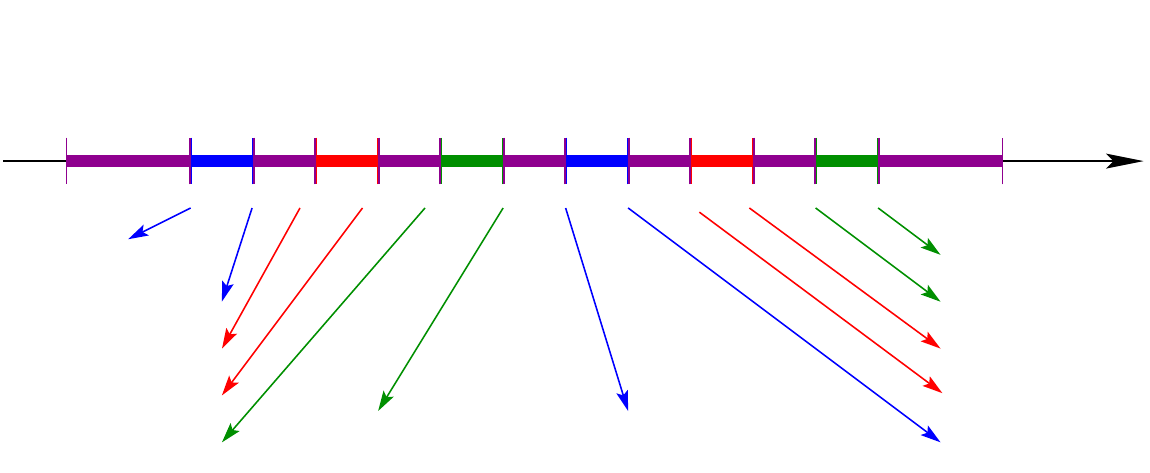}%
\end{picture}%
\setlength{\unitlength}{3947sp}%
\begingroup\makeatletter\ifx\SetFigFont\undefined%
\gdef\SetFigFont#1#2#3#4#5{%
  \reset@font\fontsize{#1}{#2pt}%
  \fontfamily{#3}\fontseries{#4}\fontshape{#5}%
  \selectfont}%
\fi\endgroup%
\begin{picture}(5502,2253)(4786,-2155)
\put(5701,-61){\makebox(0,0)[lb]{\smash{{{\color[rgb]{.259,0,1}$\mathfrak R^3_\ell$}%
}}}}
\put(8701,-61){\makebox(0,0)[lb]{\smash{{{\color[rgb]{0,.56,0}$\mathfrak R^1_r$}%
}}}}
\put(7501,-61){\makebox(0,0)[lb]{\smash{{{\color[rgb]{.259,0,1}$\mathfrak R^3_r$}%
}}}}
\put(9901,-361){\makebox(0,0)[lb]{\smash{{{\color[rgb]{.56,0,.56}$\mathfrak R_c$}%
}}}}
\put(9901,-61){\makebox(0,0)[lb]{\smash{{{\color[rgb]{0,0,0}$\mathfrak R_w$}%
}}}}
\put(5101,-361){\makebox(0,0)[lb]{\smash{{{\color[rgb]{.56,0,.56}$\mathfrak R_\ell$}%
}}}}
\put(9376,-961){\makebox(0,0)[lb]{\smash{{{\color[rgb]{.56,0,.56}$\rho$}%
}}}}
\put(4951,-961){\makebox(0,0)[lb]{\smash{{{\color[rgb]{.56,0,.56}$-\rho$}%
}}}}
\put(6301,-61){\makebox(0,0)[lb]{\smash{{{\color[rgb]{1,0,0}$\mathfrak R^2_\ell$}%
}}}}
\put(8101,-61){\makebox(0,0)[lb]{\smash{{{\color[rgb]{1,0,0}$\mathfrak R^2_r$}%
}}}}
\put(9376,-1636){\makebox(0,0)[lb]{\smash{{{\color[rgb]{1,0,0}$ q - \lambda_2 (\underline U^{**})$}%
}}}}
\put(9376,-1861){\makebox(0,0)[lb]{\smash{{{\color[rgb]{1,0,0}$ q - \lambda_2 (\underline U^{*})$}%
}}}}
\put(6001,-361){\makebox(0,0)[lb]{\smash{{{\color[rgb]{.56,0,.56}$\mathfrak R_\ell'$}%
}}}}
\put(6601,-361){\makebox(0,0)[lb]{\smash{{{\color[rgb]{.56,0,.56}$\mathfrak R_\ell''$}%
}}}}
\put(7201,-361){\makebox(0,0)[lb]{\smash{{{\color[rgb]{.56,0,.56}$\mathfrak R_m$}%
}}}}
\put(4801,-1636){\makebox(0,0)[lb]{\smash{{{\color[rgb]{1,0,0}$ - q - \lambda_2 (\underline U')$}%
}}}}
\put(4801,-1861){\makebox(0,0)[lb]{\smash{{{\color[rgb]{1,0,0}$ - q - \lambda_2 (\underline U'')$}%
}}}}
\put(7801,-361){\makebox(0,0)[lb]{\smash{{{\color[rgb]{.56,0,.56}$\mathfrak R_r'$}%
}}}}
\put(8401,-361){\makebox(0,0)[lb]{\smash{{{\color[rgb]{.56,0,.56}$\mathfrak R_r''$}%
}}}}
\put(9226,-361){\makebox(0,0)[lb]{\smash{{{\color[rgb]{.56,0,.56}$\mathfrak R_r$}%
}}}}
\put(6901,-61){\makebox(0,0)[lb]{\smash{{{\color[rgb]{0,.56,0}$\mathfrak R^1_\ell$}%
}}}}
\put(9376,-1186){\makebox(0,0)[lb]{\smash{{{\color[rgb]{0,.56,0}$\mathfrak q - \lambda_1 (U_{I\!I\!I})$}%
}}}}
\put(9376,-1411){\makebox(0,0)[lb]{\smash{{{\color[rgb]{0,.56,0}$ \mathfrak q - \lambda_1 (\underline U^{**})$}%
}}}}
\put(4801,-1186){\makebox(0,0)[lb]{\smash{{{\color[rgb]{.259,0,1}$ -\mathfrak q - \lambda_3 (U_I)$}%
}}}}
\put(4801,-1411){\makebox(0,0)[lb]{\smash{{{\color[rgb]{.259,0,1}$ -\mathfrak q - \lambda_3 (\underline U')$}%
}}}}
\put(4801,-2086){\makebox(0,0)[lb]{\smash{{{\color[rgb]{0,.56,0}$ -\mathfrak p - \lambda_1 (\underline U'')$}%
}}}}
\put(9376,-2086){\makebox(0,0)[lb]{\smash{{{\color[rgb]{.259,0,1}$ \mathfrak p - \lambda_3 (\underline U^{*})$}%
}}}}
\put(6226,-2086){\makebox(0,0)[lb]{\smash{{{\color[rgb]{0,.56,0}$ -\mathfrak p - \lambda_1 (U_{I\!I})$}%
}}}}
\put(7726,-2086){\makebox(0,0)[lb]{\smash{{{\color[rgb]{.259,0,1}$ \mathfrak p - \lambda_3 (U_{I\!I})$}%
}}}}
\end{picture}%
\caption{Intervals defined in Equations~\eqref{e:intervalli}}
\end{figure}

To construct the function $V$, we fix the parameters $\delta$ and $\omega$ and we set 
\begin{equation}
\label{e:cosaedelta}
U_I: = (\delta, 0, -\delta).
\end{equation}
 We determine the values $\sigma,\; s, \; \tau, \; \sigma^\ast, \; s^\ast, \; \tau^\ast$ by using~\eqref{e:uduesolve} and~\eqref{e:utresolve}, respectively. Finally, we determine $\underline U'$, $\underline U''$, $\underline U^\ast$ and $\underline U^{\ast \ast}$ by using~\eqref{e:statedef}.
We now define the function $V: ]-\rho, \rho[\to \R^3$ in such a way that $V$ is a 3-compression wave on $\mathfrak R^3_\ell\cup \mathfrak R^3_r$, a 2-compression wave on $\mathfrak R^2_\ell\cup \mathfrak R^2_r$ and a 
1-compression wave on $\mathfrak R^1_\ell\cup \mathfrak R^1_r$. 
More precisely, we set 
\begin{equation}
\label{e:VI}
 V( x) :=
 \left\{
 \begin{array}{lrclr}
 U_I & x \in \mathfrak R_\ell \\
 D_3 [\varsigma, U_I] &
 \text{if there is $0< \varsigma < \tau$:} \hfill
 &x & 
 =- \mathfrak q- \lambda_3 (D_3 [\varsigma, U_I]) \\
 \underline U' & x \in \mathfrak R'_\ell \\
 R_2 [\varsigma, \underline U'] &
 \text{if there is $s< \varsigma < 0$:} \hfill
 &x& 
 = -q - \lambda_2 (R_2 [\varsigma, \underline U']) \\
 \underline U'' & x \in \mathfrak R''_\ell \\
 D_1 [\varsigma, \underline U''] &
 \text{if there is $ \sigma< \varsigma < 0$:} \hfill
 &x &
 =-\mathfrak p- \lambda_1 (D_1 [\varsigma, \underline U'']) \\
 U_{I\!I} & x \in \mathfrak R_m \\
 D_3 [\varsigma, U_{I\!I}] 
 &\text{if there is $0< \varsigma < \tau^\ast $:} \hfill
 & x& 
 =\mathfrak p - \lambda_3 (D_3 [\varsigma, U_{I\!I}]) 
 \\ 
 \underline U^\ast & 
 x \in \mathfrak R'_r \\
 R_2 [\varsigma, \underline U^\ast] &
 \text{if there is $s^\ast< \varsigma < 0$:} \hfill
 &x& 
 = q- \lambda_2 (R_2 [\varsigma, \underline U^\ast]) 
 \\
 \underline U^{\ast \ast} & x \in \mathfrak R''_r 
 \\
 D_1 [\varsigma, \underline U^{\ast \ast}] &
 \text{if there is $\sigma^\ast< \varsigma < 0$:} \hfill
 & x& 
 =\mathfrak q- \lambda_1 (D_1 [\varsigma, \underline U^{\ast \ast}]) 
 \\
 U_{I\!I\!I} & x \in \mathfrak R_r \\
 \end{array}
 \right.
\end{equation}
Note that if we had $\mathfrak q=\mathfrak p=q$, then $V$ would coincide with the function $U(0, \cdot)$ defined as in~\eqref{e:zeta}. 
\begin{figure}
\begin{center}
\begin{picture}(0,0)%
\includegraphics{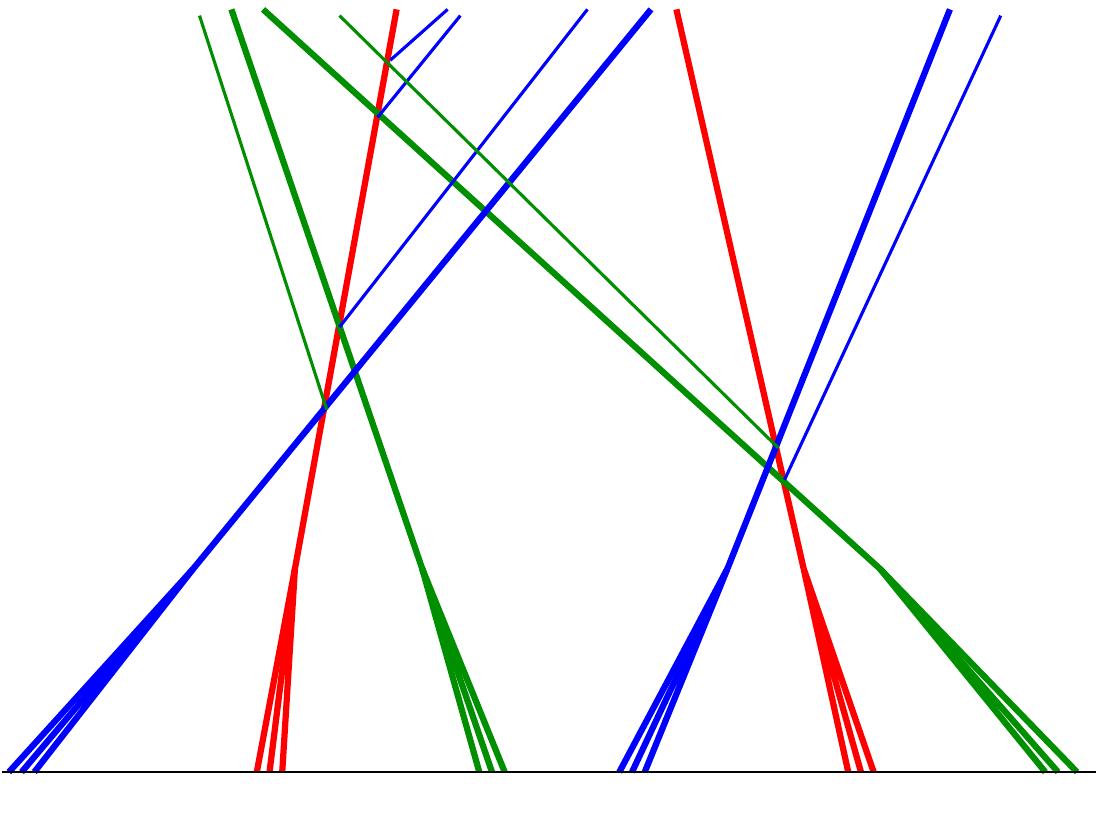}%
\end{picture}%
\setlength{\unitlength}{3947sp}%
\begingroup\makeatletter\ifx\SetFigFont\undefined%
\gdef\SetFigFont#1#2#3#4#5{%
  \reset@font\fontsize{#1}{#2pt}%
  \fontfamily{#3}\fontseries{#4}\fontshape{#5}%
  \selectfont}%
\fi\endgroup%
\begin{picture}(5274,3905)(139,-9283)
\put(2678,-8981){\makebox(0,0)[lb]{\smash{{{}{\color[rgb]{.56,0,.56}$\mathfrak R_m$}%
}}}}
\put(188,-9268){\makebox(0,0)[lb]{\smash{{{}{\color[rgb]{.259,0,1}$\mathfrak R^3_\ell$}%
}}}}
\put(792,-9268){\makebox(0,0)[lb]{\smash{{{}{\color[rgb]{.56,0,.56}$\mathfrak R'_\ell$}%
}}}}
\put(1402,-9268){\makebox(0,0)[lb]{\smash{{{}{\color[rgb]{1,0,0}$\mathfrak R^2_\ell$}%
}}}}
\put(1891,-9268){\makebox(0,0)[lb]{\smash{{{}{\color[rgb]{.56,0,.560}$\mathfrak R''_\ell$}%
}}}}
\put(2385,-9268){\makebox(0,0)[lb]{\smash{{{}{\color[rgb]{0,.56,0}$\mathfrak R^1_\ell$}%
}}}}
\put(3142,-9268){\makebox(0,0)[lb]{\smash{{{}{\color[rgb]{.259,0,1}$\mathfrak R^3_r$}%
}}}}
\put(3631,-9268){\makebox(0,0)[lb]{\smash{{{}{\color[rgb]{.56,0,.56}$\mathfrak R'_r$}%
}}}}
\put(4211,-9268){\makebox(0,0)[lb]{\smash{{{}{\color[rgb]{1,0,0}$\mathfrak R^2_r$}%
}}}}
\put(4699,-9268){\makebox(0,0)[lb]{\smash{{{}{\color[rgb]{.56,0,.56}$\mathfrak R''_r$}%
}}}}
\put(5218,-9268){\makebox(0,0)[lb]{\smash{{{}{\color[rgb]{0,.56,0}$\mathfrak R^1_r$}%
}}}}
\end{picture}%
\caption{The solution of the Cauchy problem with initial datum the function $V$ defined as in~\eqref{e:VI}}
\label{F:V}
\end{center}
\end{figure}
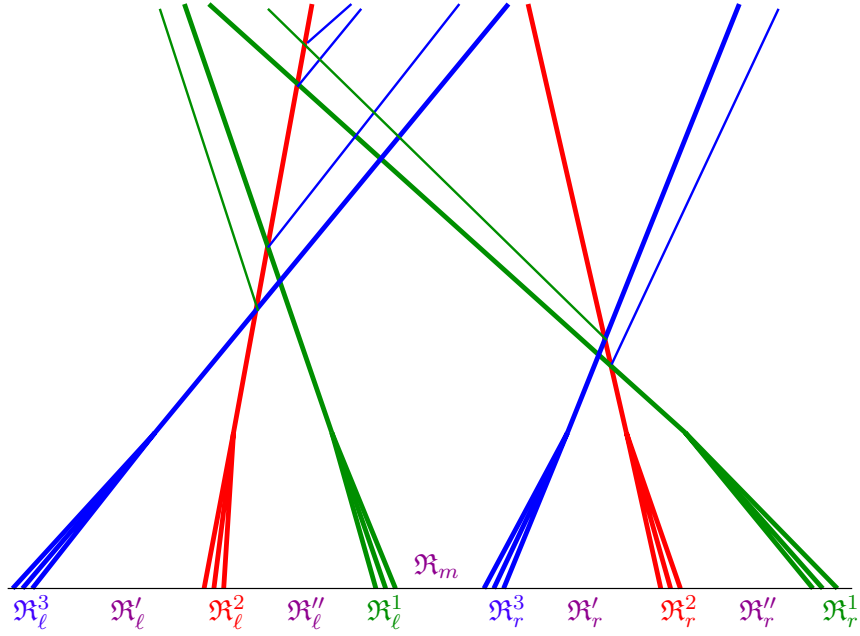 

\subsubsection{Monotonicity in the direction of the eigenvalues: definition of $\Psi$}
\label{sss:psi}
We fix the parameters  $\zeta_c>0$ and $\zeta_w>0$ and we define the function $\Psi: ]-\rho, \rho[ \to \R^3$ by requiring that $\Psi(0)= \vec 0$ and that 
\begin{gather}
\label{e:psi} 
 \Psi '(x) : = 
 \begin{cases}- \zeta_{c} \vec r_{1I}
 -\zeta_{c} \vec r_{2I} +\zeta_{c} \vec r_{3I} &\text{if $x\in \mathfrak R_{c}$} \\ - \zeta_{w} \vec r_{1I}
 -\zeta_{w} \vec r_{2I} +\zeta_{w} \vec r_{3I}
 &\text{if $x\in \mathfrak R_{w}$}\\ 
 \end{cases} 
\end{gather}
In the previous expression, we used the notation $\vec r_{1I}= \vec r_1 (U_I)$, $\vec r_{2I}= \vec r_2 (U_I)$ and {$\vec r_{3I}~=~\vec r_3 (U_I)$}. 
\subsubsection{Definition of the initial datum $\widetilde U$}
\label{sss:tildeu}
We now define the Lipschitz continuous function
$\widetilde U:~\R \to~\R^3$ by setting 
\begin{equation}
\label{e:tildeu}
 \widetilde U(x) : =
 \left\{
 \begin{array}{ll}
 \Phi^-(x) & x<- \rho \\
 V(x) + \Psi (x) & - \rho < x < \rho\\ 
 \Phi^+ (x) & x >\rho. \\
 \end{array}
 \right.
\end{equation} 
In the above expression, the function $V$ is as in~\eqref{e:VI}, the function $\Psi$ is defined in \S~\ref{sss:psi} and the functions $\Phi^-, \Phi^+: \R \to \R^3$ are Lipschitz continuous and defined in such a way that the 
function $\widetilde U$ is continuous and compactly supported. 
We also require that each component of $\Phi^-(x)$ and $\Phi^+(x)$ is monotone. 

We can now state the main result of the present section. Proposition~\ref{p:perturbation} below states that\begin{enumerate}
\item the solution of the Cauchy problem obtained by coupling~\eqref{e:cl2} with the initial datum $U(0, x) = \widetilde U$ has infinitely many shocks;
\item this behavior is robust with respect to sufficiently small perturbations of the initial datum. 
\end{enumerate}
\begin{proposition}
\label{p:perturbation} Fix $q=20$. Let $0< \varepsilon <1$ and fix the parameters
\begin{subequations}
\label{e:allparameters}
\begin{align}
 &\delta: = \varepsilon, \quad \zeta_w := \varepsilon /2,\label{e:c:parametribis} \\
& \eta : = \varepsilon^2, \quad  \omega := \varepsilon^3, \label{e:c:parametri2} \\
& \zeta_c := \varepsilon^9, \quad r := \varepsilon^{10}/2, \label{e:c:parametri3bis} 
\end{align}
Note that by combining the above choices with~\eqref{e:T} and~\eqref{e:uduetre} we get
\begin{equation}
         \widetilde T=\frac{20}{\varepsilon^{3}} \label{e:c:parametriT}
\end{equation}
We also require
\begin{equation}
  \rho := 12 \widetilde T + 40 = 40\left(\frac{6}{\varepsilon^{3}}+1\right)\label{e:c:parametrirho}.
\end{equation}
\end{subequations}
Consider the same function $\widetilde U$ as in~\eqref{e:tildeu}. If the constant $\varepsilon$ is sufficiently small, then, for every initial datum $U_0$ such that 
\begin{equation}
\label{e:palla}
\| U_0 - \widetilde U \|_{W^{1 \infty}} < r,
\end{equation}
the admissible solution of the Cauchy problem obtained by coupling system~\eqref{e:cl2} with the initial datum $U(0, \cdot) = U_0$ has infinitely many shocks 
in the bounded set $]0, 2 \widetilde T[ \times ]- 2q, 2q[$.
\end{proposition}
The proof of Proposition~\ref{p:perturbation} is the most technical part of the paper and it is given in~\S~\ref{s:pprop}. The main result of the present paper, namely Theorem~\ref{T:main}, follows as a corollary from Proposition \ref{p:perturbation}, see \S~\ref{ss:proofteo}. 
\section{Proof of the main results}
\label{s:pprop}
In this section we establish the proof of Theorem~\ref{T:main} and Proposition~\ref{p:perturbation}. More precisely, in \S~\ref{ss:proofteo} we show that Theorem~\ref{T:main} follows as a corollary of Proposition~\ref{p:perturbation}. The rest of the present section is devoted to the proof of Proposition~\ref{p:perturbation}. Since the proof is fairly technical and articulated, we provide a roadmap in~\S~\ref{ss:roadmap}. The proof is established in the remaining paragraphs.

\subsection{Proof roadmap}
\label{ss:roadmap}
In this paragraph we provide the proof outline and we discuss the basic ideas underpinning the analysis in the following paragraphs. 

We start with some heuristic considerations. We recall that the function $V: \mathbb R \to \mathbb R$ is defined as in~\eqref{e:VI}. The qualitative structure of solution of the Cauchy problem with initial datum $V$ is illustrated in Figure~\ref{F:V}: by the time $t=1$, six shocks have formed. More precisely, moving from the left to the right there are a 3-shock, a 2-shock, a 1-shock, a large interval where the solution is constant and then again a 3-shock, a 2-shock and a 1-shock. These shocks interact at some later time and produce a wave pattern with infinitely many shocks. The initial datum $U_0$ is obtained from $V$ by adding the function $\Psi$ and the perturbation $U_0 - \widetilde U$, which is $W^{1 \, \infty}$ small, see~\eqref{e:tildeu},~\eqref{e:palla}. Loosely speaking, the goal of the following paragraphs is to show that adding $\Psi$ and $U_0 - \widetilde U$ to the initial datum does not affect too much the qualitative structure of the solution of the Cauchy problem and, in particular, does not 
jeopardize the formation of infinitely many shocks. Since computing explicit solutions is prohibitive, we rely on the wave-front tracking approximation. 

The proof of Proposition~\ref{p:perturbation} is organized as follows:
\begin{itemize}
\item[\S~\ref{ss:preli}:] We make some preliminary remarks that will be used in the following paragraphs. 
\item[\S~\ref{ss:wft:id}:] We introduce the wave front tracking approximation $U^\nu$
($\nu$ is the approximation parameter)
 of the solution of the Cauchy problem obtained by coupling the Baiti-Jenssen system~\eqref{e:cl2} with the initial datum $U(0, \cdot) = U_0$ satisfying~\eqref{e:palla}. In particular, we construct a piecewise constant approximation of the initial datum and we discuss the waves that are generated at $t=0$. A feature that will be very useful in the analysis at the following paragraphs is that at $t=0$ only shock waves are generated. This is the reason why we introduced the function $\Psi$, monotone in the direction of the eigenvectors, see~\eqref{e:psi},~\eqref{e:tildeu} and the analysis in~\S~\ref{sss:erreelle} and~\S~\ref{sss:id:cw}.
\item[\S~\ref{Ss:qualitativeanalysis}:] We carry on a qualitative analysis of the waves of the wave front-tracking approximation $U^\nu$. In particular, we split the wave generated at $t=0$ in two groups: group A comprises the waves that will contribute to the formation of six ``big shocks'' like in the solution of the Cauchy problem 
with initial datum $V$. Group B comprises all the other wave generated at $t=0$, which in the following will be regarded as perturbation waves. In~\S~\ref{Ss:qualitativeanalysis} we also introduce groups of waves generated at interactions occurring at times $t>0$. They will also be regarded as perturbation waves in the following. Note that perturbation waves are important, even if they are small, because they contribute to the formation of infinitely many shocks.
\item[\S~\ref{Ss:quantitativeanalysis}:] We establish quantitative bounds on the total strength of the waves belonging to the various groups introduced in~\S~\ref{Ss:qualitativeanalysis}. 
\item[\S~\ref{ss:shockgeneration}:] We eventually establish the results concerning the shock formation, see Lemmas~\ref{l:collapse} and~\ref{l:collapse2}. In particular, we show that the wave front-tracking approximation $U^\nu$ contains six ``big shocks'' like the solution of the Cauchy problem with initial datum $V$, see the discussion in~\S~\ref{sss:wpgeneration}. 
\item[\S~\ref{ss:conclusion}:] We eventually conclude the proof of Proposition~\ref{p:perturbation}. In particular, we firstly provide a bound from below on the number of shock fronts in the wave front-tracking approximation $U^\nu$, see Lemma~\ref{l:almeno}. Next, we pass to the limit $\nu \to 0^+$ and we conclude that the number of shocks of the limit solution is infinite on a given compact set. The limit analysis relies on fine properties of the wave front-tracking approximation established by Bressan and LeFloch~\cite{BressanLeFloch}. 
\end{itemize}
We conclude this paragraph with two technical remarks. First, as pointed out in~\S~\ref{ss:wft} in this paper we use the version of the wave front-tracking approximation discussed in the book by Bressan~\cite{Bre}. This version involves the use of two kinds of procedures to solve wave interactions: the \emph{accurate} Riemann solver and the \emph{simplified} Riemann solver. Whether one or the other is used depends on the product of the strength of the incoming waves, see the discussion at the beginning of~\S~\ref{sss:nonphysical} and the analysis in~\cite[Chapter 7]{Bre} for more detailed information. 
To simplify the exposition, in~\S~\ref{Ss:qualitativeanalysis},~\S~\ref{Ss:quantitativeanalysis} and~\S~\ref{ss:shockgeneration} we pretend we always use the \emph{accurate} Riemann solver. The fact that there are actually two kinds of solvers is taken into account in~\S~\ref{ss:conclusion}. 
 
Second, to simplify the notation in the following we denote by 
$\mathcal O(1)$ any quantity which is uniformly bounded and bounded away from $0$, namely there are universal constants $c,C>0$ that satisfy
$$
 0 < c \leq \mathcal O(1) \leq C.
$$. 
\subsection{Preliminary considerations}
\label{ss:preli}
In this paragraph we collect various remarks that we will use in the following. 
First, we fix $\varepsilon>0$ sufficiently small so that Lemmas~\ref{L:ie},~\ref{C:22shocks},~\ref{l:wp1},~\ref{l:wp1bis},~\ref{l:startlip},~\ref{l:solvecw} apply.
Next, we recall formula~\eqref{e:allparameters}: 
\begin{align*}
&v_I - v_{I\, I\, I} = \omega &&\rho =\mathcal O(1) \omega^{-1}, &&\widetilde T=\mathcal O(1) \omega^{-1}, &&\omega=\varepsilon^{3}, &&\delta = \varepsilon,\\ 
& \zeta_w = \varepsilon/2 , &&\eta  = \varepsilon^2,  && \zeta_c = \varepsilon^9, && r = \varepsilon^{10}/2.
\end{align*}
This in particular implies 
\begin{subequations}
\label{e:parameters}
\begin{align} 
&  \zeta_w \omega + \zeta_c \rho + r \rho < \varepsilon \omega <  \varepsilon^{3/4}  \zeta_w \eta  \label{e:c:parametri5} \\
 &  r < \varepsilon \zeta_c <     \varepsilon\omega < \varepsilon \zeta_w \label{e:c:parametri3} 
 \end{align}
\end{subequations}
We will use the above inequalities in the following.

 We recall that the intervals $\mathfrak R_\ell, \dots, \mathfrak R_r$ are as in~\eqref{e:regions} and that the function $V$ is defined in~\eqref{e:VI}. By construction, 
we have 
\begin{equation}
\label{e:totvarV}
 \TV V \leq \mathcal O(1) \omega
\end{equation}
and 
\begin{equation}
\label{e:estimateczero}
 \| V - U_I \|_{C^0} \leq O(1) \omega, \qquad \| V \|_{C^0} \leq O(1) (\delta + \omega). 
\end{equation}
We can infer from~\eqref{e:gnl1},~\eqref{e:uduesolve},~\eqref{e:utresolve},\eqref{e:statedef},\eqref{e:regions} and ~\eqref{e:errrew} that the length of $\mathfrak R_w$ is $\mathcal O(1) \omega$ because the length of $\mathfrak R^3_\ell, \mathfrak R^1_\ell, \mathfrak R^3_r, \mathfrak R^1_r$ is $\mathcal O(1) \omega\eta$ while the length of $\mathfrak R^2_\ell, \mathfrak R^2_r$ is $\mathcal O(1) \omega$. Since $\Psi(0)=\vec 0$, from~\eqref{e:psi},~\eqref{e:c:parametribis} and~\eqref{e:parameters} we get that
\begin{equation}
\label{e:vartotpsi}
 \| \Psi \|_{C^0} \leq \TV \Psi \leq 
 \mathcal O(1) \zeta_w \omega + \mathcal O(1) \zeta_c \rho < \mathcal O(1) \varepsilon\omega. 
\end{equation}
Also, we recall that each component of $\Phi^-$ and $\Phi^+$ is monotone and that $\Phi^-$ and $\Phi^+$ both attain the value $\vec 0$. 
This implies that 
\begin{equation}
\label{e:vartotphi}
\begin{split}
 \| \Phi^- \|_{C^0}+ \| \Phi^+ \|_{C^0} \leq \TV \, \Phi^- + \TV \, \Phi^+
 &\leq |V (-\rho)+ \Psi (-\rho)|+ |V (\rho)+\Psi(\rho)| \\
 & \leq |U_{I}|+|U_{I\!I\!I}|+ \mathcal O(1) (\zeta_{c}\rho + \zeta_{w} \omega ) \\
 & \leq \mathcal O(1) \big[ \delta + \omega + \zeta_{c}\rho+ \zeta_{w} \omega 
   \big]. \\
\end{split}
\end{equation}
By recalling~\eqref{e:allparameters}, \eqref{e:totvarV} and~\eqref{e:vartotpsi} we conclude that
\begin{equation}
\label{e:totvaruzero}
    \TV \widetilde U \leq \mathcal O(1)  \varepsilon. 
\end{equation}
Owing to~\eqref{e:palla}, we have 
\begin{equation}
\label{e:totvarpert}
 \| U_0 - \widetilde U \|_{C^0} + \TV (U_0 - \widetilde U) \leq 
 \mathcal O(1) r\rho
\end{equation}
If $U_0$ satisfies~\eqref{e:palla}, which means that $U_{0}$ is a perturbation of $\widetilde U$, then 
\begin{equation}
\label{e:veperturbation}
 U_0(x) = V(x) + \Psi (x) + \big[U_0(x) - \widetilde U(x)\big] \quad 
 \text{for every $x \in ]-\rho, \rho[$.} 
\end{equation} 
{By the explicit expression of $V$ and by~\eqref{e:gnl} we infer that $|V'(x)| \leq \mathcal O(1) \eta^{-1}$ for every $x\in ]-\rho, \rho[$.  By using~\eqref{e:psi},~\eqref{e:palla} and~\eqref{e:veperturbation} we arrive at 
 \begin{equation}
\label{e:derivata}
 |U'_0(x) |\leq \mathcal O(1)  \eta^{-1}  \quad 
 \text{for every $x \in ]-\rho, \rho[$.} 
\end{equation}  }
By taking into account~\eqref{e:estimateczero},~\eqref{e:vartotpsi},~\eqref{e:palla} and~\eqref{e:parameters}, we infer from~\eqref{e:allparameters} and~\eqref{e:veperturbation} that 
\begin{align}
\label{e:distancefromUno}
 | U_0(x) - U_I | &\leq | U_0(x) - \widetilde U(x) | +|\Psi(x)|+|V(x) - U_I | 
 \\
 \notag & \leq r+ \mathcal O(1) \varepsilon \omega  + \mathcal O(1) \omega  \leq \mathcal O(1) \varepsilon^{3} 
 \quad \text{for every $x \in ]- \rho, \rho[$.}
\end{align}
We point out that by estimates~\eqref{e:c:parametri2},~\eqref{e:c:parametri5}, \eqref{e:totvarV}~\eqref{e:vartotpsi} and~\eqref{e:totvarpert} one has the bound
\begin{equation}
\label{e:totvarU0}
 \TV U_{0} \leq \mathcal O(1) \omega \leq \mathcal O(1) \varepsilon^{3}
 \quad \text{ on $]-\rho, \rho[$}
\end{equation}
Since $U_I = (\delta, 0, - \delta )$, owing to~\eqref{e:c:parametribis} we arrive at 
\begin{equation}
\label{e:stimacizero}
 | U_0(x) | \leq 
 \mathcal O(1) \varepsilon, 
 \quad \text{for every $x \in ]- \rho, \rho[$.}
\end{equation}
\begin{remark}\label{rem:qrho}
We point out that the values attained on $]-2q,2q [ \times ]0, 2 \widetilde T[$ by the admissible solution of the Cauchy problem are only determined by the behavior of the initial datum on $]-\rho, \rho[$. This follows by combining our choice~\eqref{e:c:parametrirho} of $\rho$ with the finite propagation speed, more precisely with~\eqref{e:boundautovalori2}. Indeed, we have 
$$
 \rho - 2 q \ge 12 \widetilde T \ge \max_{|U| \leq 1, i=1, 2, 3} |\lambda_i (U)|
 \cdot 2 \widetilde T. 
$$
In the following, we will only be concerned with the behavior of the initial datum on the interval $]-\rho, \rho[.$ This is justified by the previous considerations and by the fact that we are only interested in the behavior of the solution on $]-2q,2q [ \times ]0, 2 \widetilde T[$. 
\end{remark}

\subsection{Wave front-tracking approximation: initial datum}
\label{ss:wft:id}
In this paragraph we discuss the wave-front tracking approximation of the initial datum. We recall that the intervals $\mathfrak R_\ell, \dots, \mathfrak R_r$ are defined in~\eqref{e:regions}. 

\subsubsection{Mesh definition}
\label{sss:mesh}
We fix an approximation parameter $\nu>0$ and a mesh size $h_\nu>0$. We require that $h_\nu \to 0^+$ when $\nu \to 0^+$. We 
choose $x^\nu_{ 0} < x^\nu_{ 1}< \dots <
 x^\nu_{m_\nu}$ in $]-\rho, \rho[$ so that 
\begin{subequations}\label{e:approx0}
\begin{equation}
\label{e:meshsize}
(1-\varepsilon)  h_\nu \leq 
 x^\nu_{i+1} - x^\nu_{i}  \leq    h_\nu 
 \quad \text{for every $i=0, \dots, m_\nu -1$}. 
\end{equation} 
If $h_\nu$ is sufficiently small, one can as well assume that the extrema of the intervals $\mathfrak R_\ell, \dots, \mathfrak R_r$ are all contained in the set $\big\{ 
x^\nu_{0} , \dots, 
 x^\nu_{m_\nu} \big\}$. Define the wave-front tracking approximation of the initial datum by setting 
 \begin{equation}
 \label{e:unuzero}
 U^\nu_0(x) : = U_0 (x^\nu_i)
 \qquad \quad \text{for $x \in ]x^\nu_i, x^\nu_{i+1}[$ 
 and $i=0, \dots, m_\nu -1$}. 
 \end{equation}
 \end{subequations}
We now describe the waves generated at the grid points
$x^\nu_{0} , \dots, 
 x^\nu_{m_\nu}$ by separately considering the regions 
 $\mathfrak R_\ell, \dots, \mathfrak R_r$.

\subsubsection{Waves generated in $\mathfrak R_{c}=\mathfrak R_\ell \cup \mathfrak R'_\ell \cup\mathfrak R''_\ell\cup \mathfrak R_m \cup \mathfrak R'_r\cup\mathfrak R''_r\cup\mathfrak R_r$} 
\label{sss:erreelle}
We only focus on the analysis of the interval $\mathfrak R_\ell$ because the analysis of the other intervals is entirely similar. 

We fix $x^\nu_i \in \mathfrak R_\ell$ and we consider the Riemann problem between the states 
\[U^- : = \lim_{x \uparrow (x^{\nu }_i) ^{-} } U^{\nu}_0(x)=U_{0}(x_{i-1}^{\nu}) \text{ (on the left),}\quad
U^+ : = \lim_{x \downarrow (x^{\nu }_i) ^{+} } U^{\nu}_0(x)=U_{0}(x_{i}^{\nu}) \text{ (on the right).}\]
\begin{claim}
If~\eqref{e:c:parametri5} and~\eqref{e:c:parametri3} hold, then
the states $U^{-}, U^{+}, U_{I}:=U_{I}$ satisfy the hypotheses of Lemma~\ref{l:wp1} with the choice 
$b=\zeta_{c}(x_{i}^{\nu}-x_{i-1}^{\nu})$
\end{claim} 
\begin{proof}
Hypothesis~\eqref{e:hyp1} in the statement of Lemma~\ref{l:wp1} follows by~\eqref{e:distancefromUno}. Next, we focus on hypothesis~\eqref{e:hyp2a}. We use~\eqref{e:veperturbation} and we recall that $V$ is constant on each connected component of $\mathfrak R_c$, while $\Psi'= \zeta_c (-\vec r_{1I} -\vec r_{2I}+\vec r_{3I} ).$ This implies that, if 
$b=\zeta_{c}(x_{i}^{\nu}-x_{i-1}^{\nu})$, then 
\[
|U^{+}-U_{-} + b \vec r_{1 I} +b \vec r_{2 I} - b \vec r_{2 I}|=|(U_0- \widetilde U)(x_{i}^{\nu})-(U_0- \widetilde U)(x_{i-1}^{\nu})|\stackrel{\eqref{e:palla}}{\leq} r(x_{i}^{\nu}-x_{i-1}^{\nu}),
\]
which owing to~\eqref{e:c:parametri3} gives inequality~\eqref{e:hyp2a}.
\end{proof}

{\sc Conclusion:} By using Lemma~\ref{l:wp1}, we conclude that the only waves created in the open set $\mathfrak R_c$ are 1-, 2- and 3-shocks. In particular, no rarefaction waves are generated. Moreover, owing to~\eqref{e:soloshock} the total variation of all these waves is bounded by $\mathcal O(1) \zeta_{c} \rho.$

\subsubsection{Waves generated in $\mathfrak R_{w}=\mathfrak R^3_\ell \cup \mathfrak R^3_r\cup\mathfrak R^2_\ell\cup\mathfrak R^2_r\cup\mathfrak R^1_\ell\cup\mathfrak R^1_r$} 
\label{sss:id:cw}
We only focus on the analysis of the interval $\mathfrak R^3_\ell$ since the analysis of the other intervals is entirely similar. 
We fix $x^\nu_i \in \mathfrak R^3_\ell$ and we consider the Riemann problem between the states 
\[U^- : = \lim_{x \uparrow x^{\nu }_i } U^{\nu}_0(x)=U_{0}(x_{i-1}^{\nu}) \text{ (on the left),}\quad
U^+ : = \lim_{x \downarrow x^{\nu }_i } U^{\nu}_0(x)=U_{0}(x_{i}^{\nu}) \text{ (on the right).}\]
\begin{claim}
Assume that $U^{-}, U^{+}$ are as at the previous line and that $V^{-}:=V(x_{i-1}^{\nu})$. 
Let $\xi>0$ be the strength of the 3-shock between $V^-$ (on the left) and $V(x_{i}^{\nu})$
(on the right), namely
$$
 V(x_{i}^{\nu}) = D_3 [\xi, V^- ]. 
$$
If $b=\zeta_{w}(x_{i}^{\nu}-x_{i-1}^{\nu})$, then all the hypotheses of Lemma~\ref{l:wp1bis} are satisfied.
\end{claim}
\begin{proof}
Hypothesis~\eqref{e:hyp1} in the statement of Lemma~\ref{l:wp1} follows by~\eqref{e:distancefromUno}. Next, we point out that the condition $0< b <\varepsilon$ is satisfied provided that $\nu$ is sufficiently small. Indeed, $b \leq  \mathcal O(1) \zeta_w h_\nu$ and $h_\nu \to 0^+$ when $\nu \to 0^+$. 

To check the other hypotheses, we first recall that $x^\nu_i \in \mathfrak R^3_\ell$. By combining  the explicit expression of $V$~\eqref{e:VI} with~\eqref{e:gnl1} and~\eqref{e:gnl2} we infer that the derivative of $V$ satisfies $V'(x) = \mathcal O(1) \eta^{-1}$. 
This implies that 
$\xi = \mathcal O(1) (x_{i}^{\nu}-x_{i-1}^{\nu}) \eta^{-1}$ and hence that 
\[
   \sqrt{\varepsilon} \frac{b}{ \xi} = \mathcal O(1) \sqrt{\varepsilon }\zeta_w \eta.
\]
Next, we plug~\eqref{e:vartotpsi} and~\eqref{e:totvarpert} into~\eqref{e:veperturbation}
and we get that the first condition in~\eqref{e:constraintetaxi} is satisfied:
\begin{equation}
\label{e:ufferbacco}
\begin{split}
 |V^- - U^-| =  | \Psi ( x_{i-1}^{\nu}) + U_0 (x_{i-1}^{\nu})- \widetilde U(x_{i-1}^{\nu}) | &
\leq \mathcal O(1) \big(\zeta_w \omega  + \zeta_c \rho+ r  \big)
\\
 &  \leq \mathcal O(1)  \varepsilon  \zeta_w \eta   <
 \sqrt{\varepsilon} \frac{b}{ \xi}. \\
 \end{split}
\end{equation}
To establish the last inequality we used~\eqref{e:c:parametri5}. The condition $\xi^2 < \varepsilon b$ is satisfied because 
$$
 \xi^2 = \mathcal O(1) (x_{i}^{\nu}-x_{i-1}^{\nu})^2 \eta^{-2}=
 \mathcal O(1) h_\nu^2  \eta^{-2} < b = \mathcal O(1) \varepsilon \zeta_w h_\nu
$$ 
 provided that $h_\nu$ is sufficiently small. Finally, we check that~\eqref{e:hyp2aCompression3} holds. We use again~\eqref{e:veperturbation} and we recall that $\Psi'= \zeta_w (-\vec r_{1I} -\vec r_{2I}+\vec r_{3I} )$ on $\mathfrak R_w$. By using~\eqref{e:palla} and~\eqref{e:c:parametri3}, this implies
\begin{align*}
 |U^+-U^{-}& - D_3 [\xi, V^-] +V^{-}+ b \vec r_{1I} + b \vec r_{2I} - b \vec r_{3I}| 
 \\
 &=|U_{0}(x_{i}^{\nu})-U_{0}(x_{i-1}^{\nu}) - V(x_{i}^{\nu}) +V(x_{i-1}^{\nu}) - \Psi(x_{i}^{\nu}) +\Psi(x_{i-1}^{\nu}) |\\
 &=|(U_0- \widetilde U)(x_{i}^{\nu})-(U_0- \widetilde U)(x_{i-1}^{\nu})|{\leq} r(x_{i}^{\nu}-x_{i-1}^{\nu})< \varepsilon b
\end{align*}
All the hypotheses of Lemma~\ref{l:wp1bis} are therefore satisfied. This concludes the proof of the claim.
\end{proof}
By applying Lemma~\ref{l:wp1bis} and using~\eqref{e:tscompression} we arrive at the following conclusion. \\
{\sc Conclusion:} By Lemma~\ref{l:wp1bis}, the only waves created in the intervals 
$\mathfrak R^3_\ell$, $\mathfrak R^3_r$, $\mathfrak R^2_\ell$, $\mathfrak R^2_r$ $\mathfrak R^1_\ell$ and $\mathfrak R^1_r$ are 1-, 2- and 3-shocks. In particular, no rarefaction waves are generated. 

The total variation of all the 3-shocks generated in the intervals $\mathfrak R^3_\ell$, $\mathfrak R^3_r$ is $\mathcal O(1) \omega$ and 
the total variation of all the 1 and 2-shocks generated in the intervals $\mathfrak R^3_\ell$, $\mathfrak R^3_r$ is bounded by $\mathcal O(1) \zeta_{w}\omega \eta $. 

The total variation of all the 2-shocks generated in the intervals $\mathfrak R^2_\ell$, $\mathfrak R^2_r$ is $ \mathcal O(1) \omega  $ and 
the total variation of all the 1 and 3-shocks generated in the intervals $\mathfrak R^2_\ell$, $\mathfrak R^2_r$ is $ \mathcal O(1) \zeta_{w}\omega $.

The total variation of all the 1-shocks generated in the intervals $\mathfrak R^1_\ell$, $\mathfrak R^1_r$ is $\mathcal O(1) \omega$ and 
the total variation of all the 2 and 3-shocks generated in the intervals $\mathfrak R^1_\ell$, $\mathfrak R^1_r$ is $\mathcal O(1) \zeta_{w}\omega \eta$. 

\subsection{Wave front-tracking approximation: qualitative interaction analysis}
\label{Ss:qualitativeanalysis}
In this paragraph we split the waves of the wave front-tracking approximation $U^\nu$ into several groups, that are defined in the following. As we will see in~\S~\ref{ss:shockgeneration} and as we pointed out in the proof roadmap in~\S~\ref{ss:roadmap}, the waves of group A are the waves that will contribute to the formation of a wave pattern similar to the one of the solution with initial datum $V$ (see Figure~\ref{F:V}). The waves of groups B  can be heuristically speaking regarded as perturbation waves. 

We now define the groups A, B, C$_1$, $\dots$, C$_m$. 
In~\S~\ref{ss:wft:id} we discussed the waves that are generated at $t=0$. 
In particular, we proved that only shocks are generated at $t=0$. We split these waves into two groups:
\begin{itemize}
\item {\sc Shocks of group A:} group A comprises 
\begin{itemize}
\item the 3-shocks generated in the intervals $\mathfrak R^3_\ell$ and $\mathfrak R^3_r$ and their right extreme;
\item the 2-shocks generated in the intervals $\mathfrak R^2_\ell$ and $\mathfrak R^2_r$ and their right extreme;
\item the 1-shocks generated in the intervals $\mathfrak R^1_\ell$ and $\mathfrak R^1_r$ and their right extreme.
\end{itemize}
\label{def:strength}
We fix a shock $i \in \mathrm{A}$. Let $\V_{i}$ be its strength, which is defined as in \S~\ref{ss:wft}. Due to the conclusions at the end of~\S~\ref{sss:id:cw}, the total strength of all the shocks of group A is $\mathcal O(1) \omega$, namely 
\begin{equation}
\label{e:gruppoa}
 \sum_{i \in \mathrm{A}} \V_{i} = \mathcal O(1) \omega
\end{equation}
 \item {\sc Shocks of group B:} group B comprises all the shocks generated at $t=0$ in the interval $]-\rho, \rho[$ which are not comprised in group A. In other words, group B comprises 
 \begin{itemize}
 \item the 1, 2 and 3-shocks generated in the open interval $\mathfrak R_c$;
 \item the 1 and 2-shocks generated in the intervals $\mathfrak R^3_\ell$ and $\mathfrak R^3_r$;
\item the 1 and 3-shocks generated in the intervals $\mathfrak R^2_\ell$ and $\mathfrak R^2_r$;
\item the 2 and 3-shocks generated in the intervals $\mathfrak R^1_\ell$ and $\mathfrak R^1_r$.
\end{itemize}
Owing to the conclusions at the end of~\S~\ref{sss:erreelle} and of~\S~\ref{sss:id:cw}, the total strength of all these shocks can be bounded by 
\begin{equation}
\label{e:gruppob}
 \sum_{i \in \mathrm{B}} \V_{i} \leq \mathcal O(1) 
 \big( \rho \zeta_c
 + \omega \zeta_w  \big). 
\end{equation}
\end{itemize}
We now want to track the evolution of the shocks of groups A and B by discussing their interactions. 
For the time being, we do not take into account the fact that in some cases we have to use a \emph{simplified Riemann solver} (see~\S~\ref{ss:wft}). We will take into account the presence of non-physical waves in \S~\ref{ss:conclusion}. 
We separately consider the following cases: 
\begin{enumerate}
\item\label{item:1gruppi} We fix two shocks, $i$ and $j$, and we assume that $i$ is either a 1 or a 3-shock and $j$ is a 2-shock. Just to fix the ideas, let us assume that $i$ is a 3-shock. 
In this step, we do not care whether $i$ and $j$ belong to group A or B. Let $\V_i$ and $\V_j$ be their strengths and we assume that $i$ and $j$ interact at some point. By Lemma~\ref{L:ie} all the outgoing waves are shocks. By definition, we still call $i$ the outgoing 3-shock and we still call $j$ the outgoing 2-shock. 
Also, the outgoing $i$ belongs to the same group (A or B) as the incoming $i$, and the same happens for $j$. We say that the outgoing 1-shock is the \emph{new shock} which is created at the interaction. This new shock belongs neither to A nor to B: we define a new group C$_1$ in the following. 

We conclude by recalling some interaction estimates: let $\V_i'$ and $\V_j'$ be the strengths of $i$ and $j$ after the interaction. By Lemma~\ref{L:ie}, $\V_j' = \V_j$. Also, we recall~\cite[formula (7.31) p. 133]{Bre}, which states that 
\begin{equation}
\label{e:formulabressan}
 |\V'_i - \V_i| \leq \mathcal O(1) 
 \V_i \V_j . 
\end{equation}
Also,~\cite[formula (7.31) p. 133]{Bre} implies that the strength of the new shock generated at the interaction is bounded by 
$\mathcal O(1) 
 \V_i \V_j $. 
 
\item \label{item:2gruppi} We fix two shocks, $i$ and $j$, and we assume that $i$ is a 3-shock and $j$ is a 1-shock. Owing to the analysis in~\S~\ref{sss:13}, the outgoing waves are a 3-shock and a 1-shock. By definition, we still call $i$ the outgoing 3-shock $j$ the outgoing 1-shock. We say that the outgoing $i$ belongs to the same group as the incoming $i$. The same holds for $j$. Finally, we recall some quantitative interaction estimates: we term $\V_i$ and $\V_j$, $\V'_i$ and $\V'_j$ the strengths of $i$ and $j$ before and after the interaction, respectively. 
Owing to the analysis in~\S~\ref{sss:13}, 
\begin{equation}
\label{e:interae2}
 \V'_i = \V_i, \qquad 
 \V'_j = \V_j
\end{equation}
\item \label{item:3gruppi} We fix two shocks $i$ and $j$ and we assume that they are both 2-shocks. We term $\V_i$ and $\V_j$ their strengths and we assume that they interact at some point. We set
\begin{align*}
&a : = \delta, &&U^\sharp: = U_I = (\delta, 0, -\delta), &&s_1= - \mathcal V_i, &&s_2 = - \mathcal V_j
\end{align*}
and we claim that the hypotheses of Lemma~\ref{C:22shocks} are satisfied. The hypothesis
\begin{itemize}
\item $|U_\ell - U^\sharp| \leq \varepsilon a$ holds owing to~\eqref{e:distancefromUno} and to the fact that $a = \delta = \varepsilon$ (see~\eqref{e:c:parametribis}). 
\item $0 \leq \eta \leq \varepsilon a$ is satisfied because $\eta = \varepsilon^2$ owing to~\eqref{e:c:parametri2}. 
\item $|s_1|, |s_2| \leq \varepsilon$ are satisfied. Indeed, owing to Lemma~\ref{L:v} 
 the maximal strength of a 2-shock is bounded by the total variation of the $v$ component. The total variation of the $v$ component at $t=0$ satisfies $\mathrm{TotVar \,} v_0 \leq \mathcal O(1) \varepsilon^3$ by~\eqref{e:totvarU0}. Since the total variation of a scalar conservation law is a monotone non increasing function with respect to time~\cite{Kru}, we can conclude that $|s_1|, |s_2| \leq \varepsilon$. 
\end{itemize}
Lemma~\ref{C:22shocks} states then that the outgoing waves at the interaction point are three shocks. 
We now separately consider the following cases:
\begin{itemize}
\item if $i$ belongs to A and $j$ belongs to B, then we term $i$ the outgoing 2-shock and we prescribe that it still belongs to A. Note that the strength of $i$ after the interaction is $\V'_i = \V_i + \V_j$. We set $\V'_j=0$, in such a way that
\begin{equation}
\label{e:interae3}
 \V'_i + \V'_j = \V_i + \V_j .
\end{equation}
\item if $i$ and $j$ both belong to either A or B, then we proceed as follow. Just to fix the ideas, assume that $i$ is the fastest shock among the two, namely $i$ is on the left of $j$ before the interaction. By definition, we still call $i$ the outgoing 2-shock and we prescribe that it belongs to the same group (A or B) as the incoming shocks. We also set $\V'_j=0$, in such a way that~\eqref{e:interae3} holds. 
\end{itemize}
In both cases, we say that the outgoing 1 and 3-shock are \emph{new shocks} generated at the interaction. Note again that these new shocks belong neither to A nor to B: they will belong to the group C$_1$ defined in the following. 

\item \label{item:4gruppi} We fix two shocks, $i$ and $j$, and we assume that they belong to same family, which can be either 1 or 3. Just to fix the ideas, let us assume that they are both 3-shocks. Owing to the analysis in~\S~\ref{sss:1133}, the only outgoing wave is a 3-shock. 
We separately consider the following cases:
\begin{itemize}
\item if $i$ belongs to A and $j$ belongs to B, then we term $i$ the outgoing 3-shock and we prescribe that it still belongs to A. Note that the strength of $i$ after the interaction is $\V'_i = \V_i + \V_j$. We set $\V'_j=0$ in such a way that~\eqref{e:interae3} holds. 
\item if $i$ and $j$ both belong to either A or B, then we proceed as in case \ref{item:3gruppi}). Just to fix the ideas, assume that $i$ is the fastest shock among the two, namely $i$ is on the left of $j$ before the interaction. By definition, we still call $i$ the outgoing 3-shock and we say that it belongs to the same group as the incoming shocks. We also set $\V'_j=0$, in such a way that~\eqref{e:interae3} holds.  
\end{itemize} 
\end{enumerate}
We explicitly stress three properties following from the analysis of cases \ref{item:1gruppi})--\ref{item:4gruppi}) above. First, the outgoing waves are always shocks. Second, new shocks are only created when the interaction involves at least one 2-shock. Third, the new shocks created at the interaction are either 1- or 3-shocks, namely no new 2-shocks are created. 

We now focus on the new shocks created at interaction points. 
They can only be 1- and 3-shocks, since by Lemma~\ref{L:v} and by the definitions in \ref{item:1gruppi}) and \ref{item:3gruppi}) above no new 2-shock can arise.
We now collect them into a sequence of groups C$_{m}$ defined by recursion on $m\in\N$. 
\begin{itemize}
\item {\sc Group C$_1$:} we term C$_1$ the group of \emph{new shocks} that are generated at the following interactions:
\begin{itemize}
\item between a shock $i$ and a shock $j$ both belonging to group A. 
\item between a shock $i$ belonging to group A and a shock $j$ belonging to group B. 
\item between a shock $i$ and a shock $j$ both belonging to group B.
\end{itemize}
As mentioned before, C$_1$ only comprises 1- and 3-shocks. 
It follows from the analysis in~\S~\ref{sss:13} that if two shocks of group C$_1$ interact then either they basically cross each other or they merge: to label the outgoing waves at the interaction, we proceed as in case {\ref{item:4gruppi})} above. Note furthermore that, if a shock $i \in \mathrm{C}_1$ merges with a shock $j \in \mathrm{A} \cup \mathrm{B}$ of the same family, we term $j$ the outgoing shock and we set $\V'_i=0$, $\V'_j = \V_i + \V_j$. Hence, the only possibility for the generation of new waves is the one discussed at the next item. 
\item Group C$_{m+1}$: we term C$_{m+1}$ the group of \emph{new shocks} that are generated at interactions between a shock $i$ belonging to group C$_m$ and a 2-shock $j$ belonging to either group A or B.
As mentioned before, C$_m$ only comprises 1- and 3-shocks.
At interactions among shocks in C$_{m+1}$ the shocks can basically either basically cross each other or they can merge. In any case, no new shock is created: to label the outgoing waves at interaction points, we proceed as in case {\ref{item:4gruppi})} above.
Also, if a shock $i \in \mathrm{C}_{m+1}$ merges with a shock $j \in \mathrm{C}_m \cup \mathrm{A} \cup \mathrm{B}$ of the same family, we denote by $j$ the outgoing shock and we set $\V'_i=0$: in this way equality \eqref{e:interae3} is satisfied. 
\end{itemize}
In this way we have classified all the shocks of the wave-front tracking approximation $U^\nu$. 
\subsection{Wave front-tracking approximation: quantitative interaction estimates}
\label{Ss:quantitativeanalysis}
This paragraph aims at establishing Lemma~\ref{l:quantitative} below. In the statement, $\mathcal V_i$ denotes as usual the strength of the shock $i$.  
\begin{lemma}
\label{l:quantitative}
There is a constant $K>0$ such that, if 
the constant $\varepsilon$ in the statement of Proposition~\ref{p:perturbation} is sufficiently small, then we have the following estimates: for every $ t >0 $
\begin{subequations}
\label{e:shock}
\begin{align}
\label{e:shockA}
 &\sum_{i_A \in \mathrm{A}} \V_{i_A} ( t^+) \leq 
 K \omega,
\\
\label{e:shockB}
 & \sum_{i_B \in \mathrm{B}} \V_{i_B} ( t^+) \leq 
 K \omega \varepsilon, 
\\
\label{e:shockC}
 & \sum_{i \in \mathrm{C}_m} \V_{i} (t^+) \leq 
 (2 K \omega)^{m+1}. 
\end{align}
\end{subequations}
\end{lemma}
\begin{proof}
We point out that owing to~\eqref{e:gruppoa} and~\eqref{e:gruppob} combined with~\eqref{e:c:parametri5} we can choose $K$ in such a way that we have the inequalities
\begin{equation}
\label{e:allinizio}
 \sum_{i_A \in \mathrm{A}} \V_{i_A} ( t=0) \leq 
 \frac{1}{2} K \omega, \qquad
 \sum_{i_B \in \mathrm{B}} \V_{i_B} ( t=0) \leq 
 \frac{1}{2} K (\omega \zeta_w + \rho \zeta_c )\leq 
 \frac{1}{2} K \omega \varepsilon. 
\end{equation}
The shocks of groups C$_m$, $m \in \mathbb N$, do not exist at $t=0$, but we can adopt the notation that their strength is $0$, in such a way that~\eqref{e:shockC} is formally satisfied. The proof of the lemma is based on the following argument: we assume that estimates~\eqref{e:shockA},~\eqref{e:shockB} and~\eqref{e:shockC} are satisfied for every $t < \bar t$ and we show that they are satisfied for $t= \bar t$. The technical details are organized in the following four steps. \\
\firststep\step{we make some preliminary considerations}. We first introduce a new notation: we denote by D the group 
\begin{equation}
\label{e:D}
 \mathrm D:= \mathrm B \cup \bigcup_{m =1}^\infty \mathrm{C}_{m}.
\end{equation}
In the above expression, the groups B and C$_m$ are as in \S~\ref{Ss:qualitativeanalysis}. 
Note furthermore that here an in the following we term \emph{groups} the sets $\mathrm A$, $\mathrm B$, $\mathrm C_m$, $\mathrm D$, while we use the term \emph{family} as a shorthand for \emph{characteristic family}.

Note that by combining all the inequalities in~\eqref{e:shock} and~\eqref{e:c:parametri2} we get that, if $\varepsilon$ is sufficiently small, then 
\begin{subequations}
\label{e:shocktot}
\begin{align}
\label{e:shockD}
 & \sum_{i \in \mathrm{D} } \V_{i} (t^+) \leq 
 K \omega\varepsilon + \sum_{m=1}^\infty 
 (2 K \omega)^{m+1} \leq 2K \omega \varepsilon 
 && \text{for every $ t < \bar t $}
 \\
\label{e:shockAD}
 &  \sum_{i \in \mathrm{A}\cup\mathrm{D} } \V_{i} (t^+) \leq 
 K \omega+ 2 K \omega\varepsilon \leq 2K\omega
 && \text{for every $ t < \bar t $}
\end{align}
\end{subequations}
Note furthermore that the quantities at the left hand side of~\eqref{e:shock},~\eqref{e:shockD} can only change at interaction times. 
 
\step{we establish the bound on $\sum_{i_A \in \mathrm{A}} \V_{i_A} $} Note that the only ways $\sum_{i_A \in \mathrm{A}} \V_{i_A} $ can change are the following interactions:
\begin{enumerate}
\item Interactions where a shock $i\in$ A with strength $\V_{i}$ merges with a shock $j$ with strength $\V_{j}$ of the same family and of the same group A. In this case~\eqref{e:interae3} ensures that $\sum_{i_A \in \mathrm{A}} \V_{i_A} $ does not change at this interaction. For this reason, in the following we neglect these interactions.
\item \label{item:2inter}  Interactions where a shock $i_A \in$A with strength $\V_{i_A}$ merges with a shock $j_{D}$ of the same family but of group D.
In this case $\V'_{i_A} = \V_{i_A} + \V_{j_{D}}$. Each shock $j_{D}$ of group $D$ may have at most one of these interactions: let $J^{i_{A}}_{D}$ be the subset of shocks of group $D$ that merge with $i_{A}$. 
\item \label{item:3inter} Interactions where a shock $i_A \in$A with strength $\V_{i_A}$ interacts with a shock $j$ of a different family. In this case by the  interaction estimate~\eqref{e:formulabressan} one has
\[
 \V'_{i_A} \leq \V_{i_A} + \mathcal O(1) \V_{i_A} \V_j = \V_{i_A} \Big(1 + \mathcal O(1) \V_j \Big) .
\]
Each shock $j$ may interact at most once with a given shock $i_{A}$ of a  different family.
\end{enumerate}
We recall that all shocks in group A are generated  at time $t=0$ and we track the evolution of a given shock $i_A \in $A between time $t=0$ and $t = \bar t$.  
If the shock $i_{A}$ only interacts with shocks $j \in J_{D}^{i_{A}}\subseteq$D and with waves $j_{1},\dots,j_{k}$ of different families then by~\ref{item:2inter}),~\ref{item:3inter}) above one has the inequality
\begin{align*}
 \V_{i_A} ( \bar t^+)
 & \leq 
 \left(\V_{i_A} (t = 0) + \sum_{j \in J_{D}^{i_{A}} } \V_{j} \right) \prod_{j \in A \cup D} \Big(1 + \mathcal O(1) \V_{j} \Big). 
 \end{align*}
 Note that the last factor in the above expression does not depend on $i_{A}$. Also, by using the inequality   $e^{x}\geq1+x$ we get 
 \[
 \prod_{j \in A \cup D} \Big(1 + \mathcal O(1) \V_{j} \Big) 
 \leq e^{\mathcal O(1)\sum_{j \in A \cup D}\V_{j}}.
 \]
 We now sum over all the shocks $i_A \in$A and be obtain
\begin{align*}
 \sum_{i_A \in \mathrm{A}} \V_{i_A} ( \bar t^+)
 & \leq 
 \left(\sum_{i_A \in \mathrm{A}}\V_{i_A} (t = 0) +\sum_{i_{A}\in\mathrm{A}} \sum_{j \in J_{D}^{i_{A}} } \V_{j} \right) \prod_{j \in A \cup D} \Big(1 + \mathcal O(1) \V_{j} \Big)
\\
 & \leq 
 \left(\sum_{i_A \in \mathrm{A}}\V_{i_A} (t = 0) +\sum_{j_{D}\in\mathrm{D}} \V_{j} \right) \exp\left({\mathcal O(1)\sum_{j \in A \cup D}\V_{j}}\right)
\end{align*}
The last inequality holds because two sets $J^{i_{A}}_{D}$ and $J^{i_{A}'}_{D}$ are disjoint subsets of D (keep in mind that we are neglecting the fact two shocks $i_A \in$A and $i_A' \in$A can merge).
We now plug the above inequality into~\eqref{e:allinizio} and we recall that by assumption at time $t < \bar t$ estimates~\eqref{e:shock} hold. Owing to~\eqref{e:shocktot} we obtain
\begin{equation}
\label{e:estimatesua}
\begin{split}
 \sum_{i_A \in \mathrm{A}} \V_{i_A} ( \bar t^+)
 & \leq \left(\frac{1}{2} K \omega+2K \omega \varepsilon \right) \exp\left({\mathcal O(1)K \omega }\right)
<K\omega
\end{split}
\end{equation}
provided that $\varepsilon$ is sufficiently small, since owing to~\eqref{e:c:parametri2} $\omega = \varepsilon^3$. Note that~\eqref{e:estimatesua} 
 implies that~\eqref{e:shockA} holds for $t= \bar t$ provided that~\eqref{e:shockB} and~\eqref{e:shockC} hold for $t < \bar t$ and~\eqref{e:allinizio} holds at $t=0$. 

\step{we control $\sum_{i_B \in \mathrm{B}} \V_{i_B} $} The only ways $\sum_{i_B \in \mathrm{B}} \V_{i_B} $ can increase are the following:
\begin{enumerate}
\item If a  a shock $i_B \in $B merges with a shock $j \in$A$\cup$B of the same family. In this case $\sum_{i_B \in \mathrm{B}} \V_{i_B} $ does not increase owing to~\eqref{e:interae3}. 
\item If a  shock $i_B \in$B with strength $\V_{i_B}$ merges with a shock $j \in \cup_{m\in\N}$C$_{m}$ of the same family. 
\item If a  shock $i_B \in$B with strength $\V_{i_B}$ interacts with a shock $j$ of a different family. 
\end{enumerate}
By repeating the same argument we used in {\sc Step 2} we conclude that, if the initial estimate~\eqref{e:allinizio} holds and moreover~\eqref{e:shock} holds for  $t<\bar t$, then by~\eqref{e:shocktot}
\begin{equation}
 \begin{split}
 \label{e:estimateB}
\sum_{i_B \in \mathrm{B}} \V_{i_B} ( \bar t^+)
 & \leq 
 \left( \sum_{i_B \in \mathrm{B}} \V_{i_B} (t= 0) +
 \sum_{m=1}^\infty \sum_{i \in \mathrm{C}_m} \; \V_{i} 
 \right) 
 \exp\left({\mathcal O(1)\sum_{j \in A \cup D}\V_{j}}\right)
\\
 & \leq 
 \Big( \frac{1}{2} K \omega \varepsilon +
 \frac{(2K\omega)^{2}}{1-2K\omega} 
 \Big) \exp\left({\mathcal O(1) K \omega}\right) \leq K \omega \varepsilon \\
\end{split}
\end{equation}
provided that $\varepsilon$ is sufficiently small, due to~\eqref{e:c:parametri2}. 
Inequality~\eqref{e:estimateB} implies that~\eqref{e:shockB} holds for $t = \bar t$ provided that~\eqref{e:shock} hold for $t < \bar t$ and~\eqref{e:allinizio} holds at $t=0$. \\
\step{we conclude the proof} To control $\sum_{i \in \mathrm{C}_1} \V^i$ we firstly recall that a shock of group C$_1$ can be generated when a shock $j_1$ belonging to either group A or B interacts with a shock $j_2$ belonging to either group A or B. As pointed out before, the strength of the outgoing new shock is bounded by $\mathcal O(1) \V_{j_1} \V_{j_2}$. We denote by $\V^0_{i}$ the strength of the shock $i\in \mathrm{C}_1$ at the time when the shock $i$ is generated. We then have 
 \begin{align*}
 \sum_{i \in \mathrm{C}_1} \V^0_{i} & \leq 
 \sum_{j_1 \in A} 
 \sum_{j_2 \in A} \mathcal O(1) \V_{j_1} 
 \V_{j_2}+ 
 \sum_{j_1 \in A} 
 \sum_{j_2 \in B} \mathcal O(1) \V_{j_1} 
 \V_{j_2} + 
 \sum_{j_1 \in B} 
 \sum_{j_2 \in B} \mathcal O(1) \V_{j_1} 
 \V_{j_2} \\ & 
 \leq K^2 \Big( \omega^2 + \omega^2 \varepsilon + (\omega \varepsilon )^2 
 \Big)
 \leq 2 K^2 \omega^2. 
 \end{align*}
Next, recall that $\sum_{i \in \mathrm{C}_1} \V_{i}(t)$ can increase not only when a new shock is generated, but also: 
\begin{enumerate}
\item if a shock $i \in $C$_1$ with strength $\V_i$ merges with a shock $j \in$ C$_m$, $m >1$ of the same family. 
\item if a shock $i \in $C$_1$ with strength $\V_{i}$ interacts
with a shock $j$ of a different family. 
\end{enumerate}
By arguing as in {\sc Step 2}, we find that if~\eqref{e:shock} holds for $t < \bar t$ and~\eqref{e:allinizio} holds at $t=0$, then
\begin{align*}
 \sum_{i \in \mathrm{C}_1} \V_{}i (\bar t^+) 
 &\leq \left( \sum_{i \in \mathrm{C}_1} \V^0_{i} +
 \sum_{m=2}^\infty \sum_{i \in \mathrm{C}_m} \; \V_{i} \right)\exp\left({\mathcal O(1)\sum_{j \in A \cup D}\V_{j}}\right)\\
 & \leq \left( 2 K^2 \omega^2 +\frac{(2K\omega)^{3}}{1-2K\omega} \right)
 \exp\left({\mathcal O(1)K \omega }\right)
 \leq (2 K \omega)^2 \phantom{\int_R}
\end{align*}
provided that $\varepsilon$ is sufficiently small, since owing to~\eqref{e:c:parametri2} $\omega = \varepsilon^3$. 
We have thus established~\eqref{e:shockC} for all $t>0$ when $m=1$. The case when $m>1$ can be handled in an entirely similar way. This concludes the proof of Lemma~\ref{l:quantitative}. 
\end{proof}
\subsection{Wave front-tracking approximation: shock generation analysis}
\label{ss:shockgeneration}
In this paragraph we finally show that the in wave-front tracking approximation one can recognize a wave pattern like the one of the solution of the Cauchy problem with initial datum $V$, see Figure~\ref{F:V}. In particular, 
in \S~\ref{sss:smalltime} we establish the generation of six ``big shocks'': Lemma~\ref{l:collapse} establishes the formation of two 1-shocks and two 3-shocks, while Lemma~\ref{l:collapse2} established the formation of two 2-shocks, which are moreover approaching. In~\S~\ref{sss:wpgeneration} we conclude the analysis of the wave pattern generation. 
\subsubsection{Shock formation: small times}
\label{sss:smalltime}
We recall that the interval $\mathfrak R^3_\ell$ is defined by formula~\eqref{e:regions} and we establish the following lemma. 
\begin{lemma}
\label{l:collapse}
 By the time $t = 6/5$, some of (or all of) the 3-shocks of group A generated at time $t=0$ in the interval $\mathfrak R^3_\ell$ merge into a single 3-shock with strength greater than $\omega \sqrt{\varepsilon}/2$.
\end{lemma}
\begin{proof}
We  first describe the idea underpinning our argument. We term $j_\ell$ and $j_r$ the 3-shocks that are generated at $t=0$ at the left and the right extrema of the interval $\mathfrak R^3_\ell$, respectively. In {\sc Step 3} below we show that at $t=0$ these two 3-shocks are approaching. We then track the the evolution of $j_\ell$ and $j_r$ on the time interval $]0, 6/5[$  and we point out that there are only two possibilities:
\begin{itemize} 
\item The strength of both $j_\ell$ and $j_r$ remains smaller than $\omega \sqrt{\varepsilon}$. In this case we show in {\sc Step 5} below that $j_\ell$ and $j_r$ keep approaching and they merge by time $t =6/5$. We also show that this implies the creation of a 3-shock with strength at least $\mathcal O(1) \omega $.
\item The strength of either $j_\ell$ or $j_r$ surpasses  $\omega \sqrt{\varepsilon}$ at some time $\bar t \in ]0, 6/5[$: just to fix the ideas, let us assume that it is the strength of $j_\ell$. In {\sc Step 6} below we show that this implies that the strength of the   $j_\ell$ remains bigger than $\omega \sqrt{\varepsilon}/2$ on the whole interval $]\bar t, 6/5]$. 
\end{itemize}
The technical details are organized as follows. \\
\firststep\step{we point out that in the time interval $]0, 6/5[$ the waves of group A generated in $\mathfrak R^3_\ell$ can only interact among themselves and with the waves of group $D$ (see~\eqref{e:D} for the definition of group D)} In other words, they cannot interact with shocks of group A generated in other intervals. 

To see this, we proceed as follows. We recall definition~\eqref{e:regions} and that the shocks of group A are only generated in the intervals $\mathfrak R^3_\ell$, $\mathfrak R^2_\ell$, $\mathfrak R^1_\ell$, 
$\mathfrak R^3_r$, $\mathfrak R^2_r$ and $\mathfrak R^1_r$. The closest interval to $\mathfrak R^3_\ell$ is $\mathfrak R^2_\ell$ and the distance between the right extreme of $\mathfrak R^3_\ell$ and the left extreme of $\mathfrak R^2_\ell$ is 
\begin{equation}
\label{e:separiamo} 
 - q + \mathfrak q - \lambda_2 (\underline U' ) + \lambda_3 (\underline U') \ge 3 - \lambda_2 (\underline U' ) + \lambda_3 (\underline U') \ge 6.
\end{equation}
To establish the last inequality, we used~\eqref{e:stimacizero} and the explicit expression of the eigenvalues, see~\eqref{e:eigenvaluese}. This implies that, if the constant $\varepsilon$ in the statement of Proposition~\ref{p:perturbation} is sufficiently small, then in the time interval $]0, 6/5[$ the 3-shocks generated at $t=0$ in $\mathfrak R^3_\ell$ cannot interact with the 2-shocks generated at $t=0$ in $\mathfrak R^2_\ell$. 

\step{we focus on the time $t=0$ and we introduce some notation}
Let $x_\ell$ and $x_r$ be the right and left extrema of $\mathfrak R^3_\ell$, namely $x_\ell = -\mathfrak q - \lambda_3 ( U_{I})$ and 
$x_r = - \mathfrak q - \lambda_3 (\underline U')$. We recall that by the mesh definition discussed in~\S~\ref{sss:mesh} $x_\ell$ and $x_r$ are both points of discontinuity for $U^\nu_0$ . Next, we define the states $U_\ell^-(0)$, 
$U_\ell^+(0)$, $U_r^-(0)$, $U_r^+(0)$ by setting 
\begin{subequations}
\label{e:definizioneuelle}
\begin{align}
U_\ell^-(0) : = \lim_{x \uparrow x_\ell} U^\nu_0 (x), \qquad 
&U_\ell^+(0) : = \lim_{x \downarrow  x_\ell} U^\nu_0 (x) \\ 
U_r^-(0) : = \lim_{x \uparrow  x_r} U^\nu_0 (x), \qquad 
&U_r^+(0) : = \lim_{x \downarrow  x_r} U^\nu_0 (x) 
\end{align}
\end{subequations}
We denote by $j_\ell$ and $j_r$ the 3-shocks of group $A$ generated at $t=0$ at $x_\ell$ and $x_r$, respectively. We also denote by $\mathrm{speed}_{j_\ell} (0)$ and $\mathrm{speed}_{j_r} (0)$ their speed at $t=0$. \\
\step{we control from below the initial difference in speed of $j_{\ell}$ and $j_{r}$} More precisely, we establish the following estimate: 
\begin{equation}
\label{e:velocitat0}
       \mathrm{speed}_{j_\ell} (0) - \mathrm{speed}_{j_r} (0)
       \geq \frac{11}{12} \, \mathrm{length} (\mathfrak R_\ell^3)=
       \frac{11}{12} \Big( \lambda_3 (U_I) - 
        \lambda_3 (\underline U') \Big).
\end{equation}
To this end, we point out that owing to~\eqref{e:derivata}, 
$$
 | U_\ell^-(0) - U_\ell^+ (0) | \leq \mathcal O(1) \frac{ h_\nu}{ \eta}, 
\qquad 
| U_r^-(0) - U_r^+ (0) | \leq \mathcal O(1)\frac{ h_\nu}{ \eta}. 
$$
The explicit expression~\eqref{e:eigenvaluese}  of $\lambda_3$ implies that
$|\nabla \lambda_3 | \leq \mathcal O(1) \eta$ and hence by using the above inequalities we arrive at
\begin{equation}
\label{e:vicinivicini}
    \left| \mathrm{speed}_{j_\ell} (0) - \lambda_3 \big( U_\ell^-(0) \big) \right|
     \leq \mathcal O(1) h_\nu, 
    \qquad
    \left| \mathrm{speed}_{j_r} (0) - \lambda_3 \big( U_r^+(0) \big) \right|
    \leq \mathcal O(1) h_\nu. 
\end{equation}
Next, we use~\eqref{e:palla},~\eqref{e:vartotpsi},~\eqref{e:veperturbation},~\eqref{e:unuzero} and the equalities $V(x_\ell)= U_I$ and $V(x_r) = \underline U'$
to get 
\begin{equation}
\label{e:sopra}
   | U_\ell^-(0) -  U_I |
   \leq \mathcal O(1) \left(
\varepsilon \omega  + r + 
    \frac{h_\nu}{\eta} \right),
    \qquad | U_r^+(0) - \underline U' | \leq 
   \mathcal O(1) \big( 
\varepsilon \omega   + r   \big). 
\end{equation}
Exploiting again the equality
$|\nabla \lambda_3 | \leq \mathcal O(1) \eta$, we get that~\eqref{e:sopra} implies   
\begin{equation*}
\begin{split}
    &
    \left| \lambda_3 \big( U_\ell^-(0) \big) -\lambda_3 (U_I)
    \right|
     \leq  \mathcal O(1) 
     \left( \varepsilon \omega\eta    + r \eta + 
  h_\nu    \right) 
    \\ & 
   \left| \lambda_3 \big( U_r^+(0) \big)  - \lambda_3 (\underline U_I)   \right|
   \leq  \mathcal O(1) 
     \left( \varepsilon \omega\eta  + r \eta  \right) . \\
\end{split}
\end{equation*}
By  plugging the above estimate into~\eqref{e:vicinivicini} we 
arrive at 
\begin{equation}
\label{e:allafine}
      \mathrm{speed}_{j_\ell} (0) - \mathrm{speed}_{j_r} (0) \geq 
      \lambda_3 (U_I) - 
        \lambda_3 (\underline U') - \mathcal O(1) 
         \left(\varepsilon \omega \eta   + r \eta + 
 h_\nu    \right) 
\end{equation}
Next, we point out that the equality $|\nabla \lambda_3 | \leq \mathcal O(1) \eta$ implies that
\begin{equation}
\label{e:quantovale}
    \lambda_3 (U_I) - 
        \lambda_3 (\underline U') 
        = \mathcal O(1) \omega \eta,
\end{equation}
because by Lemma~\ref{l:solvecw} the parameter $\tau$ in~\eqref{e:statedef} is of order $\omega$.
We eventually obtain~\eqref{e:velocitat0} by observing that terms in the last parenthesis in~\eqref{e:allafine} are of lower order than $\omega \eta$: this follows by recalling~\eqref{e:c:parametri3},  and the fact that $h_\nu \downarrow 0$ when $\nu \downarrow 0$.  \\
\step{we consider the evolution of the shocks $j_\ell$ and $j_r$ in the time interval $]0, 6/5[$}
Let $U_\ell^+ (t)$ and $U_\ell^-(t)$, $U_r^+ (t)$ and $U_r^-(t)$ be the left and right state at time $t$ of $j_\ell$ and $j_r$, respectively. Note that the above functions are piecewise constant: to define their pointwise values, in the following we choose their right continuous representative. 
One of the following two cases must occur:
\begin{enumerate}
\item we have
\begin{equation}
\label{e:shockpiccoli}
 |U_\ell^+ (t) - U_\ell^-(t) | < \omega \sqrt{\varepsilon}, 
 \quad |U_r^+ (t) - U_r^-(t) | < \omega \sqrt{\varepsilon}
 \quad \text{for every $t \in ]0, 6/5[$}.
\end{equation}
 We handle this case  in {\sc Step 5} below. 
\item There is $\bar t \in ]0, 6/5[$ such that 
\begin{equation}
\label{e:shockgrandi}
\text{either} \quad |U_\ell^+ (\bar t) - U_\ell^-(\bar t) | \ge \omega \sqrt{\varepsilon}
\quad \text{or} \quad
 |U_r^+ (\bar t) - U_r^-(\bar t) | \ge \omega \sqrt{\varepsilon}.
\end{equation}
We handle this case  in {\sc Step 5} below. 
\end{enumerate}
\step{we conclude the proof of the lemma under the assumption that~\eqref{e:shockpiccoli} holds} 

We recall from {\sc Step 1} that in the time interval $t \in ]0, 6/5[$ both $j_\ell$ and $j_r$ can either merge with other 3-shocks of group A or interact with 1-, 2- and 3-shocks of group D~\eqref{e:D}, but they cannot interact with other 1- or 2-shocks of group A. This implies that $U^-_\ell(t)$ and $U^+_r(t)$ can only change  owing to the interaction with some shock of group D: we recall~\eqref{e:shockD} and we conclude that 
\begin{equation}
\label{e:shockpiccoli1}
 |U^-_\ell (t) - U^-_\ell (0)| + 
 |U^+_r (t) - U^+_r (0)| \leq 
 \mathcal O(1) \omega \varepsilon. 
\end{equation} 
Next, we proceed as in {\sc Step 3} and by combining~\eqref{e:shockpiccoli} with~\eqref{e:shockpiccoli1} we conclude that 
\begin{equation}
\label{e:shockpiccoli3}
 \mathrm{speed}_{j_\ell} (t) - \mathrm{speed}_{j_r} (t)
 \geq \frac{5}{6} \, \mathrm{length} (\mathfrak R_\ell^3)=
 \frac{5}{6} \Big( \lambda_3 (U_I) - 
 \lambda_3 (\underline U') \Big)\quad 
 \text{for every $t \in ]0, 6/5[$}
\end{equation}
provided that $\varepsilon$ (and hence $\omega$, owing to~\eqref{e:c:parametri2}) are sufficiently small. In the previous expression, we denote by $\mathrm{speed}_{j_\ell} (t)$
and $\mathrm{speed}_{j_r} (t)$ the speed of $j_\ell$ and $j_r$ at time $t$. 
 Note that~\eqref{e:shockpiccoli3} implies that by the time $t =6/5$ the shocks $j_\ell$ and $j_r$ merge. By construction, this implies that all the 3-shocks of group A generated at $t=0$ in $\mathfrak R^3_\ell$ merge by time $t=6/5$. In the following,  we denote  by A$^3_\ell$ the group of the 3-shocks of group A generated at $t=0$ in $\mathfrak R^3_\ell$. We follow the same argument as in {\sc Step 2} of Lemma~\ref{l:quantitative} and we use the inequality
\[
     \prod_{j \in A\cup D} \big( 1 - \mathcal O(1) \mathcal V_j \big)
     \ge 1 - \mathcal O(1) \sum_{j \in A\cup D} V_j,
\]
which is a consequence of the elementary inequality $(1-x) (1-y) \ge 1 - (x+y)$ if $x, y \ge 0$. We conclude that the total strength of the shocks in A$^3_\ell$ can bounded from below, more precisely by recalling~\eqref{e:shockD} and the analysis in \S~\ref{sss:id:cw} we have 
\[
 \sum_{i \in \mathrm{A}^3_\ell} \V_i (t) 
 \ge \Big( 1 - \mathcal O(1)\sum_{i \in D} \mathcal V_i \Big)
 \sum_{i \in \mathrm{A}^3_\ell} \V_i (t=0) \ge
 \mathcal O(1) \omega. 
\]
We eventually obtain that by time $t=6/5$ the shocks of group A$^3_\ell$ merge into a single shock with strength $\mathcal O(1) \omega$. \\
\step{we conclude the proof of the lemma under the assumption that~\eqref{e:shockgrandi} holds}

First, we point out that~\eqref{e:shockgrandi} implies that at $t= \bar t$ part of the waves of group $A^3_\ell$ have merged into a shock 
of strength $\omega \sqrt{\varepsilon}$. Hence, we are left to prove that this shock ``survives'' with a sufficiently large strength up to time $t=3/2$. To this end, we point out that for $t > \bar t$ this 
shock can merge with other 3-shocks of group A$^3_\ell$ and hence increase its strength. Also, it can interact with other shocks of group $D$: however, by following the same argument as in {\sc Step 2} of Lemma~\ref{l:quantitative} and by recalling~\eqref{e:shockD}
 the strength of the shock is bounded from below by 
$$
  \Big( 1 -  \mathcal O(1)\sum_{i \in D} \mathcal V_i \Big) \omega \sqrt{\varepsilon} 
  \ge \omega \sqrt{\varepsilon} /2, 
$$
provided that $\varepsilon$ is sufficiently small. 
This concludes the proof of Lemma~\ref{l:collapse}. 
\end{proof}
Note that by repeating the above proof we obtain the an analogous of Lemma~\ref{l:collapse} holds for the 3-shocks of group A generated at time $t=0$ in the interval $\mathfrak R^3_r$ and the 1-shocks of group A generated at time $t=0$ in the intervals $\mathfrak R^1_\ell$ and $\mathfrak R^1_r$. In the case of 2-shocks we have a stronger result. 
\begin{lemma}
\label{l:collapse2} Let $\widetilde T$ be the same constant as in~\eqref{e:T}. 
The following conclusions hold true:
 \begin{enumerate}
 \item\label{item:primocollapse2}
 By the time $t = 6/5$, all the 2-shocks of group A generated at time $t=0$ in the interval $\mathfrak R^2_\ell$ merge into a single 2-shock $J^2_\ell$ having strength greater or equal than $\mathcal O(1) \omega$. The same holds for the 2-shocks of group A generated at time $t=0$ in the interval $\mathfrak R^2_r$, let $J^2_r$ be the resulting shock. 
 \item\label{item:secondocollapse2} The 2-shocks $J^2_\ell$ and $J^2_r$ are approaching and they merge by the time 
$
 t = 2 \widetilde T. 
$ 
\end{enumerate}
\end{lemma}
\begin{proof}
The proof of \ref{item:primocollapse2}) is organized in two steps. We only discuss the 2-shocks of group A generated at time $t=0$ in the interval $\mathfrak R^2_\ell$, the argument for the 2-shocks generated in $\mathfrak R^2_r$ is completely analogous. 

\firststep\step{we discuss the situation at time $t=0$} We recall that, owing to Lemma~\ref{L:v}, the speed of a 2-shock $j$ between $U^-$ and $U^+$ is 
\begin{equation}
\label{e:speedgei}
 \mathrm{speed}_j = {v^- + v^+}. 
\end{equation}
Next, we fix $x_i^\nu$, $x^\nu_{i+1} \in \mathfrak R^2_\ell$. We denote by $j_i$ and $j_{i+i}$ the 2-shocks generated at $t=0$ at the points $x=x_i^\nu$ and $x^\nu_{i+1}$, respectively, and by $\mathrm{speed}_{j_i} (0)$ their speed at $t=0$. 
Let $v'_{U_{0}}$, $v'_{V}$ be the first derivative of the second components of $U_0$ and $V$, respectively. By combining~\eqref{e:VI},~\eqref{e:psi},~\eqref{e:palla} and~\eqref{e:veperturbation} we have
$$
 \left|v_{U_{0}}' (x) + \frac{1}{2}\right| =\left|v_{U_{0}}' (x) -v_{V}'\right| \leq \mathcal O(1) \zeta_w + \mathcal O(1)r , 
 \quad \text{for every $x \in \mathfrak R^2_\ell$}
$$
By using the relations $\zeta_w = \varepsilon$,~$r < \varepsilon^2$,
~\eqref{e:meshsize} and~\eqref{e:speedgei}, we get that, when $\varepsilon$ is small enough, the last inequality brings us to 
\begin{equation}
\label{e:siavvicinano}
\begin{split} 
 \mathrm{speed}_{j_i} (0) - \mathrm{speed}_{j_{i+1}} (0)
 & \ge  \left(\frac{1}{2}- \mathcal O(1)
 \varepsilon \right) (x_{i-1}^\nu - x^\nu_{i+1})
 \ge  \left(\frac{1}{2}- \mathcal O(1)
 \varepsilon \right) 2 (1-\varepsilon) h_\nu \\ &
  \ge 
 \frac{5}{6} (x_{i}^\nu - x^\nu_{i+1}). 
 \end{split}
\end{equation}

\step{we show that the shocks $j_i$ and $j_{i+1}$ merge in the time interval $]0, 6/5[$} By the arbitrariness of $x^\nu_i$ and $x^\nu_{i+1}$ this establishes~
\ref{item:primocollapse2}).

We recall that the speed of a 2-shock does not change at the interaction with a 1- or a 3-shock. Hence, the speed of $j_i$ and $j_{i+1}$ can only change when they merge with a 2-shock. Three cases can occur:
\begin{itemize} 
\item the shocks $j_i $ and $j_{i+1} $ merge: this proves the claim of the present step.
\item $j_i $ merges with a 2-shock $\ell$ on the left of $j_i$: this implies that the speed of $j_i$ increases.
\item $j_{i+1} $ merges with a 2-shock $\ell$ on the right of $j_{i+1}$: this implies that  the speed of $j_{i+1}$ decreases.
\end{itemize}
If only the last two cases occur, by recalling~\eqref{e:siavvicinano} we conclude that 
$$
 \mathrm{speed}_{j_i} (t) - \mathrm{speed}_{j_{i+1}} (t)
 \ge \mathrm{speed}_{j_i} (0) - \mathrm{speed}_{j_{i+1}} (0)
 >\frac{5}{6} (x_i^\nu - x^\nu_{i+1}) \quad \text{for every $t \in ]0, 6/5[$.}
$$ 
This implies that by the time $6/5$ the shocks $j_i$ and $j_{i+1}$ merge, and hence concludes the proof of \ref{item:primocollapse2}).

\step{We are now left with establishing \ref{item:secondocollapse2}), namely 
proving that the shocks $J^2_\ell$ and $J^2_r$ (defined as in the statement of Lemma~\ref{l:collapse2}) merge by the time $2\widetilde T$} To this end, we recall the explicit expression of $\mathfrak R^2_\ell$ and $\mathfrak R^2_r$~\eqref{e:regions}:
\begin{align*}
 &\mathfrak R^2_\ell: = ] -q - \lambda_2 (\underline U'),
 -q - \lambda_2 (\underline U'') [,
 &&
 \mathfrak R^3_r: =  ] q - \lambda_2 (\underline U^\ast),
 q - \lambda_2 (\underline U^{\ast \ast}) [.
\end{align*}
We also introduce the following notation:  we term
\begin{itemize}
\item $v^-_\ell (t)$ the  second component of the left state of the 2-shock created at $t=0$ at the left extreme of $\mathfrak R^2_\ell$,
\item $v^+_\ell (t)$ the second component of the  right state of the 2-shock created at $t=0$ at the right extreme of $\mathfrak R^2_r$,
\item$\mathrm{speed}_{ \ell}(t)$ the speed of the 2-shock arising at $t=0$ at the left extreme of $\mathfrak R^2_\ell$.
\end{itemize}
The functions $v^-_r (t)$, $v^+_r (t)$ and $\mathrm{speed}_{r}(t)$ are similarly defined by considering $\mathfrak R^2_r$.
Note that $v^-_\ell (t)$ and $v^+_\ell(t)$ are the left, right state and speed of $J^2_{r}$ for $t > 6/5$, because by~\eqref{item:primocollapse2} all the  2-shocks generated at $t=0$ in $\mathfrak R^2_\ell$ merge by the time $t=6/5$. By using an analogous argument we prove that $v^-_r (t)$, $v^+_r(t)$ and $\mathrm{speed}_{r}(t)$ are the left state, the right state and the speed of $J^2_r$, respectively.

By combining~\eqref{e:VI},~\eqref{e:palla},~\eqref{e:vartotpsi} and~\eqref{e:veperturbation} with~\eqref{e:c:parametri5} and~\eqref{e:unuzero}
we infer that
\begin{equation}
\label{e:atzero}
 \Big| \left\{ [ v^-_\ell(0) + v^+_\ell(0)] - [v^-_r(0) + v^+_r(0)] \right\}-
 [v_{I} - v_{I\!I\!I}] \Big| 
 \leq \mathcal O(1) (\omega \varepsilon +h_{\nu}). 
\end{equation}
Next, we point out that $v^-_\ell$ $v^+_\ell$ $v^-_r$ $v^+_r$ can only vary with respect to $t$ owing to the interactions with 2-shocks of group B. Owing to~\eqref{e:shockB}, this implies that 
\begin{equation}
\label{e:accerrsor}
 |v^-_\ell (t) - v^-_\ell(0) | +
 |v^+_\ell (t) - v^+_\ell(0) | +
 |v^-_r (t) - v^-_r (0) | +
 |v^+_r (t) - v^+_r (0) | \leq \mathcal O(1) \omega \varepsilon. 
\end{equation}By combining~\eqref{e:atzero} and~\eqref{e:accerrsor}
we infer 
\begin{equation}
\label{e:atmaggiore} 
\begin{split}
 \mathrm{speed}_{ \ell}(t) - 
 \mathrm{speed}_{r}(t) 
 &\geq 
 [ v^-_\ell(t) + v^+_\ell(t)] - [v^-_r(t) + v^+_r(t)]
 \\&\geq
 [v_{I} - v_{I\!I\!I}] -\mathcal O(1) (\omega \varepsilon+h_{\nu}) .
 \end{split}
\end{equation}
By using~\eqref{e:atmaggiore} and the definitions~\eqref{e:T},~\eqref{e:regions} of $\widetilde T$ and $\mathfrak R^2_r$, we realize that the shocks $J^2_\ell$ and $J^2_r $ merge by time
\begin{align*}
t \leq  \frac{\left[q - \lambda_2 (\underline U^{\ast \ast}) \right] -\left[-q - \lambda_2 (\underline U')\right]}{\sup_t [ \mathrm{speed}_{\ell}(t) - 
 \mathrm{speed}_{r}(t)]}
& \leq
\frac{ 2q+\mathcal O(1) \omega}{v_{I}-v_{I\!I\!I}-\mathcal O(1)(\varepsilon\omega+h_{\nu})}
=
\frac{ 2q+\mathcal O(1) \omega}{{2q}/{\widetilde T}-\mathcal O(1)(\varepsilon\omega+h_{\nu})}
\\&=
\frac{ 1+ \mathcal O(1) \omega}{1-\mathcal O(1) (\varepsilon+\widetilde T h_{\nu})}\cdot \widetilde T .
\end{align*}
To get the last equality we have used the equalities $ \widetilde T=\mathcal O(1)\omega^{-1}$ and $q=20$. Since $h_\nu \to 0^+$, this implies that, if $\omega = \varepsilon^3$ is sufficiently small, then $J^2_\ell$ and $J^2_r$ merge by the time $t = 2 \widetilde T$. This concludes the proof of Lemma~\ref{l:collapse2}.   
\end{proof}
\subsubsection{Shock formation: wave pattern generation}
\label{sss:wpgeneration}
By relying on the analysis at the previous paragraph, at $t=6/5$ the wave-front tracking approximation $U^\nu (t, \cdot)$ contains at least six ``big shocks''. Going from the left to the right, i.e.~as $x$ increases, we encounter: a 3-shock with strength at least $ \omega \sqrt \varepsilon /2$ (see Lemma~\ref{l:collapse}), a 2-shock with strength greater or equal than $\mathcal O(1) \omega$, a 1-shock with strength greater or equal than $\omega \sqrt \varepsilon /2$, and then again 3-shock with strength at least $ \omega \sqrt \varepsilon /2$, a 2-shock with strength greater or equal than $\mathcal O(1) \omega$, a 1-shock with strength greater or equal than $\omega \sqrt \varepsilon /2$. Note that the two 2-shocks are approaching and they meet by time $t = 2 \widetilde T$. Also, the six big shocks do not interact on the time interval $]0, 6/5[$ because the generation regions $\mathfrak R^3_\ell$, $\mathfrak R^2_\ell$, $\mathfrak R^1_\ell$, $\mathfrak R^3_r$, $\mathfrak R^2_r$, $\mathfrak R^1_r$ are sufficiently separated, see~\eqref{e:qs} and~\eqref{e:regions}. Besides those six ``big shocks'' there are in general other waves, which however are all shocks by the analysis in~\S~\ref{Ss:qualitativeanalysis}. 
\subsection{Conclusion of the proof}
\label{ss:conclusion}
In this paragraph we conclude the proof of Proposition~\ref{p:perturbation}. In~\S~\ref{sss:nonphysical} we take into account the 
 presence of non-physical waves in the wave-front tracking approximation. In~\S~\ref{sss:boundfrombelow} we establish a bound from below on the number of shocks in the wave-front tracking approximation. Finally, in~\S~\ref{sss:theend} we complete the proof of Proposition~\ref{p:perturbation}. 
\subsubsection{Non-physical waves}
\label{sss:nonphysical}
In this paragraph we take into account the presence of non-physical waves. We firstly recall some facts about the \emph{simplified Riemann solver} and we refer to~\cite[\S7.2]{Bre} for a complete discussion. 

First, one chooses a threshold parameter $\mu_\nu>0$. We discuss the choice of $\mu_\nu$ later in this paragraph, however we point out that $\mu_\nu \to 0^+$ as $\nu \to 0^+$. The \emph{accurate Riemann solver} is used to solve the interaction of a wave $\alpha$ of strength $\V_{\alpha}$ with a wave $\beta$ of strength $\V_{\beta}$ in the wave front-tracking approximation if the product of the strengths of the incoming waves satisfies 
\begin{equation}
\label{e:nonphysical}
 \V_{\alpha} \cdot \V_{\beta} \ge \mu_\nu. 
\end{equation}
If the above condition is violated, we use the \emph{simplified Riemann solver}, which is defined at~\cite[p.131]{Bre} and involves the introduction of so-called \emph{non-physical waves}. 
Non-physical waves travel at a speed faster than any other wave and the simplified Riemann solver is defined in such a way that their interaction with the other waves has a minimal effect. 
To simplify the exposition, here we do not recall all the technical details and we only discuss the properties of the simplified Riemann solver for the Baiti-Jenssen system~\eqref{e:pertSyst} that we need in the following. These properties are either a direct consequences of the definition of simplified Riemann solver or can be straightforwardly recovered by combining the definition with the features of the Baiti-Jenssen system discussed in \S~\ref{SS:bjcon}. 
\begin{enumerate}
\item If the incoming waves are a 1-shock and a 3-shock, then the simplified Riemann solver coincides with the accurate Riemann solver. 
\item\label{item:2nonphyswaves} If we use the simplified Riemann solver to solve the interaction between a 2-shock and a 1-shock (respectively a 3-shock), then the outgoing waves are a 2-shock, a 1-shock (respectively a 3-shock) and a non-physical wave. The value of the $v$ component is constant across the non-physical wave. 
\item If the incoming waves are both 1-shocks then the simplified Riemann solver coincides with the accurate Riemann solver. The same happens if the incoming waves are both 3-shocks. 
\item\label{item:4nonphyswaves} If the incoming waves are both 2-shocks, then the speed of the outgoing shock is the same in the simplified and in the accurate Riemann solver. Also, the value of the $v$ component is constant across the outgoing non-physical wave. 
\item\label{item:5nonphyswaves} By combing all the above features we conclude that the strength of the $v$ component is \emph{always} constant across non-physical waves. 
\item\label{item:6nonphyswaves} If a non-physical wave interact with a 2-shock, then the speed of the 2-shock does not change. 
\item The strength of each non-physical wave is at most $\mu_\nu$. Also, owing to the analysis in~\cite[p.142]{Bre} we can choose $\mu_\nu$ in such a way that the total strength of non-physical waves satisfies 
\begin{equation}
\label{e:totalstrengthnp}
 \text{total strength non physical waves $\leq \nu$},
\end{equation}
where $\nu$ is our approximation parameter. Owing again to the 
analysis in~\cite[p.142]{Bre}, this is consistent with the requirement that 
$\mu_\nu \to 0^+$ as $\nu \to 0^+$. 
\end{enumerate}
We now discuss how the presence of the non-physical waves affect the analysis at the previous paragraphs. First, we point out that it does not affect at all the discussion on the initial datum in~\S~\ref{ss:wft:id} because by definition the simplified Riemann solver is only used at time $t>0$. Next, we point out that the use of the simplified Riemann solver forces the total number of waves to be finite. In particular, there actually are fewer waves of groups C$_1,\dots,$ C$_m$ than those considered in~\S~\ref{Ss:qualitativeanalysis}. Lemma~\ref{l:quantitative} does not change if we take into account the presence of non-physical waves, provided that we say that if a group C$_m$ is empty, then the the total strength of its waves is $0$. The reason why Lemma~\ref{l:quantitative} does not change is because the proof is based on interaction estimates on the strength of waves and by definition the interaction with a non-physical wave does not change the strength of a shock. Also Lemma~\ref{l:collapse} does not change: indeed, the proof is based on the quantitative estimates given by Lemma~\ref{l:quantitative}, which are still valid. The further perturbation provided by the non-physical waves is arbitrarily small owing to~\eqref{e:totalstrengthnp} and hence does not affect the proof. Finally, Lemma~\ref{l:collapse2} does not change because the proof is based on estimates that, as a matter of fact, involve only the \emph{second} component (i.e., the component $v$) of the wave front- tracking approximation. Owing to properties \ref{item:2nonphyswaves}), \ref{item:4nonphyswaves}), \ref{item:5nonphyswaves}) and \ref{item:6nonphyswaves}) of the non-physical waves of the Baiti-Jenseen systems, non-physical waves have basically no effect on the $v$ component of the wave front-tracking approximation and hence the proof of Lemma~\ref{l:collapse2} is still valid if we take into account non-physical waves. 
\subsubsection{A bound from below on the number of shocks}
\label{sss:boundfrombelow}
This paragraph aims at establishing Lemma~\ref{l:almeno} below. In the statement, $J^2_\ell$ and $J^2_r$ are the same as in the statement of Lemma~\ref{l:collapse2} and we denote by $[\cdot ]$ the entire part. Also, 
$\mu_\nu$ is the threshold to determine whether we use the accurate or the simplified Riemann solver, see~\eqref{e:nonphysical}. 
\begin{lemma}
\label{l:almeno}
Fix a threshold $\theta > \mu_\nu / \omega^2$. In the bounded set ${(t, x) \in ]-\rho, \rho[ \times ]0, 2 \widetilde T}[$, the wave-front tracking approximation $U^\nu$ admits at least
\begin{equation}
\label{e:almeno}
 n_\theta: = \left[ \log_{\omega /2} \left( \frac{ \mathcal O(1)
 \theta}{\sqrt{\varepsilon} } \right) \right]
\end{equation}
shocks $j$ such that strength $\V_{j}$ of $j$ satisfies
\begin{equation}
\label{e:cosalmeno}
\V_{j}\ge \theta. 
\end{equation}
\end{lemma}
\begin{proof}
If there were only the six ``big shocks'' mentioned in~\S~\ref{sss:wpgeneration}, then the wave pattern would be qualitatively like the one represented in Figure~\ref{F:V}.
To understand the impact of the other waves and to establish~\eqref{e:almeno} we track the evolution of the left 3-``big shock'' $J^3_\ell$, which has strength at least $\omega \sqrt{\varepsilon}/4$ when it interacts with the left 2-``big shock'' $J^2_\ell$. We recall that the strength of $J^2_\ell$ is $\mathcal O(1) \omega$ and we use estimate (7.31) in~\cite[p.133]{Bre}: we conclude that after this interaction the strength of $J^3_\ell$ is at least 
$$
 \frac{ \omega \sqrt{\varepsilon}}{4} - \mathcal O(1) \sqrt{\varepsilon} \omega^2 \ge 
 \frac{ \omega \sqrt{\varepsilon}}{8}. 
$$
After this interaction, the shock $J^3_\ell$ moves towards the right 2-``big shock'' $J^2_r$. Before interacting with $J^2_r$, however, $J^3_\ell$ can interact with 1- and 3-shocks and with 2-shocks different than $J^2_\ell$ and $J^2_r$. The interaction with a 3-shock increases the strength of $J^3_\ell$ because the shock merges with $J^3_\ell$. The interaction with a 1-shock does not affect the strength of $J^3_\ell$. We are left to consider the interactions with 2-shocks different than $J^2_\ell$ and $J^2_r$. We recall that 2-shocks are only generated
at $t=0$ and that, owing to Lemma~\ref{l:collapse2}, all the 2-shocks generated at $t=0$ in $\mathfrak R^2_\ell$ and $\mathfrak R^2_r$ have merged by the time $t=3/2$ to generate $J^2_\ell$ and $J^2_r$, respectively. Hence, what we are left to consider are the interactions of $J^3_\ell$ with the 2-shocks that are not generated at $t=0$ in $\mathfrak R^2_\ell \cup \mathfrak R^2_r$. Note that all these 2-shocks belong to group $B$. We recall ~\eqref{e:shockB} and the interaction estimate (7.31) in~\cite[p.133]{Bre} and we infer that after the interaction with all these 2-shocks the strength of 
$J^3_\ell$ is at least 
\begin{equation}
\label{e:riflessione}
 \frac{ \omega \sqrt{\varepsilon}}{{{}8}} - \mathcal O(1) \frac{ \omega \sqrt{\varepsilon}}{{{}8}} 
 \sum_{i_B \in B} \V_{i_B} \ge \frac{ \omega \sqrt{\varepsilon}}{{{}8}} - \mathcal O(1)
 \omega^2 \varepsilon^{3/2} \ge \frac{ \omega \sqrt{\varepsilon}}{{{}16}}.
\end{equation}
Owing to Lemma~\ref{L:ie}, when $J^3_\ell$ interacts with $J^2_r$, then a 1-shock is created: by combining~\eqref{e:riflessione} with the fact that the strength of $J^2_r$ is $\mathcal O(1) \omega$, we get that the strength of this 1-shock is at least $\mathcal O(1) \sqrt{\varepsilon} \omega^2/ {{}16}$. Also, this 1-shock moves towards $J^2_\ell$, but before reaching $J^2_\ell$ may interact with 1-, 2- and 3-shocks. By arguing as before, we infer that when it reaches $J^2_\ell$ its strength is at least $\mathcal O(1) \sqrt{\varepsilon} \omega^2/ {{}32}$. When this 1-shock interacts with 
$J^2_\ell$, a 3-shock with strength at least $\mathcal O(1) \sqrt{\varepsilon} \omega^3/ {{}32}$ is created. We repeat this argument as long as the strength of the reflected 1- or 3-shock $j$ satisfies~\eqref{e:cosalmeno}, namely we can repeat it a number $n_\theta$ of times, where $n_\theta$ satisfies 
$$
 \mathcal O(1) \sqrt{\varepsilon} \omega \left( \frac{\omega}{2} \right)^{n_\theta}
 \ge \theta,
$$
This implies~\eqref{e:almeno}. We are left to justify the fact that we used the accurate and not the simplified Riemann solver. Note that, owing to the inequality $\theta \ge \mu_\nu / \omega^2$ and since the strengths $\V_{J^2_r}$ of $J^{2}_{r}$ and $\V_{J^2_\ell}$ of $J^{2}_{\ell}$ are equal to $\mathcal O(1) \omega$, if~\eqref{e:cosalmeno} holds then 
$$
 \V_{j} \cdot \max \big\{\V_{J^2_r}, \ \V_{J^2_\ell} 
 \big\} \ge \mathcal O(1) \theta \omega \ge 
 \mathcal O(1) \frac{\mu_\nu}{\omega}
 \ge \mu_\nu
$$
provided that $\omega$ is sufficiently small. This implies that we must use the accurate and not the simplified Riemann solver and it concludes the proof of Lemma~\ref{l:almeno}. 
\end{proof}
\subsubsection{The limit solution has infinitely many shocks}
\label{sss:theend}
We are eventually ready to conclude the proof of Proposition~\ref{p:perturbation}. First, we rely on the analysis in~\cite[Chapter 7]{Bre} and we conclude that when $\nu \to 0^+$ the wave front-tracking approximation $U^\nu(t, \cdot)$ converges strongly in $L^1_\loc (\R)$ to a limit function $U(t, \cdot)$ for every $t > 0$. Also, the function $U$ is the admissible solution of the Cauchy problem obtained by coupling~\eqref{e:cl2} with the initial datum $U(0, \cdot)= U_0$. 

We are left to prove that $U$ admits infinitely many shocks in 
$]-\rho, \rho[ \times ]0, 2 \widetilde T[.$ We rely on Lemma~\ref{l:almeno} and on fine properties of the wave front-tracking approximation established in~\cite{BressanLeFloch} (see also~\cite[\S10.3]{Bre} for an introductory exposition). 

More precisely, we firstly point out that the function $n_\theta$ defined as in~\eqref{e:almeno} satisfies 
\begin{equation}
\label{e:limit}
 \lim_{\theta \to 0^+} n_\theta = + \infty
\end{equation}
since $\omega <1$. Next, we refer to the definition of~\emph{maximal $\theta$-shock front} given in~\cite[p.219]{Bre}: loosely speaking, a maximal $\theta$-shock front is a polygonal line made by consecutive shocks of the same family where the strength of each shock is greater or equal than $\theta/2$ and there is at least one shock having strength greater or equal than $\theta$. 

Also, we consider the ``big'' 2-shocks $J^2_\ell$ and $J^2_r$ given by the statement of Lemma~\ref{l:collapse2}. We term $(t^\ast_\nu, x^\ast_\nu)$ their intersection point and we remark that by construction $x^\ast_\nu \in [-2q, 2q]$. 
Note that by looking at the proof of Lemma~\ref{l:almeno} we realize that, if 
$\theta > \mu_\nu /\omega^2$, then there are at least $n_\theta$ shocks with strength bigger or equal than $\theta$ and that cross the part of the plane between $J^2_\ell$ and $J^2_r$, namely they intersect the vertical line $x = x^\ast_\nu$ at some time $t < t^\ast_\nu$. 

We now argue inductively as follows. We fix a threshold $\theta_1>0$ such that 
$n_{\theta_1} \ge 1$, namely there is at least one shock  $j^1_\nu$ such that the strength of $j^1_\nu$ is at a some point greater or equal than $\theta_1$. In particular, $j^1_\nu$ is a maximal  
$\theta_1$-shock front. 

Next, we fix $\theta_2$ such that $n_{\theta_2} - n_{\theta_1/2} >1$: this implies that, for very $\nu$ sufficiently small, $U^\nu$ has at least a shock $j^2_\nu$ satisfying
\begin{equation}
\label{e:traidue}
    \theta_2  \leq \mathrm{strenght} \, j^2_\nu < \frac{\theta_1}{2}.
\end{equation}
By arguing as in~\cite[p. 220]{Bre} we infer that when $\nu \to 0^+$ the shock curves $j^1_\nu$ and $j^2_\nu$ converge uniformly (up to subsequences) to two shocks of the limit function $U$: we term them $j^1_\infty$ and $j^2_\infty$. Also, 
the value $x^\ast_\nu$ converges (up to subsequences) to some limit value $x^\ast$. Note that the limit shocks $j^1_\infty$ and $j^2_\infty$ both intersect the vertical line $x=x^\ast$ and, moreover, the strength of $j^1_\infty$ is greater or equal than $\theta_1$ and the strength of $j^2_\infty$ is comprised between $\theta_2$ and $\theta_1/2$. This implies that $j^1_\infty$ and $j^2_\infty$ are two distinct shock curves and, hence, the limit solution has at least $2$ shocks with strength greater or equal than $\theta_2$. 

Owing to~\eqref{e:limit}, we can iterate the above argument: for every natural number $k$, there is $\theta_k>0$ such that the limit  $U$ has at least $k$ distinct shocks with strength greater or equal 
than $\theta_k$. This implies that $U$ has infinitely many shocks and concludes the proof of Proposition~\ref{p:perturbation}.  

\subsection{Proof of Theorem~\ref{T:main}} 
\label{ss:proofteo}
This paragraph aims at establishing the proof of Theorem~\ref{T:main}. Before entering the 
technical details, we make some preliminary heuristic considerations. To establish Theorem~\ref{T:main} we need to construct a set $\mathfrak B \subseteq \mathcal S(\R)$ that satisfies condition \ref{item:2main}) and \ref{item:3main}) in the statement of the theorem. Proposition~\ref{p:perturbation} states that, if $\widetilde U$ is the same as in~\eqref{e:tildeu}, then  the admissible solution of the Cauchy problem with initial datum $\widetilde U$ develops infinitely many shocks and this behavior is stable with respect to $W^{1 \infty}$-perturbations.  Note, however, that both $\widetilde U$ and its $W^{1 \infty}$-perturbations have discontinuous first order derivatives and hence they do not belong to $\mathcal S(\R)$. To construct $\mathfrak B$, we mollify $\widetilde U$ to obtain a smooth function and we consider $W^{1 \infty}$-perturbations of the mollified function. 

We now provide the technical details: we first introduce the notation. We fix a convolution kernel $\phi$, namely a smooth function 
\begin{equation}
\label{e:phi}
        \phi: \R \to [0, + \infty[, \qquad 
        \int_\R \phi(x) dx =1, \qquad 
        \phi(x) = 0 \; \text{if $|x| \ge 1$}. 
\end{equation}
We fix $\varsigma>0$ and we define the mollified function $\widetilde U_\varsigma: \R \to \R^3$ by setting 
\begin{equation}
\label{e:mollificazione}  
\widetilde U_{\varsigma}(x): = \int_{-1}^1 
\widetilde U(x+\varsigma z)\phi(z)dz,
\end{equation}
where $\widetilde U$ is the same function as in~\eqref{e:tildeu}. Note that $\widetilde U_\varsigma \in \mathcal S(\R)$ since 
\begin{itemize}
\item $\widetilde U_\varsigma$ is compactly supported because so it is $\widetilde U$. 
\item $\widetilde U_\varsigma$ is smooth by the classical properties of convolution. 
\end{itemize}
Theorem~\ref{T:main} is a direct corollary of the 
following result.
\begin{proposition}
\label{p:perturbation2}
There is a sufficiently small constant $\varepsilon >0$ such that the following holds. Assume that $q=20$ and that $\delta, \zeta_w, \zeta_c, \eta, \omega, r$ and $\rho$ are as in the statement of Proposition~\ref{p:perturbation}. Assume furthermore that $\varsigma < \varepsilon^2 \eta \zeta_c$. Let $\widetilde U_\varsigma$ be the same function as in~\eqref{e:mollificazione} and set 
\begin{equation}
\label{E:secondapalla}
 \mathfrak B : = \mathcal S (\R) \cap \Big\{ U_0 \in W^{1 \infty}
 (\R): \; \| U_0 - \widetilde U_{\varsigma} \|_{W^{1 \infty}} < r   \Big\}. 
\end{equation} 
Then condition \ref{item:1main}), \ref{item:2main}) and \ref{item:3main}) in the statement of Theorem~\ref{T:main} are satisfied. 
\end{proposition}
The proof of Proposition~\ref{p:perturbation2} is divided into two parts: in \S~\ref{sss:pp21} we establish a technical lemma which loosely speaking says that Lemma~\ref{l:wp1tris} applies to the Riemann problems obtained from the piecewise constant approximation of $U_0 \in \mathfrak B$. In \S~\ref{sss:pp22} we conclude the proof of Proposition~\ref{p:perturbation2} and hence of Theorem~\ref{T:main}. 
\subsubsection{Analysis of the Riemann problems arising at initial time}
\label{sss:pp21}
This paragraph is devoted to the proof of Lemma~\ref{l:sampling}.

We assume that the hypotheses of Proposition~\ref{p:perturbation2} are satisfied and we recall that the set $\mathfrak B$ is defined as in~\eqref{E:secondapalla}. 
We also recall the mesh definition in \S~\ref{sss:mesh}: we fix  $\nu>0$, $h_\nu>0$ and  ${x^\nu_{ 0} < x^\nu_{ 1}< \dots <
 x^\nu_{m_\nu}}$ in $]-\rho, \rho[$ in such a way that~\eqref{e:meshsize} holds.   
\begin{lemma}
\label{l:sampling} Assume that the same hypotheses as in the statement of Proposition~\ref{p:perturbation2} hold true. Fix $x_i^\nu \in ]-\rho + \varepsilon, \rho- \varepsilon[$ and set 
\begin{equation}
\label{e:statirie}
   U^- : = \lim_{x \uparrow x^{\nu }_i } 
   U^{\nu}_0(x)=U_{0}(x_{i-1}^{\nu}),  \qquad
U^+ : = \lim_{x \downarrow x^{\nu }_i } U^{\nu}_0(x)=U_{0}(x_{i}^{\nu}) 
\end{equation}
and 
\begin{align*}
& V^{-}(z):=V\left(x_{i-1}^{\nu}+\varsigma z\right),
\end{align*}
where $V$ is the same function as in~\eqref{e:VI}.  
Also, let $\Psi$ be the same function as in~\eqref{e:psi} and consider the functions $\widetilde b,\widetilde \xi_{1},\widetilde \xi_{2},\widetilde \xi_{3}: [-1, 1] \to \R$ which are defined for every $z \in [-1, 1]$ by the equalities  
\begin{equation}
\label{e:itildi}
\begin{split}
&
- \widetilde b(z) \vec r_{1I} - \widetilde b(z) \vec r_{2I} + \widetilde  b(z)\vec r_{3I}=\Psi\left(x_{i}^{\nu}+\varsigma z\right)-\Psi\left(x_{i-1}^{\nu} +\varsigma z  \right)
\\ &
V(x_i^\nu + \varsigma z) = D_3 \left[\widetilde  \xi_{3}(z), D_2 \left[ - \widetilde \xi_{2}(z), D_1 \left[- \widetilde \xi_{1}(z), V^{-}(z)\right]\right]\right] . 
\\ 
\end{split}
\end{equation}
Finally, let $\phi$ be the same function as in~\eqref{e:phi} and 
$\mathfrak m$ the $\Ll^1$-absolutely continuous  measure  defined by setting 
\begin{equation}
\label{e:absolutec}
     \mathfrak m (E) : = \int_E \phi(x) dx \qquad 
     \text{for every $\Ll^1$-measurable set $E$.}  
\end{equation}
Then all the hypotheses of the Lemma~\ref{l:wp1tris} are satisfied provided that $\nu$ is small enough.
\end{lemma}
Observe that~\eqref{e:phi} and~\eqref{e:absolutec} yield that the measure $\mathfrak m$ is concentrated on $[-1, 1]$.
The function $\widetilde b$ is therefore defined for $\mathfrak m$-a.e. $z \in \R$, even if the function $\Psi$ is only defined in $]-\rho, \rho[$.
\begin{proof}
We proceed according to the following steps. \\
\firststep \step{we establish~\eqref{e:hyp1bis}} We first point out that, by combining~\eqref{e:distancefromUno},~\eqref{e:phi} and~\eqref{e:mollificazione}, we can conclude that the following estimate holds for every $x \in \R$: 
$$
   | \widetilde U_\varsigma (x) - U_I | = 
   \left| \widetilde U_\varsigma (x) - U_I \int_\R \phi(z) dz 
   \right| =  \left|\int_\R  [ \widetilde U(x + \varsigma z) - U_I ]
   \phi(z) dz  \right| \leq \mathcal O(1) \varepsilon^3. 
$$ 
This implies that, if $U_0 \in \mathfrak B$, then, since $r < \varepsilon^3$ owing to~\eqref{e:c:parametri3bis}, 
we have
$$
    |U_0 (x) - U_I| \leq | U_0 (x) - \widetilde U_\varsigma(x) |+ 
    | \widetilde U_\varsigma (x) - U_I |\leq r + \mathcal O(1) \varepsilon^3
    \leq \varepsilon,
$$
namely~\eqref{e:hyp1bis} holds true. \\ 
\step{we establish the first inequality in~\eqref{e:constraintetaxitris}} We fix $z \in [-1,1]$ and we first point out that, owing to the explicit expression~\eqref{e:psi} of $\Psi$, we have  
\begin{equation}
\label{e:bdfrombelowb}
    \widetilde b(z) \ge \min \{ \zeta_c, \zeta_w \} \cdot (x_i^\nu - x_{i-1}^\nu  )  =
    \zeta_c \cdot (x_i^\nu - x_{i-1}^\nu  )> 0
\end{equation}
 owing to~\eqref{e:allparameters}.   Also, for every $z \in [-1,1]$, we have that, owing to~\eqref{e:meshsize}, 
$$
\widetilde b (z)   \leq  (x_{i}-x_{i-1})\zeta_{w} \leq 
h_\nu \zeta_w \leq \varepsilon
 $$
provided that $\nu$ is sufficiently small because $h_\nu \to 0^+$ 
when $\nu \to 0^+$. This concludes the proof of the first inequality in~\eqref{e:constraintetaxitris}. \\
\step{we establish the second inequality in~\eqref{e:constraintetaxitris}} We combine the explicit expression~\eqref{e:VI} of $V$ with the inequality $|V'| \leq \mathcal O(1) \eta^{-1}$ (see~\S~\ref{ss:preli}) and we obtain that 
$$
   0 \leq \widetilde \xi_i(z) \leq  \mathcal O(1) (x_{i}^{\nu}-x_{i-1}^{\nu}) \eta^{-1} 
   \leq \mathcal O(1)  \eta^{-1} h_\nu,
   \quad \text{for $i=1, 2, 3$ and $z \in [-1, 1]$}
$$
and this, jointly with~\eqref{e:meshsize} and~\eqref{e:bdfrombelowb}, implies the second inequality in~\eqref{e:constraintetaxitris} provided that $\nu$ is sufficiently small. \\
\step{we establish~\eqref{e:hyp2aCompression1bis}} We first point out that, owing to~\eqref{e:tildeu} and~\eqref{e:itildi},
\begin{equation*}
\begin{split}
    \widetilde U \left(x_{i}^{\nu}+\varsigma z\right)- \widetilde U\left(x_{i-1}^{\nu}+\varsigma z\right) &=
\left[V\left(x_{i}^{\nu}+\varsigma z\right)-V\left(x_{i-1}^{\nu}+\varsigma z\right)\right] +
\left[\Psi\left(x_{i}^{\nu}+\varsigma z\right)-\Psi\left(x_{i-1}^{\nu}+\varsigma z\right)\right] \\
    & = -
    \widetilde b(z) \vec r_{1I} - \widetilde b(z) \vec r_{2I} + \widetilde  b(z)\vec r_{3I}
    \\
    & \quad +  
    D_3 \left[\widetilde  \xi_{3}(z), D_2 \left[ - \widetilde \xi_{2}(z), D_1 \left[- \widetilde \xi_{1}(z), V^{-}(z)\right]\right]\right] -V^{-}(z)  \\
\end{split}
\end{equation*} 
By integrating the above equality with respect to the measure $\mathfrak m$
and by recalling~\eqref{e:integrali} and~\eqref{e:mollificazione}  we conclude that
\begin{equation}
\label{e:mintegration}
\begin{split}
    \widetilde U_\varsigma \left(x_{i}^{\nu}\right)- & 
    \widetilde U_\varsigma\left(x_{i-1}^{\nu}\right) =   
     - b \vec r_{1I} - b \vec r_{2I} +  b\vec r_{3I}
    \\
    & \quad +  \int_\R 
    \left\{D_3 \left[\widetilde  \xi_{3}(z), D_2 \left[ - \widetilde \xi_{2}(z), D_1 \left[- \widetilde \xi_{1}(z), V^{-}(z)\right]\right]\right]
     -V^{-}(z)\right\}d \mathfrak m (z).  
    \\
\end{split}
\end{equation} 
 Next, we recall~\eqref{e:statirie}
 and we infer that
 \begin{equation}
 \label{e:decomposem}
     U^+ - U^-= U_0 (x_{i-1}^\nu) -
      U_0 (x_{i}^\nu)=
      \widetilde U_\varsigma (x_{i-1}^\nu)  - \widetilde U_\varsigma (x_{i}^\nu)
      - \widetilde U_\varsigma (x_{i-1}^\nu)  
+ U_0 (x_{i-1}^\nu) -  U_0 (x_{i}^\nu)       +\widetilde U_\varsigma (x_{i}^\nu). 
 \end{equation}
 Owing to~\eqref{E:secondapalla}, if $U_0 \in \mathfrak B$, then
 $$
     \Big| U_0 (x_{i-1}^\nu)  - \widetilde U_\varsigma (x_{i-1}^\nu)  -
     \big[ U_0 (x_{i}^\nu)
      -\widetilde U_\varsigma (x_{i}^\nu) \big]
      \Big|
      \leq  r   (x_{i}^\nu - x_{i}^{\nu-1}).
 $$
 By recalling~\eqref{e:decomposem}  with~\eqref{E:secondapalla} and~\eqref{e:mintegration} we conclude that, if $U_0 \in \mathfrak B$, then
\begin{equation*}
\begin{split} 
   \left| U^+ - U^- + b \vec r_{1I}   +b \vec r_{2I} -  b\vec r_{3I} 
   \right. & \left.
    -  \int_\R \left\{ 
    D_3 \left[\widetilde  \xi_{3}(z), D_2 \left[ - \widetilde \xi_{2}(z), D_1 \left[- \widetilde \xi_{1}(z), V^{-}(z)\right]\right]\right] -V_{\varsigma}^{-}(z)
    \right\} d \mathfrak m (z)  \right|
    \\
   &\leq  r   (x_{i}^\nu - x_{i}^{\nu-1}) \leq 
   \frac{1 }{4}\zeta_c
    (x_{i}^\nu - x_{i}^{\nu-1}) \leq \frac{ b}{4}. \\ 
\end{split}
\end{equation*}   
To achieve the last two equalities we have used~\eqref{e:allparameters} and~\eqref{e:bdfrombelowb}. This establishes~\eqref{e:hyp2aCompression1bis}.  \\
\step{we establish~\eqref{e:contraintxi}} We first recall that the first derivative of $V$ satisfies $|V'| \leq \mathcal O(1) \eta^{-1}$ and we infer that the same bound holds for $|\tilde U'|$. By recalling~\eqref{e:mollificazione},~\eqref{e:phi} and the inequality $\varsigma < \varepsilon^2 \eta \zeta_c$
we conclude that
\begin{equation*}
\begin{split}
    |\tilde U(x_{i-1}^\nu) - \tilde U_\varsigma(x_{i-1}^\nu) | & =
    \left| \int_{-1}^{1} 
     \big[ U(x_{i-1}^\nu) - U(x_{i-1}^\nu + \varsigma 
     z ) \big] \phi (z) dz \right| \\
     & \leq \int_{-1}^{1} \mathcal O(1) 
     \eta^{-1} \varsigma 
      \phi (z) dz  \leq \mathcal O(1)  \varepsilon^2 \zeta_c.  
\end{split}
\end{equation*}
By using again the inequality
$|V' | \leq \mathcal O(1) \eta^{-1}$, we infer that, for every $U_0 \in \mathfrak B$ and every $z$ such that $|z| \leq 1$
we have 
\begin{equation}
\label{e:vmenoumeno}
\begin{split}
|V^{-}(z)-U^{-}| \leq& 
|V\left(x_{i-1}^{\nu}+\varsigma z\right)-V\left(x_{i-1}^{\nu}\right)| + |V\left(x_{i-1}^{\nu}\right)- 
\widetilde U\left(x_{i-1}^{\nu}\right)|\\
& \quad + | \widetilde 
U \left(x_{i-1}^{\nu}\right)-
\widetilde U_{\varsigma}\left(x_{i-1}^{\nu}\right) |
 +|\widetilde U_{\varsigma}\left(x_{i-1}^{\nu}\right)-U_{0}\left(x_{i-1}^{\nu}\right) |
\\
 \leq &  \mathcal O(1) \eta^{-1} \varsigma |z|+
 \mathcal O(1)  \varepsilon  \zeta_w \eta  +
 \mathcal O(1) \varepsilon^2 \zeta_c + r
 \leq \mathcal O(1) \varepsilon^2\eta.\\
\end{split}
\end{equation}
To establish the last inequality we have used~\eqref{e:allparameters}, \eqref{e:ufferbacco}
and the inequality $\varsigma < \varepsilon^2 \eta \zeta_c$.  Next, we recall the explicit expression~\eqref{e:VI} of $V$ and~\eqref{e:itildi}
and we conclude that   
\begin{equation}
\label{e:xiunoduetre}
0\leq \xi_{1}(z)+\xi_{2}(z)+\xi_{3}(z)\leq  \mathcal O(1)\eta^{-1}
 \Ll^{1}\Big(\left[x_{i-1}^{\nu}+
 \varsigma z, x_{i}^{\nu}+\varsigma z\right] 
\cap \mathfrak R_{w}\Big).
\end{equation}
By using the explicit expression~\eqref{e:psi} of $\Psi$ we infer that, for every $z$ such that $|z|\leq 1$,  
\begin{equation*}
\begin{split}
b(z) &=\zeta_{c}\,\Ll^{1}\Big(
\left[x_{i-1}^{\nu}+\varsigma z, x_{i}^{\nu}+\varsigma z\right]\cap \mathfrak R_{c}\Big)+\zeta_{w}\,\Ll^{1}\Big(\left[x_{i-1}^{\nu}+\varsigma z, x_{i}^{\nu}+\varsigma z\right]\cap \mathfrak R_{w} \Big)
\\ &
\ge \zeta_{w}\,\Ll^{1}
\Big(\left[x_{i-1}^{\nu}+\varsigma z, x_{i}^{\nu}+\varsigma z\right]\cap \mathfrak R_{w} \Big).  \\
\end{split}
\end{equation*}
By combining the above formula with~\eqref{e:vmenoumeno} and~\eqref{e:xiunoduetre} we eventually arrive at~\eqref{e:contraintxi}. This concludes the proof of the lemma. 
\end{proof} 
\subsubsection{Proof of Proposition~\ref{p:perturbation2}: conclusion}
\label{sss:pp22}
In this paragraph we complete the proof of Proposition~\ref{p:perturbation2}, which has Theorem~\ref{T:main} as a direct corollary.

We consider the Baiti-Jenssen system~\eqref{e:cl2} and the set $\mathfrak B$ defined as in~\eqref{E:secondapalla}. We have to show that conditions \ref{item:1main}), \ref{item:2main}) and \ref{item:3main}) in the statement of Theorem~\ref{T:main} are satisfied. Condition \ref{item:1main}) is satisfied owing to the considerations in \S~\ref{sss:baj1}. Condition \ref{item:2main}) is satisfied: indeed 
\begin{itemize}
\item $\mathfrak B $ is nonempty since it contains $U_\varsigma $, because $U_\varsigma \in \mathcal S(\R)$. Also, 
\item $\mathfrak B$ is open in the topology of $\mathcal S(\R)$, which is stronger than the strong $W^{1, \infty}$ topology.
\end{itemize} 
 We are left to show that condition \ref{item:3main}) is also satisfied. The reason why condition condition \ref{item:3main}) is satisfied is because the proof Proposition~\ref{p:perturbation} continues to work if we replace the function $\widetilde U$ with the function $\widetilde U_\varsigma$, provided that $\varsigma < \varepsilon^2 \eta \zeta_c$. To see this, we first fix $U_0 \in \mathfrak B$ and we introduce its wave front-tracking approximation by arguing as in \S~\ref{sss:mesh}. Next, we point out that, owing to Lemma~\ref{l:sampling}, we can apply Lemma~\ref{l:wp1tris}.  This implies that the same conclusions as at the end of \S~\ref{sss:erreelle} and \S~\ref{sss:id:cw} hold true. This in turn implies that all the analysis in \S~\ref{Ss:qualitativeanalysis}-\S~\ref{ss:conclusion} applies. We can infer that Proposition~\ref{p:perturbation} holds true if we replace $\widetilde U$ with $\widetilde U_\varsigma$ and hence we conclude the proof of Proposition~\ref{p:perturbation2} and Theorem~\ref{T:main}.

\begin{thenomenclature} 
\label{notations}

 \nomgroup{A}

  \item [{$\Ll^N$:}]\begingroup the Lebesgue measure on $\R^N$\nomunit{}\nomeqref {1.5}
		\nompageref{5}
  \item [{$\mathcal O(1)$:}]\begingroup any function satisfying $ 0 < c \leq \mathcal O(1) \leq C$ for suitable constants $c,C>0$. The precise value of $C$ and $c$ can vary from line to line\nomunit{}\nomeqref {1.5}
		\nompageref{5}
  \item [{$\mcS(\R)$:}]\begingroup the Schwartz space of rapidly decreasing functions, endowed with the standard topology (see for instance~\cite[p.133]{ReedSimon} for the precise definition)\nomunit{}\nomeqref {1.5}
		\nompageref{5}
  \item [{$\norm{\cdot}_{W^{1 \infty}}$:}]\begingroup the standard norm in the Sobolev space $W^{1 \infty}$\nomunit{}\nomeqref {1.5}
		\nompageref{5}
  \item [{$\TV U$:}]\begingroup the total variation of the function $U: \R \to \R^N$, see~\cite[\S~3.2]{AFP} for the precise definition\nomunit{}\nomeqref {1.5}
		\nompageref{5}
  \item [{$\vec z_1 \cdot \vec z_2:$}]\begingroup the Euclidian scalar product between the vectors $\vec z_1, \; \vec z_2 \in \R^N$\nomunit{}\nomeqref {1.5}
		\nompageref{5}
  \item [{$D_{i}[\sigma,\bar U]$:}]\begingroup the $i$-wave fan curve through $\bar U$ \nomunit{See~\eqref{e:wavefan}}\nomeqref {1.5}
		\nompageref{5}
  \item [{$F'$:}]\begingroup the first derivative of the differentiable  function $F: \R \to \R^N$\nomunit{}\nomeqref {1.5}
		\nompageref{5}
  \item [{$F(x^\pm)$:}]\begingroup the left and right limit of the function $F$ at $x$ (whenever they exist)\nomunit{}\nomeqref {1.5}
		\nompageref{5}
  \item [{$R_i [s, \bar U]$:}]\begingroup the $i$-rarefaction curve through $\bar U$ \nomunit{See~\eqref{e:intcur}}\nomeqref {1.5}
		\nompageref{5}
  \item [{$S_i[s, \bar U]$:}]\begingroup the $i$-shock curve through $\bar U$ \nomunit{See~\S~\ref{ss:wft}}\nomeqref {1.5}
		\nompageref{5}
  \item [{$W^{1,\infty}$:}]\begingroup the space of Lipschitz continuous functions\nomunit{}\nomeqref {1.5}
		\nompageref{5}
  \item [{a.e. $(t, x)$:}]\begingroup for $\Ll^2$-almost every $(t,x)$\nomunit{}\nomeqref {1.5}
		\nompageref{5}
  \item [{a.e. $x$:}]\begingroup for $\Ll^1$-almost every $x$\nomunit{}\nomeqref {1.5}
		\nompageref{5}

 \nomgroup{C}

  \item [{$\delta$:}]\begingroup a strictly positive parameter\nomeqref {1.5}
		\nompageref{5}
  \item [{$\eta$:}]\begingroup the perturbation parameter in the flux function $F_\eta$ \nomunit{See~\eqref{e:pertSyst},~\eqref{e:c:parametri2}}\nomeqref {1.5}
		\nompageref{5}
  \item [{$\lambda_i(U)$:}]\begingroup the $i$-th eigenvalue of the Jacobian matrix $JF_{\eta}$ \nomunit{See~\eqref{e:eigenvaluese}}\nomeqref {1.5}
		\nompageref{5}
  \item [{$\mathfrak q$, $\mathfrak p$:}]\begingroup strictly positive parameters\nomeqref {1.5}
		\nompageref{5}
  \item [{$\mathfrak R_\ell, \dots, \mathfrak R_r$:}]\begingroup open subsets of $\R$ \nomunit{See~\eqref{e:regions}}\nomeqref {1.5}
		\nompageref{5}
  \item [{$\mu_\nu$:}]\begingroup the threshold for using the accurate Riemann solver \nomunit{See~\S~\ref{ss:wft}}\nomeqref {1.5}
		\nompageref{5}
  \item [{$\nu$, $h_{\nu}$:}]\begingroup parameter and mesh size for the wave front-tracking approximation \nomunit{See~\S~\ref{sss:mesh}}\nomeqref {1.5}
		\nompageref{5}
  \item [{$\omega$:}]\begingroup a strictly positive parameter \nomunit{See~\eqref{e:uduetre},~\eqref{e:c:parametri2}}\nomeqref {1.5}
		\nompageref{5}
  \item [{$\Psi$:}]\begingroup the function $\Psi: \R \to \R^3$ \nomunit{See~\S\S~\ref{sss:psi},~\ref{sss:tildeu}}\nomeqref {1.5}
		\nompageref{5}
  \item [{$\rho$:}]\begingroup the strictly positive parameter in~\eqref{e:c:parametrirho} \nomunit{See~\S~\ref{sss:tildeu},~Remark~\ref{rem:qrho}}\nomeqref {1.5}
		\nompageref{5}
  \item [{$\underline U'$, $\underline U''$, $\underline U^\ast$ and $\underline U^{\ast \ast}$:}]\begingroup fixed states in $\R^3$ \nomunit{See~\eqref{e:statedef}}\nomeqref {1.5}
		\nompageref{5}
  \item [{$\V_i$:}]\begingroup the strength of a shock $i$ \nomunit{See Page~\pageref{label:strength}}\nomeqref {1.5}
		\nompageref{5}
  \item [{$\varepsilon$:}]\begingroup a strictly positive, sufficiently small parameter\nomunit{See~Proposition~\eqref{p:perturbation}}
		\nompageref{5}
  \item [{$\vec r_i(U)$:}]\begingroup the $i$-th right eigenvector of the Jacobian matrix $JF_{\eta}$\nomunit{See~\S~\ref{sss:baj2}}\nomeqref {1.5}
		\nompageref{5}
  \item [{$\widetilde T$:}]\begingroup the strictly positive interaction time \nomunit{See~\eqref{e:T},~\eqref{e:c:parametriT}}\nomeqref {1.5}
		\nompageref{5}
  \item [{$\widetilde U$, $\widetilde U_{\varsigma}$:}]\begingroup the function $\widetilde U: \R \to \R^3$ and its mollification \nomunit{See~\S~\ref{sss:tildeu},~\eqref{e:mollificazione},~\eqref{E:secondapalla}}\nomeqref {1.5}
		\nompageref{5}
  \item [{$\zeta_c$:}]\begingroup a strictly positive parameter \nomunit{See~\eqref{e:psi},~\eqref{e:c:parametri3bis},~\eqref{e:c:parametri5}}\nomeqref {1.5}
		\nompageref{5}
  \item [{$\zeta_w$:}]\begingroup a strictly positive parameter \nomunit{See~\eqref{e:psi},~\eqref{e:c:parametribis},~\eqref{e:c:parametri5}}\nomeqref {1.5}
		\nompageref{5}
  \item [{$q$:}]\begingroup a strictly positive parameter that we fix equal to $20$ \nomunit{See Lemma~\ref{l:infty},~Remark~\ref{rem:qrho}}\nomeqref {1.5}
		\nompageref{5}
  \item [{$r$:}]\begingroup a strictly positive parameter \nomunit{See~\eqref{e:c:parametri3bis},~\eqref{e:parameters},~\eqref{e:palla}}\nomeqref {1.5}
		\nompageref{5}
  \item [{$u,w,v$:}]\begingroup the first, second and third component of the vector-valued function $U$\nomunit{See \S~\ref{sss:baj1}}\nomeqref {1.5}
		\nompageref{5}
  \item [{$U^\nu$:}]\begingroup wave front-tracking approximation of the admissible solution $U$ \nomunit{See~\S~\ref{ss:wft}}\nomeqref {1.5}
		\nompageref{5}
  \item [{$U^\nu_0$:}]\begingroup wave front-tracking approximation of the initial datum $U_{0}$ \nomunit{See~\S~\ref{ss:wft}}\nomeqref {1.5}
		\nompageref{5}
  \item [{$V$:}]\begingroup the Lipschitz continuous function $V: \R \to \R^3$ \nomunit{See~\eqref{e:VI}}\nomeqref {1.5}
		\nompageref{5}
  \item [{$W$:}]\begingroup the piecewise constant function $W: \R \to \R^3$ \nomunit{See~\eqref{e:W}}\nomeqref {1.5}
		\nompageref{5}
  \item [{$x^\nu_{i}$:}]\begingroup mesh points for the wave front-tracking approximation \nomunit{See~\eqref{e:meshsize}}\nomeqref {1.5}
		\nompageref{5}

\end{thenomenclature}

\vskip\baselineskip
\paragraph{\bf Ackowledgments}
This work was originated by a question posed by Tai Ping Liu, whom the authors wish to thank. Part of this work was done when the first author was affiliated to the University of Oxford, which she thanks for the stimulating scientific environment. The second author thanks the University of Oxford for supporting her visits, during which part of this work was done. Both authors are members of the Gruppo Nazionale per l'Analisi Matematica, la Probabilit\`a e le loro Applicazioni (GNAMPA) of the Istituto Nazionale di Alta Matematica (INdAM) and are supported by the PRIN national project ``Nonlinear Hyperbolic Partial Differential Equations, Dispersive and Transport Equations: theoretical and applicative
aspects''.

\end{document}